\documentclass[12pt]{elsarticle}

\usepackage{mystyleElsArticle}
\usepackage{mathrsfs}
\usepackage{booktabs}
\usepackage{color}
\newcommand{\STATE}{\State}
\newcommand{\FOR}{\For}
\newcommand{\ENDFOR}{\EndFor}
\newcommand{\IF}{\If}
\newcommand{\ENDIF}{\EndIf}

\makeatletter
\def\ps@pprintTitle{%
  \let\@oddhead\@empty
  \let\@evenhead\@empty
  \let\@oddfoot\@empty
  \let\@evenfoot\@oddfoot
}
\makeatother
\usepackage[margin=0.90in]{geometry}    % Margin: 1-inch all round

\usepackage{lineno}
\begin{document}
\begin{frontmatter}
		%\title{An efficient model reduction method for solving viscous G-equations in incompressible cellular flows}	
		%\title{Error estimates for an efficient model reduction method for solving viscous G-equations in incompressible cellular flows}
		\title{\textcolor{black}{Error estimates for a POD method for solving viscous G-equations in incompressible cellular flows}}								
		\author[ucb]{Haotian Gu}
		\ead{htgu@math.berkeley.edu}
		\author[uci]{Jack Xin}
		\ead{jxin@math.uci.edu}
		\author[hku]{Zhiwen Zhang\corref{cor1}}
		\ead{zhangzw@hku.hk}
		
		\address[ucb]{Department of Mathematics, University of California at Berkeley, Berkeley, CA 94720, USA.}
		\address[uci]{Department of Mathematics, University of California at Irvine, Irvine, CA 92697, USA.}
		\address[hku]{Department of Mathematics, The University of Hong Kong, Pokfulam Road, Hong Kong SAR, China.}
		%\address[pnnl]{Computational Mathematics, Pacific Northwest National Laboratory, Richland, WA 99352, USA.}
		\cortext[cor1]{Corresponding author}
\begin{abstract}
\noindent
The G-equation is a well-known model for studying front propagation in turbulent combustion.  In this paper, we develop an efficient model reduction method for computing \textcolor{black}{regular solutions} of viscous G-equations in incompressible steady and time-periodic cellular flows. Our method is based on the Galerkin proper orthogonal decomposition (POD) method. To facilitate the algorithm design and convergence analysis, we decompose the solution of the viscous G-equation into a mean-free part and a mean part, where their evolution equations can be derived accordingly. We construct the POD basis from the solution snapshots of the mean-free part. With the POD basis, we can efficiently solve the evolution equation for the mean-free part of the solution to the viscous G-equation. After we get the mean-free part of the solution, the mean of the solution can be recovered. We also provide rigorous convergence analysis for our method. Numerical results for \textcolor{black}{viscous G-equations and curvature G-equations} are presented to demonstrate the accuracy and efficiency of the proposed method. In addition, we study the turbulent flame speeds of the viscous G-equations in incompressible cellular flows.

% based on the POD method and fully resolved computations. \\ 
% 65M12 Stability and convergence of numerical methods
% 70H20  	Hamilton-Jacobi equations
% 76F25  	Turbulent transport, mixing
% 80A25  	Combustion
% 78M34  	Model reduction
\end{abstract}	
\end{frontmatter}

\section{Introduction}\label{sec:introduction}
\noindent  Front propagation in turbulent combustion is a nonlinear and complicated dynamical process. 
The G-equation has been a very popular field model in combustion and physics literature for studying premixed turbulent flame propagation \cite{markstein2014nonsteady,matkowsky1979asymptotic,embid1995comparison,peters2000turbulent,fedkiw2002level,xin2009introduction}. The G-equation model is a sound phenomenological approach to study turbulent combustion, which uses the level-set formulation to study the flame front motion laws with the front width ignored. The simplest motion law is that the normal velocity of the front is equal to a constant $S_l$ (the laminar speed) plus the projection of fluid velocity $V(\vec{x},t)$  along the normal. This gives the inviscid G-equation

\begin{equation}\label{eq:invis-Geq}
G_t+\vec{V}\cdot\nabla G+S_l|\nabla G|=0,
\end{equation}
where the set $\{(\vec{x},t):G(\vec{x},t)=0\}$ corresponds to the location of the flame front at time $t$.
\textcolor{black}{The derivation of the inviscid G-equation will be elaborated in detail in Section 2.1. Developing numerical methods for \eqref{eq:invis-Geq} has been an active research topic for decades; see
e.g. \cite{osher1988fronts,jin1998numerical,cockburn2000DG,kurganov2000new,falcone2013semi} and references therein. } 
	
As the fluid turbulence is known to cause stretching and corrugation of flames, additional modeling terms need to be incorporated into the basic G-equation. If the curvature term is added into the basic equation to model the curvature effects and the curvature term is further linearized, we will arrive at the viscous G-equation
\begin{equation}\label{eq:vis-G}
G_t+\vec{V}\cdot\nabla G+S_l|\nabla G|=dS_l\Delta G.
\end{equation}
In order to compute numerical solutions of Eq.\eqref{eq:vis-G}, the authors of \cite{liu2013numerical} first approximated the G-equations by a monotone discrete system, then applied high resolution numerical methods such as WENO (weighted essentially non-oscillatory) finite difference methods \cite{jiang2000weighted} with a combination of explicit and semi-implicit time stepping strategies, depending on the size and property of dissipation in the equations. However, these existing numerical methods become expensive when one needs to discretize the Eq.\eqref{eq:vis-G} with a fine mesh or one needs to solve Eq.\eqref{eq:vis-G} many times with different parameters. This motivates us to exploit the low-dimensional structures of \textcolor{black}{the regular solutions} of viscous G-equation \eqref{eq:vis-G} and develop model reduction methods to solve them efficiently. 
 
One of the successful model reduction ideas in the study of turbulent flows is the proper orthogonal decomposition (POD) method \cite{sirovich1987,berkooz1993POD}. The POD method uses the data from an accurate numerical simulation and extracts the most energetic modes in the system by using the singular value decomposition. This approach generates low-dimensional structures that play an important role in the dynamics of the flow. The Galerkin POD method has been used to solve many types of partial differential equations, including linear parabolic equations, Burgers equations, Navier‐Stokes equations and \textcolor{black}{Hamilton–Jacobi–Bellman (HJB) equations}; see  \cite{kunisch2001galerkin,kunisch2004hjb,borggaard2011artificial,volkwein2013proper,alla2013time,kalise2014reduced,benner2015survey,kalise2018polynomial,alla2020hjb} and references therein for details. The interested reader is referred to \cite{quarteroni2015reduced,benner2015survey,hesthaven2016certified} for a comprehensive introduction of the model reduction methods. 

In this paper, we study the POD method to solve the viscous G-equation \eqref{eq:vis-G}, which is a 
Hamilton-Jacobi type equation. To deal with the periodic boundary condition of the problem and \textcolor{black}{facilitate the convergence analysis of the POD method}, we decompose the solution of the viscous G-equation into a mean-free part and a mean part, where their evolution equations can be derived accordingly; see Eq.\eqref{eq:vis-hatbar}. We construct the POD basis from the snapshots of the mean-free part of the solutions, which can be used to solve the evolution equation for the mean-free part. Then, the mean part of the solution can be computed from the mean-free part; see Eq.\eqref{evolutionEq_ubar}. \textcolor{black}{Notice that the bilinear form of the evolution equation for the mean-free part satisfies the coercive condition, which follows from the \textit{Poincar\'e-Wirtinger inequality}}. We provide rigorous convergence analysis and show that the accuracy of our method is guaranteed. \textcolor{black}{We remark that the idea of decomposing the data into a mean-free part and a mean part plays an essential role in the convergence analysis of the POD method for solving the viscous G-equation}. Finally, we conduct numerical experiments to demonstrate the accuracy and efficiency of the proposed method. 

We find that the POD basis can be used to compute long-time solution of the viscous G-equation or the viscous G-equations with different parameters. Moreover, we study the turbulent flame speeds of viscous G-equations in incompressible steady and time periodic cellular flows, which help our understanding of turbulent combustion. \textcolor{black}{To further reduce the numerical error, we develop an adaptive strategy to dynamically enrich the POD basis and demonstrate its effectiveness through a viscous G-equation with time-periodic fluid velocity}. We remark that our POD method can easily be extended to solve other types of G-equations \cite{liu2013numerical}. \textcolor{black}{Though we are unable to provide rigorous convergence analysis, we find that the POD method is still efficient in solving curvature G-equation when its solution is smooth}. \textcolor{black}{To the best of our knowledge, our result is the first one in the literature that develops POD based model reduction method to solve G-equations and compute their front speeds}. 
  
%In our research, we make use of this reduced-order model to design efficient numerical algorithms to solve G-equations. In this paper, our main focus is to solve Eq.\eqref{eq:vis-G} using the POD method and to analyze its performance, but the POD algorithm we propose can be easily extended to solve other types of G-equations.
 
The rest of this paper will be organized as follows. In Section 2, we shall give a brief derivation of G-equation models. In Section 3, we show the detailed derivations of the model reduction method for G-equations. In Section 4, we provide the convergence analysis of the proposed method. Our proof is based on the backward Euler-Galerkin-POD approximation scheme. \textcolor{black}{The proofs for other discretization schemes, such as Crank-Nicolson scheme,  can be obtained in a similar manner}. In Section 5, we shall perform numerical experiments to test the performance and accuracy of the proposed method. We find that POD method can
provide considerable savings over finite difference methods in solving the G-equation while 
its numerical error is relatively small. Finally, the concluding remarks will be given in Section 6.
In the appendix Section, we provide the derivation of a finite difference scheme to solve G-equations proposed in \cite{liu2013numerical} and the procedure of constructing POD basis. 
%G-equations are well-known front propagation models in turbulent combustion that describe the front motion law in the form of local normal velocity equal to a constant (laminar speed) plus the normal projection of fluid velocity. 
\section{Turbulent combustion and G-equations}\label{sec:Turbulent}
\subsection{Derivation of the G-equations}
\noindent
In this section, we briefly introduce the derivation of the G-equation in turbulent combustion. In a thin reaction zone regime and the corrugated flamelet regime of premixed turbulent combustion 
(Chapter 2 of \cite{peters2000turbulent}), the flame front is modeled by a level set function: $\{(\vec{x},t):G(\vec{x},t)=0,\vec{x}\in R^n\}$, which is the interface between the burned area, denoted by $\{G<0\}$ and the unburned area, denoted by $\{G>0\}$, respectively. Therefore, one can study the propagation of the flame front by solving the dynamic equation for the level set function, namely the G-equation.  

\begin{figure}[h]
	\centering
	\includegraphics[width=0.60\linewidth]{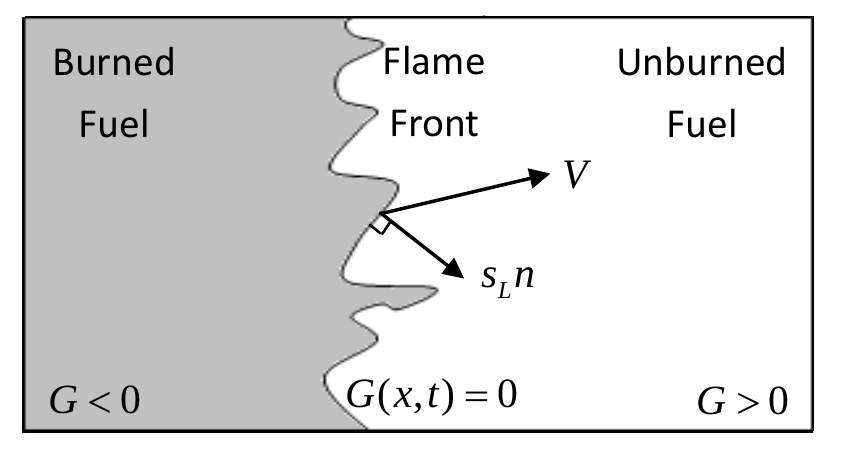}
	\caption{Illustration of local interface velocities in the G-equation and a flame front in a 2D space.}
	\label{fig:Illustration}
\end{figure}
The simplest motion law for the particles on the interface is that the normal velocity of the interface is the
sum of a constant $S_l$ and the projection of fluid velocity $\vec{V}(\vec{x},t)$ along the normal direction (see Fig.\ref{fig:Illustration}). Hence, the trajectory of a particle $\vec{x}(t)$ on the interface satisfies, 

\begin{equation}\label{eq:motion-G}
\frac{\mathrm{d} \vec{x}}{\mathrm{d}t}=\vec{V}(\vec{x},t)+S_l\vec{n},
\end{equation}
where $S_l$ is the laminar flame speed and $\vec{n}$ is the normal vector. In terms of the level set function, the motion law \eqref{eq:motion-G} gives the inviscid G-equation, 
\begin{equation}\label{eq:invis-G}
G_t+\vec{V}\cdot\nabla G+S_l|\nabla G|=0,
\end{equation}
where $\nabla$ denotes the gradient operator. Thus, the normal vector in \eqref{eq:motion-G} can be computed 
using $\vec{n}=\nabla G/|\nabla G|$. Notice that the set $\{G(\vec{x},t)=0\}$ corresponds to the location of the flame front at time $t$.

To take into account the effect of flame stretching and corrugation, additional terms are added into the inviscid G-equation \eqref{eq:invis-G} and one can obtain extended G-equation models involving curvature effects. The curvature G-equation is 
\begin{equation}\label{eq:curvature-G}
G_t+\vec{V}\cdot\nabla G + S_l |\nabla G|=dS_l|\nabla G|\big(\nabla\cdot\frac{\nabla G}{|\nabla G|}\big),
\end{equation}
where $d$ is the Markstein number. \textcolor{black}{The mean curvature term leads to a nonlinear degenerate second-order PDE that has been widely studied in the literature from both the analytical and numerical point of views. The curvature dependent motion is well-known; see \cite{evans1991motion,osher1988fronts,fedkiw2002level,brakke2015motion} and references therein. However, we are not aware of a POD approximation of curvature motion in the general setting}. \textcolor{black}{We remark that the classical boundary conditions to deal with first-order level set equations and second-order mean curvature motion equations are Neumann type boundary conditions. We are interested in computing the turbulent flame speeds of the G-equations in incompressible flows. Thus, we choose periodic boundary conditions in the G-equation \eqref{eq:curvature-G}, which will be further clarified in the numerical experiment; see Section \ref{sec:Test5}.}

If the curvature term is linearized, we obtain the viscous G-equation as follows, which is the research focus in this paper. 
\begin{equation}\label{eq:viscous-G}
G_t+\vec{V}\cdot\nabla G+S_l|\nabla G|=dS_l\Delta G.
\end{equation} 
%\textcolor{black}{The use of a linearized version of the curvature G-equation, i.e. the viscous G-equation \eqref{eq:viscous-G} is motivated mainly by the use of a Galerkin approximation. This gives a more regular solution due to the diffusion term. Moreover, it facilitates the convergence analysis of the Galerkin approximation method.}
\textcolor{black}{The linearized version of the curvature G-equation, i.e. the viscous G-equation \eqref{eq:viscous-G} allows us to carry out a convergent Galerkin approximation and POD error analysis, which is unavailable for \eqref{eq:curvature-G}.  Our POD method remains efficient for smooth solutions of the curvature G-equation \eqref{eq:curvature-G}, as shown numerically later.}
 
% I first comment the introduction of the strain G-equation, to simply the paper.   
%To take into account the effect of flame stretching, the surface stretch rate may be added as a first-order correction term on laminar flame speed 
%\begin{equation}\label{eq:S_l}
%\tilde{S}_L=S_l-d(\kappa+\mathcal{S})
%\end{equation}
%where $d$ is called the Markstein number, $\kappa=\nabla\cdot\vec{n}$ is the mean curvature, and $\mathcal{S}=-\vec{n}\cdot\nabla V\cdot\vec{n}$ is called the strain rate. Replacing $S_l$ by $\tilde{S}_L$ in the motion law \eqref{eq:motion-G}, we obtain the strain G-equation 
%\begin{equation}\label{eq:strain-G}
%G_t+\vec{V}\cdot\nabla G+(S_l+d\frac{\nabla G\cdot\nabla\vec{V}\cdot\nabla G}{|\nabla G|^2})|\nabla G|=d|\nabla G|(\nabla\cdot\frac{\nabla G}{|\nabla G|}).
%\end{equation}
%If the curvature term is further linearized, we arrive at the viscous G-equation \eqref{eq:vis-G}.

\subsection{A periodic initial value problem}\label{sec:periodicIVP}
\noindent
Now given a unit vector $\vec{P}\in\mathbb{R}^n$, where $n$ is the dimension of the physical space, we shall consider the viscous G-equation \eqref{eq:viscous-G} with planar initial condition
\begin{equation}\label{eq:vis-GonR}
\begin{cases}
G_t+\vec{V}\cdot\nabla G+S_l|\nabla G|=dS_l\Delta G &\text{in}\ \mathbb{R}^n\times(0,\infty),\\
G(\vec{x},0)=\vec{P}\cdot\vec{x} &\text{on}\ \mathbb{R}^n\times\{t=0\}.
\end{cases}
\end{equation}
Here, we assume the flame front propagates in direction $\vec{P}$ with the initial front being $\{\vec{P}\cdot \vec{x} = 0\}$ and $\vec{x}=(x_1,x_2,...,x_n)$. In addition, we assume $\vec{V}(\vec{x},t)$ is spatially periodic with $C^1$ in $\vec{x}$ and $C^0$ in $t$. Moreover, $\vec{V}(\vec{x},t)$ is divergence-free, i.e., $\nabla_{\vec{x}}\cdot \vec{V}(\vec{x},t)=0$, $\forall t$ and uniformly bounded, i.e., $||\vec{V}(\vec{x},t)||_{\infty}\leq A$.

If we write $G(\vec{x},t)=\vec{P}\cdot\vec{x}+u(\vec{x},t)$, then $u(\vec{x},t)$ is also spatially periodic and 
satisfies the following periodic initial value problem 
\begin{equation}\label{eq:vis-GonT}
\begin{cases}
u_t+\vec{V}\cdot(\vec{P}+\nabla u)+S_l|\vec{P}+\nabla u|=dS_l\Delta u &\text{in}\ \mathbb{T}^n\times(0,\infty),\\
u(\vec{x},0)=0 &\text{on}\ \mathbb{T}^n\times\{t=0\},
\end{cases}
\end{equation}
where $\mathbb{T}^n=[0,1]^n$. 
Hence we can reduce the problem \eqref{eq:vis-GonR} defined on the whole domain $\mathbb{R}^n$ to the problem \eqref{eq:vis-GonT} defined on $[0,1]^n$ by imposing the affine periodic condition 
\begin{equation}\label{eq:vis-Gaffine}
\begin{cases}
G_t+\vec{V}\cdot\nabla G+S_l|\nabla G|=dS_l\Delta G &\text{in}\ \mathbb{T}^n\times(0,\infty),\\
G(\vec{x},0)=\vec{P}\cdot\vec{x} &\text{on}\ \mathbb{T}^n \times\{t=0\},
\end{cases}
\end{equation}
with the assumption $G(\vec{x}+\vec{z},t)=G(\vec{x},t)+\vec{P}\cdot\vec{z}$.
%We aim to develop efficient numerical methods to solve \eqref{eq:vis-GonR0} with different $A$ and $d$. 

To illustrate the main idea of our method, we choose $\vec{P}=\vec{e}_1=(1,0,...,0)$ and have $G(\vec{x},t)=x_1+u(\vec{x},t)$. 
Let $\mathcal{A}(t)$ denote the volume of the burned area that has invaded during time interval $(0,t)$.
Then, the turbulent flame speed $S_{T}$ is defined as the linear growth rate of $\mathcal{A}(t)$. 
Notice that $G(\vec{x},0)=x_1$ and $G(\vec{x}+\vec{e}_1,t)=G(\vec{x},t)+1$. Then, the %$\mathcal{A}(t)$ and 
$S_{T}$ can be evaluated by $G(\vec{x},t)$ or $u(\vec{x},t)$ in $[0,1]^n$,
\begin{equation}\label{eq:TurbulentFlameSpeed}
S_{T} %= \lim_{t\rightarrow \infty} \frac{\mathcal{A}(t)}{t}
= \lim_{t\rightarrow \infty}\frac{-1}{t}\int_{[0,1]^n}G(\vec{x},t)d\vec{x}
= \lim_{t\rightarrow \infty}\frac{-1}{t}\int_{[0,1]^n}(x_1+u(\vec{x},t))d\vec{x}. 
\end{equation}  

\textcolor{black}{%In mathematical models and numerical simulations, periodic boundary conditions are used for approximating a large (infinite) system by using a unit cell. When an object passes through one side of the unit cell, it re-appears on the opposite side with the same velocity. 
In this work, we are interested in computing the turbulent flame speeds of the viscous G-equations in incompressible cellular flows, which is an effective (homogenized) quantity and involves a long-time computation. To this end, we impose periodic boundary conditions for \eqref{eq:vis-Gaffine}.}    
\textcolor{black}{We will study the POD method for solving other nonlinear PDEs with other types of boundary conditions in our future research}.

There are many numerical methods to solve the viscous G-equation \eqref{eq:vis-Gaffine} or Eq.\eqref{eq:vis-GonT} on a bounded physical domain $\mathbb{T}^n$. For instance, we can apply a
higher order Hamilton-Jacobi weighted essentially non-oscillatory (HJ-WENO) scheme and total variation diminishing Runge-Kutta (TVD-RK) scheme in the spatial and time discretization, respectively. See \cite{shu1988efficient,jiang2000weighted,fedkiw2002level,liu2013numerical} for details of the schemes. To make our paper self-contained, we show the numerical schemes proposed by \cite{liu2013numerical}; see the Appendix Section for more details.
 
\section{Model reduction method for G-equations}\label{sec:modelreductionviscGeq} 
\noindent
In practice, we are interested in studying the dependence of the turbulent flame speed $S_{T}$ on 
different parameters in the viscous G-equation, such as the Markstein number $d$. 
As such, we have to solve the viscous G-equation many times, which is an expensive task. 
Therefore, we need to design numerical methods that allow us to efficiently and accurately solve the viscous G-equation. We shall develop an efficient model reduction method to achieve this goal. \textcolor{black}{We point out that the model reduction methods also bring computational savings when one needs to solve high-dimensional problems, e.g. 3D viscous G-equations.}

\subsection{A decomposition strategy}\label{sec:decompotionstra}
\noindent
According to the definition in \eqref{eq:TurbulentFlameSpeed}, we can either solve 
Eq.\eqref{eq:vis-GonT} to obtain $u(x,t)$ or solve Eq.\eqref{eq:vis-Gaffine} to obtain $G(x,t)$, 
in order to compute the turbulent flame speed $S_{T}$. We shall consider to solve Eq.\eqref{eq:vis-GonT} since it is easier to deal with the boundary condition. 
Let us decompose the solution $u$ of Eq.\eqref{eq:vis-GonT} into $\hat{u}+\bar{u}$, where $\hat{u}$ is the mean-free part and $\bar{u}$ is the mean of $u$. This decomposition means that  

\begin{equation}\label{solution_decomposition}
\int_{\mathbb{T}^n}\hat{u}(\vec{x},t)d\vec{x}=0  ~~\forall t \quad \text{and} \quad\bar{u}(t)=\int_{\mathbb{T}^n}u(\vec{x},t)d\vec{x}.
\end{equation}
Then, substituting $u=\hat{u}+\bar{u}$ into \eqref{eq:vis-GonT}, we obtain

\begin{align}
\hat{u}_t+\bar{u}_t+\vec{V}\cdot\big(\vec{P}+\nabla \hat{u}\big)+S_l\big|\vec{P}+\nabla\hat{u}\big|-dS_l\Delta\hat{u}=0.
\end{align}
Integrating the above equation over the domain $\mathbb{T}^n$ gives 

\begin{equation}\label{intergral_u}
\int_{\mathbb{T}^n}\big[\hat{u}_t+\bar{u}_t+\vec{V}\cdot(\vec{P}+\nabla\hat{u})-dS_l\Delta\hat{u}\big]d\vec{x}
=-\int_{\mathbb{T}^n}S_l\big|\vec{P}+\nabla\hat{u}\big|d\vec{x}.
\end{equation}
The definitions of $\bar{u}$ and $\hat{u}$ in \eqref{solution_decomposition} imply that  $\int_{\mathbb{T}^n}\hat{u}_t=0$ and $\int_{\mathbb{T}^n}\bar{u}_t=\bar{u}_t$. Furthermore, the periodic conditions of $\vec{V}$ and $\hat{u}$ imply $\int_{\mathbb{T}^n}\vec{V}\cdot\vec{P}=0$, $\int_{\mathbb{T}^n}\Delta\hat{u}=0$, and $\int_{\mathbb{T}^n}\vec{V}\cdot\nabla\hat{u}=0$, where we have used the facts that $\vec{V}$ is divergence-free and $\hat{u}$ is the mean-free. Combining these results, we can simplify the Eq.\eqref{intergral_u} as,  
\begin{equation}\label{evolutionEq_ubar}
\bar{u}_t=-\int_{\mathbb{T}^n}S_l\big|\vec{P}+\nabla\hat{u}\big|d\vec{x}.
\end{equation}
Finally, we find that the Eq.\eqref{eq:vis-GonT} is equivalent to
\begin{equation}\label{eq:vis-hatbar}
\begin{cases}
\hat{u}_t+\vec{V}\cdot\nabla\hat{u}-dS_l\Delta\hat{u}+S_l\big|\vec{P}+\nabla\hat{u}\big|-\int_{\mathbb{T}^n}S_l\big|\vec{P}+\nabla\hat{u}\big|d\vec{x}+\vec{V}\cdot\vec{P}=0 &\text{in}\ \mathbb{T}^n\times(0,\infty)\\
\bar{u}_t=-\int_{\mathbb{T}^n}S_l\big|\vec{P}+\nabla\hat{u}(\vec{x},t)\big|d\vec{x} &\text{on}\ t\in(0,\infty)\\
\hat{u}(\vec{x},0)=0 &\text{on}\ \mathbb{T}^n\times\{t=0\}\\
\bar{u}(0)=0
\end{cases}
\end{equation}
\textcolor{black}{The strategy of decomposing the solution or data into a mean-free part and a mean part is often used in scientific computing and data science. However, it has not been studied in the G-equation setting. We shall demonstrate later that this decomposition plays a crucial role in the algorithm design and facilitates the convergence analysis of our proposed method.}

\subsection{Construction of the POD basis}\label{sec:constructPODbasis}
\noindent
In this section, we shall present our model reduction method to solve Eq.\eqref{eq:vis-hatbar}.  Since the evolution equation for 
$\bar{u}$ depends on $\hat{u}$, we first consider to solve the equation for $\hat{u}$, i.e., 
\begin{equation}\label{eq:hat}
\hat{u}_t+\vec{V}\cdot\nabla\hat{u}-dS_l\Delta\hat{u}+S_l\big|\vec{P}+\nabla\hat{u}\big|-\int_{\mathbb{T}^n}S_l\big|\vec{P}+\nabla\hat{u}\big|d\vec{x}+\vec{V}\cdot\vec{P}=0,  ~\text{in}~  \mathbb{T}^n\times(0,\infty). 
%\hat{u}(\vec{x},0)=0 &\text{on}\ \mathbb{T}^2\times\{t=0\}
\end{equation}
Let $H_{per}^1(\mathbb{T}^n)$ denote the Sobolev space on the domain $\mathbb{T}^n$ with a periodic boundary condition and let $\langle\cdot,\cdot\rangle$ denote the standard inner product on $L^2(\mathbb{T}^n)$. Let $H\subset H_{per}^1(\mathbb{T}^n)$ be the subspace consisting of all mean-free functions. Since $H$ is a closed subspace of $H_{per}^1(\mathbb{T}^n)$, $H$ itself is a Hilbert space. Let $\langle\cdot,\cdot\rangle_H$ denote the standard inner product on $H$. Define the bilinear form $a(\cdot,\cdot):H\times H\to\mathbb{R}$ to be
\begin{equation}\label{def:a}
a(u,v)=\int_{\mathbb{T}^n}\big[(\vec{V}\cdot\nabla u)v+dS_l(\nabla u\cdot\nabla v)\big]d\vec{x}.
\end{equation}
Also, define a nonlinear map from $H$ to $L^2(\mathbb{T}^n)$ to be
\begin{equation}\label{def:F}
F(u)=S_l\big|\vec{P}+\nabla u\big|-\int_{\mathbb{T}^n}S_l\big|\vec{P}+\nabla u\big|d\vec{x}.
\end{equation}
The weak formulation associated with the Eq.\eqref{eq:hat} is
\begin{equation}\label{eq:weak}
\langle\hat{u}_t,\psi\rangle+a(\hat{u},\psi)+\langle F(\hat{u}),\psi\rangle=\langle-\vec{P}\cdot\vec{V},\psi\rangle,\ \forall\psi\in H, 
%\hat{u}(\vec{x},0)=0
\end{equation}
which can be solved by using numerical methods. %, such as the finite element method. 

\textcolor{black}{The POD technique for discretization of Eq.\eqref{eq:hat} requires snapshots, which can be obtained by an independent numerical method or by the appropriate technological means related to a specific application, for instance the data from an experiment. We apply a higher-order WENO scheme and TVD-RK scheme to 
solve the problem \eqref{eq:vis-GonT}, and extract the mean part to obtain a set of numerical solutions $\{\hat{u}(\cdot,t_k)\}$ to Eq.\eqref{eq:hat}, where $t_k=k\Delta t$, $\Delta t=T/m$, $k=0,...,m$.} Then, we use $2m+1$ snapshots $\{\hat{u}(\cdot,t_0),\dots,\hat{u}(\cdot,t_m), \bar\partial\hat{u}(\cdot,t_1),\dots,\bar\partial\hat{u}(\cdot,t_m)\}$, where $\bar\partial\hat{u}(t_i)=(\hat{u}(t_i)-\hat{u}(t_{i-1}))/\Delta t$, to generate the POD basis functions. Let $S^r=\{\psi_1(\cdot),\psi_2(\cdot),\cdots,\psi_r(\cdot)\}$ denote the $r$-dimensional orthonormal POD basis functions obtained from the $2m+1$ snapshots, which minimize the following error 

\begin{align}\label{eq:PODerror}
\frac{1}{2m+1}\sum_{i=0}^{m}\Big\lVert\hat{u}(t_i)-\sum_{j=1}^{r}\langle\hat{u}(t_i),\psi_j\rangle_H\psi_j\Big\rVert^2_H+\frac{1}{2m+1}\sum_{i=1}^{m}\Big\lVert\bar{\partial}\hat{u}(t_i)-\sum_{j=1}^{r}\langle \bar{\partial}\hat{u}(t_i),\psi_j\rangle_H\psi_j\Big\rVert^2_H. 
\end{align} 
See the appendix \ref{sec:appedix_pod} for the details of the POD method. By adding the terms $\bar\partial\hat{u}(t_i)$, $i=1,...,m$ into the snapshots, we obtain more accurate POD basis functions and avoid an extra $(\Delta t)^{-2}$ factor in the convergence analysis; see Section \ref{sec:ConvergenceAnalysis} for more details.  

\textcolor{black}{In practice, we can use experimental data or reference numerical solutions to generate solution snapshots. The choice of snapshots is critical to the quality of the POD method. The snapshots should span a linear space that approximates the solution space of the original PDEs well. However, the POD method itself gives no guidance on how to select the snapshots (page 503 of \cite{benner2015survey}). One may need some \textit{a posteriori} error estimate for the solution obtained by the POD method. For time-dependent parametric PDEs, it is difficult to obtain \textit{a posteriori} error estimate for the PDEs}. \textcolor{black}{We propose an adaptive strategy to enrich the POD basis; see Section \ref{sec:numPODtimeperiodicfluid} for more details. We refer the interested reader to  \cite{gugercin2004survey,rapun2010reduced,alla2013adaptive,peherstorfer2015dynamic,rapun2015adaptive,lyu2017computing,grassle2018pod,lyu2019computing} and reference therein for the adaptive techniques in the model reduction methods. 
} 
 
%\begin{remark}
%In practice, we can use experimental data or reference numerical methods to generate solution snapshots. 
%Moreover, we can compute a set of numerical solutions $\{G(\cdot,t)\}$ to Eq.\eqref{eq:vis-Gaffine} and extract the mean-free parts to obtain our solution snapshots. 
%\end{remark}	
\begin{remark}
The construction of the POD basis can be costly. \textcolor{black}{One can apply randomized projection methods \cite{alla2019randomized} to reduce the cost in the offline stage}. Once the construction is done, the POD basis can be used to solve viscous G-equations with different parameters, which brings significant computational savings in the online stage.   	
\end{remark}
%\begin{remark}
%??? One can straightforwardly apply POD method to solve viscous G-equations, which is easy to implement but difficult to carry out the convergence analysis 	
%\end{remark}
 
%The reason why we introduce this decomposition is that the mean-free property of $\hat{u}$ plays a crucial role in the proof of error estimates of the backward Galerkin POD method in the later discussion. 
%%%%%%%%%%%%%%%%%%%%%%%%%%%%%%%%%%%%%%%%%

\subsection{A backward Euler and POD-based Galerkin method}\label{sec:PODGalerkin}
\noindent
The POD basis functions provide an efficient approach to approximate the solution in the physical space. 
If we choose the backward Euler scheme to discretize the time space, we obtain a backward Euler 
and POD-based Galerkin method to solve Eq.\eqref{eq:hat}. Specifically, let $\hat{U}_k \equiv \sum_{i=1}^ra_{i}(t_k)\psi_i$ denote the numerical solution at $t=t_k$, where $\psi_i$'s are the POD basis
functions. We want to find solutions $\{\hat{U}_k\}_{k=0}^m\subset S^r$ satisfying
\begin{equation}\label{eq:weak_backward}
\begin{cases}
\langle\bar\partial\hat{U}_k,\psi\rangle+a(\hat{U}_k,\psi)+\langle F({\hat{U}_k}),\psi\rangle=\langle-\vec{P}\cdot\vec{V},\psi\rangle,\ &\forall\psi\in S^r,\\
\hat{U}_0=0,
\end{cases}
\end{equation}
where $\bar\partial\hat{U}_k = (\hat{U}_k-\hat{U}_{k-1})/\Delta t$. By choosing the test function $\psi$ to be $\psi_i$, $i=1,...,r$ in \eqref{eq:weak_backward} and letting $\textbf{a}_k=\big(a_{1}(t_k),\cdots,a_{r}(t_k)\big)^T$ denote 
the coefficient vector, we obtain a nonlinear equation system for $\textbf{a}_k$ as  
\begin{equation}\label{eq:POD_scheme}
\textbf{M}_1\textbf{a}_k=\textbf{M}_2\textbf{a}_{k-1}+\textbf{c}+\textbf{f}_k, 
\end{equation}
where $\textbf{M}_1,\textbf{M}_2\in\mathbb{R}^{r\times r}$ with $(\textbf{M}_1)_{ij}=\langle\psi_i,\psi_j\rangle+\Delta t\cdot a(\psi_i,\psi_j)$ and $(\textbf{M}_2)_{ij}=\langle\psi_i,\psi_j\rangle$, $\textbf{c}\in\mathbb{R}^r$ with $c_i=-\Delta t\langle\vec{V}\cdot\vec{P},\psi_i\rangle$, and $F_k\in\mathbb{R}^r$ with $(\textbf{f}_k)_i=-\Delta t\big\langle F(\sum_i^ra_{i}(t_k)\psi_i),\psi_i\big\rangle$. The matrices $\textbf{M}_1$, $\textbf{M}_2$ and 
$\textbf{c}$ can be pre-computed and saved.

The Eq.\eqref{eq:POD_scheme} can be efficiently computed using the Newton's method, where the solution 
at time $t_{k-1}$ can be chosen as the initial guess for $\textbf{a}_k$. The proposed method is very efficient since the number of the POD basis is small in general. \textcolor{black}{For more challenging problems, such as 3D problems or convection-dominated problems (leading to an increase of the number of the POD basis), the effect of the nonlinear term becomes significant. One can choose the discrete empirical interpolation method (DEIM) for nonlinear model reduction \cite{chaturantabut2010}}.  

We solve Eq.\eqref{eq:POD_scheme} to get the POD solution $\hat{U}_k$, which is the numerical solution of the mean-free part $\hat{u}(\vec{x},t)$ of the G-equation \eqref{eq:hat}. As indicated by the Eq.\eqref{eq:vis-hatbar},  $u(\vec{x},t)$ is recovered by add the mean solution back

\begin{equation}\label{eq:recoveru}
u(\vec{x},t)=\hat{u}(\vec{x},t)-\int_0^t\int_{\mathbb{T}^n}S_l\big|\vec{P}+\nabla\hat{u}(\vec{x},s)\big|\mathrm{d}\vec{\mathrm{x}}\mathrm{d}\mathrm{s}.
\end{equation}

\textcolor{black}{
Since the solution $\hat{u}(\vec{x},t)$ is smooth, we use a high-order difference scheme to compute the spatial derivatives of $\nabla\hat{u}(\vec{x},\cdot)$ and use extrapolations at boundaries to approximate the spatial derivatives to maintain the second-order accuracy. Then, we apply a composite trapezoidal rule to compute the spatial integration and temporal integration in \eqref{eq:recoveru}. Finally, we obtain the numerical solution to $u(\vec{x},t)$ as follows 
\begin{equation}\label{eq:approx}
U_k=\hat{U}_k-\frac{1}{2}\Delta t\sum_{i=1}^k\bigg[\mathbb{I}\big(S_l|\vec{P}+\nabla\hat{U}_{i-1}(\vec{x})|\big)+\mathbb{I}\big(S_l|\vec{P}+\nabla\hat{U}_{i}(\vec{x})|\big)\bigg],
\end{equation}
where $\mathbb{I}(\cdot)$ denotes the numerical integrator by the composite trapezoidal rule. Moreover, we 
assume the mesh size of the composite trapezoidal rule is $h$, which is determined by the numerical method in 
computing the snapshots. One can also choose a wider finite difference scheme to compute the spatial derivatives and high-order numerical methods to compute the integration in \eqref{eq:recoveru}.  }

\section{Convergence analysis}\label{sec:ConvergenceAnalysis}
\noindent
In this section, we shall present some convergence analysis to show that the accuracy of the numerical solution is guaranteed. Our convergence analysis follows the framework of the Galerkin finite element 
methods for parabolic problems \cite{thomee1984galerkin}. To deal with the nonlinearity of the viscous G-equation, the following lemmas are useful.

%Our main goal in the section is to present an error estimate for the mean square error
%\begin{equation}
%\frac{1}{m}\sum_{k=1}^m||\hat{U}_k-\hat{u}(t_k)||^2_{L^2}
%\end{equation}
\begin{lemma}\label{lemma:nonlinear}
	There exists a constant $\gamma>0$ such that for any $u,v\in H$,
	\begin{equation}
	||F(u)-F(v)||_{L^2}\leq\gamma||u-v||_H.
	\end{equation}
\end{lemma}
\begin{proof}
	According to the definition of $F$ in Eq.\eqref{def:F}, the proof of this lemma  
	can be easily obtained by using the triangle inequality.
\end{proof}

\begin{lemma}\label{lemma:coercive}
	The bilinear form $a(\cdot,\cdot)$ defined in \eqref{def:a} is continuous and coercive, which means that there exist constants $\beta>0$ and $\kappa>0$ such that for any $\psi,\phi\in H$
	\begin{equation}
	a(\psi,\phi)\leq\beta||\psi||_H||\phi||_H,\quad a(\psi,\psi)\geq\kappa||\psi||_H^2.
	\end{equation}
\end{lemma}
\begin{proof}
Let $||\vec{V}||_{\infty}$ denote the maximum amplitude of the vector field $\vec{V}$. 
One has the estimate 

\begin{align*}
a(\psi,\phi)&\leq\int_{\mathbb{T}^n}||\vec{V}||_{\infty} |\nabla\psi|\cdot|\phi|+dS_l|\nabla\psi|\cdot|\nabla\phi|,\\
&\leq ||\vec{V}||_{\infty}||\psi||_H||\phi||_H+dS_l||\psi||_H||\phi||_H.
\end{align*}

The last inequality follows from the Cauchy-Schwarz inequality. Moreover, since $\vec{V}$ is divergence-free, $\vec{V}\cdot\nabla$ is skew-symmetric, which means

\begin{align*}
a(\psi,\psi)=dS_l\int_{\mathbb{T}^n}|\nabla\psi|^2.
\end{align*}
\textcolor{black}{Since $\psi$ is mean-free, the Poincar\'e-Wirtinger inequality implies that there exists a constant $C$, depending only on the domain $\mathbb{T}^n$, such that $\int_{\mathbb{T}^n}|\psi|^2\leq C\int_{\mathbb{T}^n}|\nabla\psi|^2$.} Therefore, we know that $a(\cdot,\cdot)$ is coercive, i.e.,

\begin{equation*}
a(\psi,\psi)\geq\kappa||\psi||^2_H,
\end{equation*}
where $\kappa=dS_l/(C+1)>0$.
\end{proof}

\textcolor{black}{In our method, we decompose the solution into a mean-free part and a mean part. The bilinear form of the evolution equation for the mean-free part satisfies the coercive condition, which plays an essential role in the convergence analysis. And the error estimate for the mean part of the solution can be obtained subsequently. }

Now we define the Ritz-projection $P^r:H\to S^r$, $u\mapsto P^ru$ such that
\begin{equation}\label{eq:ritz}
a(P^ru,\psi)=a(u,\psi),\quad \forall\psi\in S^r.
\end{equation}

Facts from functional analysis guarantee that $P^r$ is well-defined and bounded because $a(\cdot,\cdot)$ is continuous and coercive. More specifically,

\begin{equation}
||P^ru||_H\leq\frac{\beta}{\kappa}||u||_H,\quad \forall u\in H.
\end{equation}

Using the same argument in Lemma 3 and Lemma 4 in \cite{kunisch2001galerkin}, we can prove that $P^r$ has the following approximation property. More details of the approximation property of the POD basis can be found in the appendix \ref{sec:appedix_pod}.
\begin{lemma}\label{lemma:ritz}
	\begin{equation}
	\frac{1}{m}\sum_{k=0}^m||\hat{u}(t_k)-P^r\hat{u}(t_k)||_H^2\leq\frac{3\beta}{\kappa}\sum_{l=r+1}^{m}\lambda_l,
	\end{equation}
	and
	\begin{equation}
	\frac{1}{m}\sum_{k=1}^m||\bar{\partial}\hat{u}(t_k)-P^r\bar{\partial}\hat{u}(t_k)||_H^2\leq\frac{3\beta}{\kappa}\sum_{l=r+1}^{m}\lambda_l,
	\end{equation}
where $\lambda_l$ is the $l$-th largest eigenvalues of the correlation matrix $K$ associated with the solution snapshots. 
\end{lemma}

%\begin{lemma}\label{lemma:ErrorInPhysicalSpace}
%Maybe an assumption, accuracy of the WENO scheme $\hat{u}(t_j)-\hat{U}_j=h^{\alpha}$???
%\end{lemma}

\begin{theorem}\label{thm:errormain1}
	Let $\hat{u}(\cdot)$ and $\{\hat{U}_k\}_{k=0}^m$ be the solutions to Eq.\eqref{eq:weak} and its backward Euler and POD-based Galerkin approximation, respectively. Then for sufficiently small $\Delta t$, there exists a constant $C > 0$ depending on $\hat{u}$, $d$, $S_l$, $\vec{V}$, $\vec{P}$ and $T$ but independent of 
	$r$, $\Delta t$, and $m$ such that
	\begin{equation}
	\frac{1}{m}\sum_{k=1}^m\big|\big|\hat{U}_k-\hat{u}(t_k)\big|\big|^2_{L^2}\leq C\big((\Delta t)^2+\sum_{l=r+1}^{m}\lambda_l\big).
	\label{eq:estimatemeanfree}
	\end{equation}
\end{theorem}

\begin{proof}
	For $k=0,1,\dots,m$, define $\vartheta_k=\hat{U}_k-P^r\hat{u}(t_k)$ and $\varrho_k=P^r\hat{u}(t_k)-\hat{u}(t_k)$. Then
	
	\begin{equation}
	\frac{1}{m}\sum_{k=1}^m||\hat{U}_k-\hat{u}(t_k)||^2_{L^2}\leq\frac{2}{m}\sum_{k=1}^m||\vartheta_k||^2_{L^2}+\frac{2}{m}\sum_{k=1}^m||\varrho_k||^2_{L^2},
	\end{equation}
	
	\begin{equation}
	\frac{2}{m}\sum_{k=1}^m||\varrho_k||^2_{L^2}\leq\frac{2}{m}\sum_{k=1}^m||\varrho_k||^2_H\leq\frac{6\beta}{\kappa}\sum_{l=r+1}^{m}\lambda_l.
	\end{equation}
	Define $\bar{\partial}\vartheta_k=(\vartheta_k-\vartheta_{k-1})/\Delta t$. For all $\psi\in S^r$,
	
	\begin{equation}
	\langle\bar\partial\vartheta_k,\psi\rangle+a(\vartheta_k,\psi)=\langle\upsilon_k,\psi\rangle+\langle F(\hat{u}(t_k))- F({\hat{U}_k}),\psi\rangle,
	\end{equation}
	where $\upsilon_k=\hat{u}_t(t_k)-P^r\bar\partial\hat{u}(t_k)=\hat{u}_t(t_k)-\bar\partial\hat{u}(t_k)+\bar\partial\hat{u}(t_k)-P^r\bar\partial\hat{u}(t_k)$.
	Define $\omega_k=\hat{u}_t(t_k)-\bar\partial\hat{u}(t_k)$ and $\eta_k=\bar\partial\hat{u}(t_k)-P^r\bar\partial\hat{u}(t_k)$. Taking $\psi=\vartheta_k\in S^r$ in the previous equality, we obtain
	
	\begin{equation}
	\frac{1}{\Delta t}\big(||\vartheta_k||_{L^2}^2-\langle\vartheta_k,\vartheta_{k-1}\rangle\big)+\kappa||\vartheta_k||_{L^2}^2\leq||F(\hat{u}(t_k))- F({\hat{U}_k})||_{L^2}||\vartheta_k||_{L^2}+||\vartheta_k||_{L^2}||\upsilon_k||_{L^2}.
	\end{equation}
	By using the Lemma~\ref{lemma:nonlinear}, we get
	
	\begin{align*}
	||F(\hat{u}(t_k))- F({\hat{U}_k})||_{L^2}||\vartheta_k||_{L^2}&\leq\gamma||\hat{u}(t_k)-\hat{U}_k||_H||\vartheta_k||_{L^2},\\
	&\leq\gamma\big(||\varrho_k||_H+||\vartheta_k||_H\big)||\vartheta_k||_{L^2}.
	\end{align*}
	Since $\vartheta_k\in S^r$, which is a finite dimensional space, the norms defined on $S^r$ are equivalent. This means that there exists some constant $C_1>0$ such that
	
	\begin{align*}
	||F(\hat{u}(t_k))- F({\hat{U}_k})||_{L^2}||\vartheta_k||_{L^2}&\leq C_1\big(||\varrho_k||_H+||\vartheta_k||_{L^2}\big)||\vartheta_k||_{L^2},\\
	&\leq\frac{C_1}{2}||\varrho_k||_H^2+\frac{3C_1}{2}||\vartheta_k||_{L^2}^2.
	\end{align*}
	Combining these inequalities, we have
	
	\begin{align*}
	||\vartheta_k||_{L^2}^2-\langle\vartheta_k,\vartheta_{k-1}\rangle\leq\Delta t\Big(||\vartheta_k||_{L^2}||\upsilon_k||_{L^2}+\frac{C_1}{2}||\varrho_k||_H^2+\frac{3C_1}{2}||\vartheta_k||_{L^2}^2\Big).
	\end{align*}
	By using the inequality of arithmetic and geometric means, we obtain
	
	\begin{align*}
	\big(1-(1+C_1)\Delta t\big)||\vartheta_k||_{L^2}^2\leq||\vartheta_{k-1}||_{L^2}^2+\Delta t\big(||\upsilon_k||_{L^2}^2+3C_1||\varrho_k||_H^2\big).
	\end{align*}
	For sufficiently small $\Delta t$, there exists a constant $C_2>0$ such that
	
	\begin{align}
	||\vartheta_k||_{L^2}^2\leq(1+C_2\Delta t)\bigg(||\vartheta_{k-1}||_{L^2}^2+\Delta t\big(||\upsilon_k||_{L^2}^2+3C_1||\varrho_k||_H^2\big)\bigg).\label{eq:iterativelyinequality}
	\end{align}
	By iteratively using the inequality \eqref{eq:iterativelyinequality}, we have
	
	\begin{align}
	||\vartheta_k||_{L^2}^2\leq e^{C_2T}\bigg(||\vartheta_0||_{L^2}^2+\Delta t\sum_{j=1}^k\big(||\upsilon_j||_{L^2}^2+3C_1||\varrho_j||_H^2\big)\bigg).
	\label{eq:iterativelyinequality2}
	\end{align}
	
	Note that $\vartheta_0=0$ in our case. Therefore, we sum the inequality \eqref{eq:iterativelyinequality2} 
	from $k=1$ to $m$ and arrive at
	
	\begin{align*}
	\sum_{k=1}^m||\vartheta_k||_{L^2}^2&\leq e^{C_2T}\Delta t\sum_{k=1}^m\sum_{j=1}^k\big(||\upsilon_j||_{L^2}^2+3C_1||\varrho_j||_H^2\big),\\
	&\leq e^{C_2T}T\sum_{j=1}^m\big(||\upsilon_j||_{L^2}^2+3C_1||\varrho_j||_H^2\big),\\
	&\leq e^{C_2T}T\sum_{j=1}^m\big(2||\omega_j||_{L^2}^2+2||\eta_j||_{L^2}^2+3C_1||\varrho_j||_H^2\big).
	\end{align*}
	Therefore,
	
	\begin{align*}
	\frac{1}{m}\sum_{k=1}^m||\hat{U}_k-\hat{u}(t_k)||^2_{L^2}&\leq\frac{2}{m}\sum_{k=1}^m||\vartheta_k||^2_{L^2}+\frac{6\beta}{\kappa}\sum_{k=r+1}^{m}\lambda_k,\\
	&\leq \frac{2e^{C_2T}T}{m}\sum_{j=1}^m\big(2||\omega_j||_{L^2}^2+2||\eta_j||_{L^2}^2+3C_1||\varrho_j||_H^2\big)+\frac{6\beta}{\kappa}\sum_{l=r+1}^{m}\lambda_l.
	\end{align*}
	\textcolor{black}{
	From $\omega_j=\hat{u}_t(t_j)-\bar\partial\hat{u}(t_j)=\frac{1}{\Delta t}\int_{t_{j-1}}^{t_{j}}(t-t_{j-1})\hat{u}_{tt}(t)dt$, we obtain that} 
	\textcolor{black}{
	\begin{equation}\label{eq:omega}
	\sum_{j=1}^m||\omega_j||_{L^2}^2\leq \sum_{j=1}^m\frac{1}{(\Delta t)^2}\int_{t_{j-1}}^{t_{j}}(t-t_{j-1})^2dt \int_{t_{j-1}}^{t_{j}}||\hat{u}_{tt}(t)||^2_{L^2}dt
	\leq\frac{\Delta t}{3}\int_0^T||\hat{u}_{tt}(t)||^2_{L^2}.
	\end{equation}}
	Together with Lemma~\ref{lemma:ritz}, we obtain
	
	\begin{equation*}
	\frac{1}{m}\sum_{k=1}^m||\hat{U}_k-\hat{u}(t_k)||^2_{L^2}\leq\bigg(\frac{4e^{C_2T}}{3}\int_0^T||\hat{u}_{tt}(t)||^2_{L^2}\bigg)(\Delta t)^2+C_3\sum_{l=r+1}^{m}\lambda_l,
	\end{equation*}
	where $C_3=\frac{3\beta}{\kappa}\big(2e^{C_2T}T(2+3C_1)+2\big)$ and the fact $\Delta t=T/m$ is used. This completes the proof for Theorem \ref{thm:errormain1}.
\end{proof}
Theorem \ref{thm:errormain1} provides an error estimate for the mean-free part. The mean part depends on the mean-free part and the original solution can be recovered afterward. 
Thus, we obtain the error estimate for the original solution as follows.
%\begin{equation}
%\frac{1}{m}\sum_{k=1}^m||{U}_k-{u}(t_k)||^2_{L^2}.
%\end{equation}
\begin{theorem}\label{thm:errormain2}
	Let $u(\cdot)$ and $\{{U}_k\}_{k=0}^m$ be the solutions to Eq.\eqref{eq:vis-GonT} and its numerical approximation Eq.\eqref{eq:approx} based on the backward Euler and POD-based Galerkin scheme, respectively. Then for sufficiently small $\Delta t$, there exists a constant $C'\geq 0$ depending on $\beta$, $\kappa$, $\hat{u}$, $d$, $S_l$, $\vec{V}$, $\vec{P}$ and $T$ but independent of $r$, $\Delta t$, $h$, and $m$, such that
	\textcolor{black}{
	\begin{align}\label{eq:main2}
	\frac{1}{m}\sum_{k=1}^m||U_k-u(t_k)||^2_{L^2}\leq C'\big((\Delta t)^2+\sum_{l=r+1}^{m}\lambda_l + h^{4}\big).
	\end{align}}
\end{theorem}

\begin{proof}
	Let $C$ be the constant appearing in Theorem \ref{thm:errormain1}; see Eq.\eqref{eq:estimatemeanfree}. Recall that 
	
	\begin{align}
	u(t_k)=\hat{u}(t_k)+\bar{u}(t_k)=\hat{u}(t_k)-\int_0^{t_k}\int_{\mathbb{T}^n}S_l\big|\vec{P}+\nabla\hat{u}(\vec{x},t)\big|d\textbf{x}dt   \label{errorestimate-Uexc}
	\end{align}		
	and 		
	\textcolor{black}{
	\begin{align}
	U_k=\hat{U}_k+\bar{U}_k=\hat{U}_k-\frac{1}{2}\Delta t\sum_{i=1}^k\bigg[\mathbb{I}\big(S_l|\vec{P}+\nabla\hat{U}_{i-1}(\vec{x})|\big)+\mathbb{I}\big(S_l|\vec{P}+\nabla\hat{U}_{i}(\vec{x})|\big)\bigg],  \label{errorestimate-Unum}
	\end{align}}
	where $\mathbb{I}(\cdot)$ is the numerical integrator by the composite trapezoidal rule. We also denote $\tilde{U}_{i}=-\mathbb{I}\big(S_l|\vec{P}+\nabla\hat{U}_{i}(\vec{x})|\big)$, $i=0,1,...,k$. Then, we obtain

	\begin{align}
	\frac{1}{m}\sum_{k=1}^m\big|\big|U_k-u(t_k)\big|\big|^2_{L^2}\leq\frac{2}{m}\sum_{k=1}^m\big|\big|\hat{U}_k-\hat{u}(t_k)\big|\big|^2_{L^2}+\frac{2}{m}\sum_{k=1}^m\big|\bar{u}(t_k)-\bar{U}_k\big|^2.
	\label{eq:errordecomposition}
	\end{align}
	
	From Theorem \ref{thm:errormain1}, we get an estimate for the first part on the RHS of the inequality \eqref{eq:errordecomposition}. In the sequel, we shall estimate the second part of the RHS of the inequality \eqref{eq:errordecomposition}. We consider the following decomposition 
	
	\textcolor{black}{
	\begin{align} 
	\sum_{k=1}^m\big|\bar{u}(t_k)-\bar{U}_k\big|^2&\leq2\sum_{k=1}^m\Big|\bar{u}(t_k)-\sum_{j=1}^k
	\frac{1}{2}\Delta t\big(\bar{u}_t(t_{j-1})+\bar{u}_t(t_j)\big)\Big|^2 \nonumber\\
	&+2\sum_{k=1}^m\Big|\sum_{j=1}^k
	\frac{1}{2}\Delta t\big(\bar{u}_t(t_{j-1})+\bar{u}_t(t_j)\big)-
	\sum_{j=1}^k
	\frac{1}{2}\Delta t\big(\tilde{U}_{j-1}+\tilde{U}_{j}\big)
	\Big|^2.\label{eq:decom}
	\end{align}}
	The first term on the RHS of the inequality \eqref{eq:decom} is bounded by
	
	\textcolor{black}{
	\begin{align}
	&2\sum_{k=1}^m\Big|\bar{u}(t_k)-\sum_{j=1}^k
	\frac{1}{2}\Delta t\big(\bar{u}_t(t_{j-1})+\bar{u}_t(t_j)\big)\Big|^2,\nonumber\\
	\leq &2(\Delta t)^2\sum_{k=1}^m\Big|\sum_{j=1}^k\Big(\frac{\bar{u}(t_j)-\bar{u}(t_{j-1})}{\Delta t}-\frac{1}{2}\big(\bar{u}_t(t_{j-1})+\bar{u}_t(t_j)\big)\Big)\Big|^2,\nonumber\\
	\leq& 2(\Delta t)^2\sum_{k=1}^m k\sum_{j=1}^k\Big|\frac{\bar{u}(t_j)-\bar{u}(t_{j-1})}{\Delta t}-\frac{1}{2}\big(\bar{u}_t(t_{j-1})+\bar{u}_t(t_j)\big)\Big|^2, \nonumber\\
	\leq& 2(\Delta t)^2\sum_{k=1}^mk\frac{(\Delta t)^2}{12}\int_0^T\big|\big|\bar{u}_{ttt}(t)\big|\big|^2_{L^2} \leq\frac{T^2}{12}(\Delta t)^2\int_0^T|\bar{u}_{ttt}(t)|_{L^2}^2.	
	\label{eq:decom2}
	\end{align} }
	The second term on the RHS of Eq.\eqref{eq:decom} is bounded by
	
    \textcolor{black}{
	\begin{align}
	&2\sum_{k=1}^m\Big|\sum_{j=1}^k\frac{1}{2}\Delta t\big(\bar{u}_t(t_{j-1})+\bar{u}_t(t_j)\big)-
	\sum_{j=1}^k\frac{1}{2}\Delta t\big(\tilde{U}_{j-1}+\tilde{U}_{j}\big)\Big|^2 \nonumber\\
	&\leq  (\Delta t)^2\sum_{k=1}^m k\sum_{j=1}^k\Big(\big|\bar{u}_t(t_{j-1})-\tilde{U}_{j-1}\big|^2 + \big|\bar{u}_t(t_{j})-\tilde{U}_{j}\big|^2 \Big),\nonumber\\
	&=  (\Delta t)^2\sum_{k=1}^m k\sum_{j=1}^k
	\Big(\big|\int_{\mathbb{T}^n}S_l\big|\vec{P}+\nabla\hat{u}(\vec{x},t_{j-1})d\textbf{x}- \mathbb{I}\big(S_l|\vec{P}+\nabla\hat{U}_{j-1}(\vec{x})|\big)\big|^2 + \nonumber\\
	&\quad\quad\quad\quad\quad\quad\quad\quad \big|\int_{\mathbb{T}^n}S_l\big|\vec{P}+\nabla\hat{u}(\vec{x},t_{j})d\textbf{x}- \mathbb{I}\big(S_l|\vec{P}+\nabla\hat{U}_{j}(\vec{x})|\big) \big|^2 \Big),\nonumber\\
	&\leq 2 (\Delta t)^2m^2 \sum_{j=0}^{m}\Big(S_l^2 \int_{\mathbb{T}^n}\big|\nabla\hat{u}(\vec{x},t_j)-\nabla\hat{U}_j(\vec{x})\big|^2d\textbf{x}
	+ \Big|\int_{\mathbb{T}^n}S_l\big|\vec{P}+\nabla\hat{U}_{j}(\vec{x})\big|d\textbf{x}- \mathbb{I}\big(S_l|\vec{P}+\nabla\hat{U}_{j}(\vec{x})|\big)\Big|^2
	\Big) ,\nonumber\\
	&\leq 2T^2 S_l^2  \sum_{j=0}^{m}\Big( ||\hat{u}(t_j)-\hat{U}_j||_H^2 + C_4h^4\Big) 
	\leq 2T^2 S_l^2 (\frac{3\beta}{\kappa}\sum_{l=r+1}^{m}\lambda_l +  m C_4 h^4),\label{eq:decom3}	
	\end{align}}
	where $C_4$ comes from the error estimate of the trapezoidal rule and depends on third-order derivatives of the solution with respect to physical variables, and the last inequality follows from the fact that $\hat{u}(t_j)-\hat{U}_j\in S^r$ and norms are equivalent in a finite dimensional space. 
	Substituting the estimate results \eqref{eq:estimatemeanfree}, \eqref{eq:decom2} and \eqref{eq:decom3} into \eqref{eq:errordecomposition}, we obtain 
	
	\begin{align}
     	\frac{1}{m}\sum_{k=1}^m||U_k-u(t_k)||^2_{L^2}\leq C'\big((\Delta t)^2+\sum_{l=r+1}^{m}\lambda_l + h^{4}\big), 
	\end{align}
    where $C'\geq 0$ depends on $\beta$, $\kappa$, $\hat{u}$, $d$, $S_l$, $\vec{V}$, $\vec{P}$ and $T$, but is independent of $r$, $\Delta t$, $h$, and $m$.
\end{proof} 

Theorem \ref{thm:errormain1} and Theorem \ref{thm:errormain2} show that the errors
depend on the sum of eigenvalues of the correlation matrix $K$ (associated with the solution snapshots) 
that are not included in the POD basis. Thus, the decay of eigenvalues provides a quantitative
guidance for choosing the number of the POD basis. A typical approach is to choose $r$ so that 
$\sum_{l=r+1}^{m}\lambda_l\leq \min\{(\Delta t)^2,h^4\}$.

%%%%%%%%%%%%%%%%%%%%%%%%%%%%%%%%%%%%%%%%%

%\section{Stability and Convergence of the POD Method}
%In this section, we prove convergence results of the POD method based on some assumptions on the POD basis.

%%%%%%%%%%%%%%%%%%%%%%%%%%%%%%%%%%%%%%%%%
\section{Numerical Results}\label{sec:NumericalResult}
\noindent
We shall perform numerical experiments to test the performance and accuracy of the proposed method. \textcolor{black}{The numerical experiments are all carried out on a laptop with a 2.5GHz quad-core Intel Core i7 processor and 8GB RAM. The Python codes are published on  GitHub.\footnote{https://github.com/Harry970804/POD\_G-Equation.}} We consider the following viscous G-equation on $[0,1]^2$  with planar initial condition
\begin{equation}\label{eq:Geq_NumericalTests}
\begin{cases}
G_t+\vec{V}\cdot\nabla G+S_l|\nabla G|=dS_l\Delta G &\text{in}\  [0,1]^2\times(0,T),\\
G(\vec{x},0)=\vec{P}\cdot\vec{x} &\text{on}\ [0,1]^2 \times\{t=0\},
\end{cases}
\end{equation}
where $\vec{x}=(x,y)$, $\vec{P}$ is the flame front propagation direction, and the assumption $G(\vec{x}+\vec{z},t)=G(\vec{x},t)+\vec{P}\cdot\vec{z}$ is used. The setting of the fluid velocity $\vec{V}(\vec{x})$ is an important issue in turbulent combustion modeling. 
We first consider a steady incompressible cellular flow,
\begin{equation}\label{eq:velo-field}
\vec{V}(\vec{x})=(V_1,V_2)=\nabla^\perp\mathcal{H}=(-\mathcal{H}_y, \mathcal{H}_x), \quad \mathcal{H}=\frac{A}{2\pi}\sin(2\pi x)\sin(2\pi y), 
\end{equation}
where $A$ is the amplitude of the velocity filed. \textcolor{black}{In the following numerical experiments on this steady flow \eqref{eq:velo-field}, we choose $A=4.0$, $S_l=1$, and $\vec{P}=(1,0)$.} 
In our comparison, the finite difference solution refers to the solution obtained by the reference method to solve \eqref{eq:Geq_NumericalTests}; see Section \ref{sec:appedix_ref} in the appendix for more details. 
While the POD solution refers to the one obtained by our method, i.e., the solution is represented by the POD basis. We choose the method of snapshot \cite{sirovich1987} to construct the POD basis. Details were introduced in Section \ref{sec:constructPODbasis}. 

\textcolor{black}{In the finite difference scheme, we partition the physical domain $[0,1]^2$ into $(N_h+1)\times (N_h+1)$ grids with mesh size $h=1/N_h$ and $N_h=80$. For the time discretization, we choose time step $\Delta t = 1/N_t$ with $N_t=1000$. As such, solving the viscous G-equation from $0$s to $T$s will generate $2TN_t+1$ snapshots.}

\textcolor{black}{In the POD scheme, the algorithm in the Section \ref{sec:constructPODbasis} is first applied to construct a set of POD basis, denoted by $S=\{\psi_1,\cdots,\psi_r\}$, $r\leq\min\{(N_h+1)\times(N_h+1), 2TN_t+1\}$. In practice, we choose the number of the POD basis $r$ to be the smallest integer such that $\sum_{l=r+1}^{N_{snap}}\lambda_l/\sum_{l=1}^{N_{snap}}\lambda_l\leq e_{POD}$, 
where $\lambda_l$'s are the eigenvalues of the covariance matrix of the solution snapshots,
$N_{snap}$ is the number of solution snapshots, and the error threshold $e_{POD}$ is chosen to be 0.001 in our numerical experiments. We choose the inner product in $H_1$ norm in computing the POD basis. Usually, $r$ is assumed to be much smaller than the degree of freedom in the physical space discretization. In each of the following experiments, $r$ will be listed in detail later. But we want to point out in advance that, under most of the parameter settings we have experimented, a small $r$ (around 10) will be good enough to capture true solutions well with a high accuracy. With the POD basis, we further apply the scheme in Section \ref{sec:PODGalerkin} to get the POD solution. The time discretization for the POD scheme is $\Delta t_{POD} = 1/1000$.}

\subsection{Test of the POD basis within the same computational time}\label{sec:Test1}
\noindent
In the first numerical experiment, we choose $d=0.1$ and solve the viscous G-equation \eqref{eq:Geq_NumericalTests}  from time $t=0$ to 1s using the reference numerical scheme and obtain \textcolor{black}{mean-free} solution snapshots. Then, \textcolor{black}{the POD basis functions are constructed from the entire snapshots. In this case, the number of basis functions is $r=4$. Finally, we compute the (mean-free) POD solution on the same time period.}

\textcolor{black}{We compare the mean-free parts obtained by two difference methods at $T=1.0$.}  
Fig.\ref{fig:MeanFreeT1D01} shows that the POD solution agrees well with the reference solution. Moreover, based on the mean-free POD solution, 
we recover the numerical solution of the G-equation according to Eq.\eqref{eq:approx}. 
In Fig.\ref{fig:RecoveredSolutionT1D01}, we show the recovered solution using our POD method and the reference solution.  We can see the good performance of the POD method. 

We continue our experiment by computing Eq.\eqref{eq:Geq_NumericalTests} with different choices of $d$ from $0.01$ to $0.1$. 
We compare the errors between the POD method and the reference method in computing the 
mean-free part of the solution, the mean of the solution and the recovered solution, respectively.  
\textcolor{black}{Table \ref{Table:accuracy} shows the relative error \textcolor{black}{in $L^2$ norm} of these results with different choices of $d$'s, where ``Rela. error, Mean-free'', ``Rela. error, Mean'', and ``Rela. error, Recovered'' mean relative error of the mean-free part of the solution,  relative error of the mean of the solution and  relative error of the recovered solution, respectively}. \textcolor{black}{We can see that: (1) the relative error of the mean-free component is small, which shows the good performance of the POD method; (2) the relative error of the mean part is much smaller than that of the mean-free part, which we think is due to the cancellation of the error in computing the mean; and (3) the relative error of the recovered solution is smaller than that of the mean-free part since the recovered solution usually has a larger magnitude than the mean-free part.  } 

\textcolor{black}{For the dimension of the subspace spanned by the POD basis, we observe that under a fixed threshold of $e_{POD}$, in general, it will increase as $d$ decreases. But it still remains small and it further implies the efficiency of the POD method. In all the experiments we have reported, the dimensions vary in the range from $4$ to $10$.}

In Fig.\ref{fig:RecoveredSolutionT1D005} and Fig.\ref{fig:RecoveredSolutionT1D001}, we show the comparison between the recovered solutions obtained using the POD method and the reference solutions for $d=0.05$ and $d=0.01$, respectively. 
\textcolor{black}{Though we observe some sharp layers appearing in the solution as $d$ decreases, the POD basis can capture these layered structures and give accurate numerical results.} 
%We also find that the mean-free POD solution agrees well with the mean-free part of the reference solution for different diffusive number $d$ (not shown here). 

\textcolor{black}{These results show that the POD method can capture the low-dimensional structures in the regular solutions of the viscous G-equations and provide an efficient model reduction method to approximate the solutions. We remark that when $d$ is extremely small or even $d=0$ the solutions of G-equations are not smooth and cannot be computed by the POD-based Galerkin method. We will exploit some features of the discontinuous Galerkin method and develop new POD methods to address this challenge in our future work.} 
\begin{figure}[h]
	\centering
	\begin{subfigure}{0.325\textwidth}
		\includegraphics[width=1.25\linewidth]{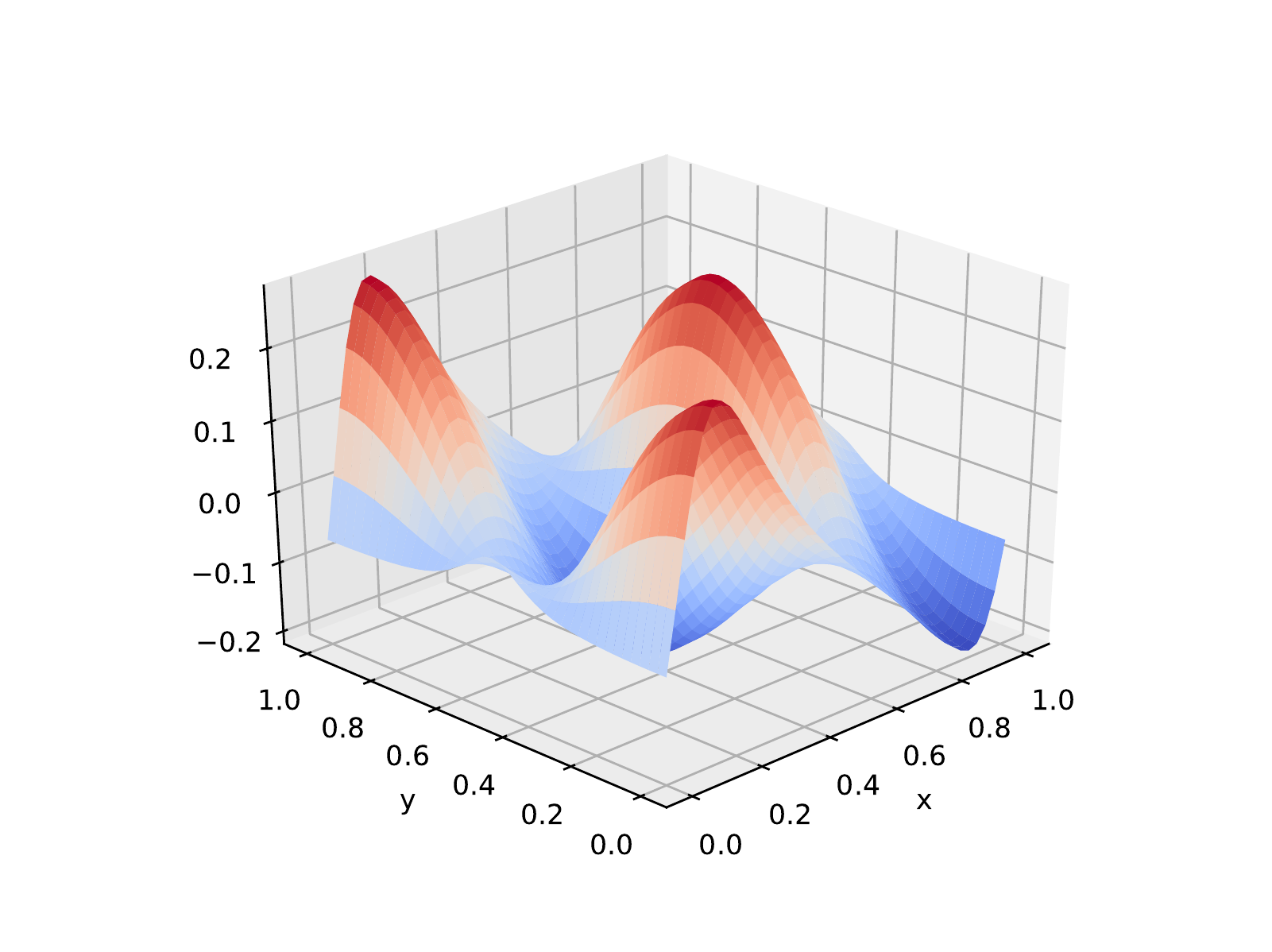}\par 
		\caption{{\small Reference solution.}}
	\end{subfigure}
	\begin{subfigure}{0.325\textwidth}
		\includegraphics[width=1.25\linewidth]{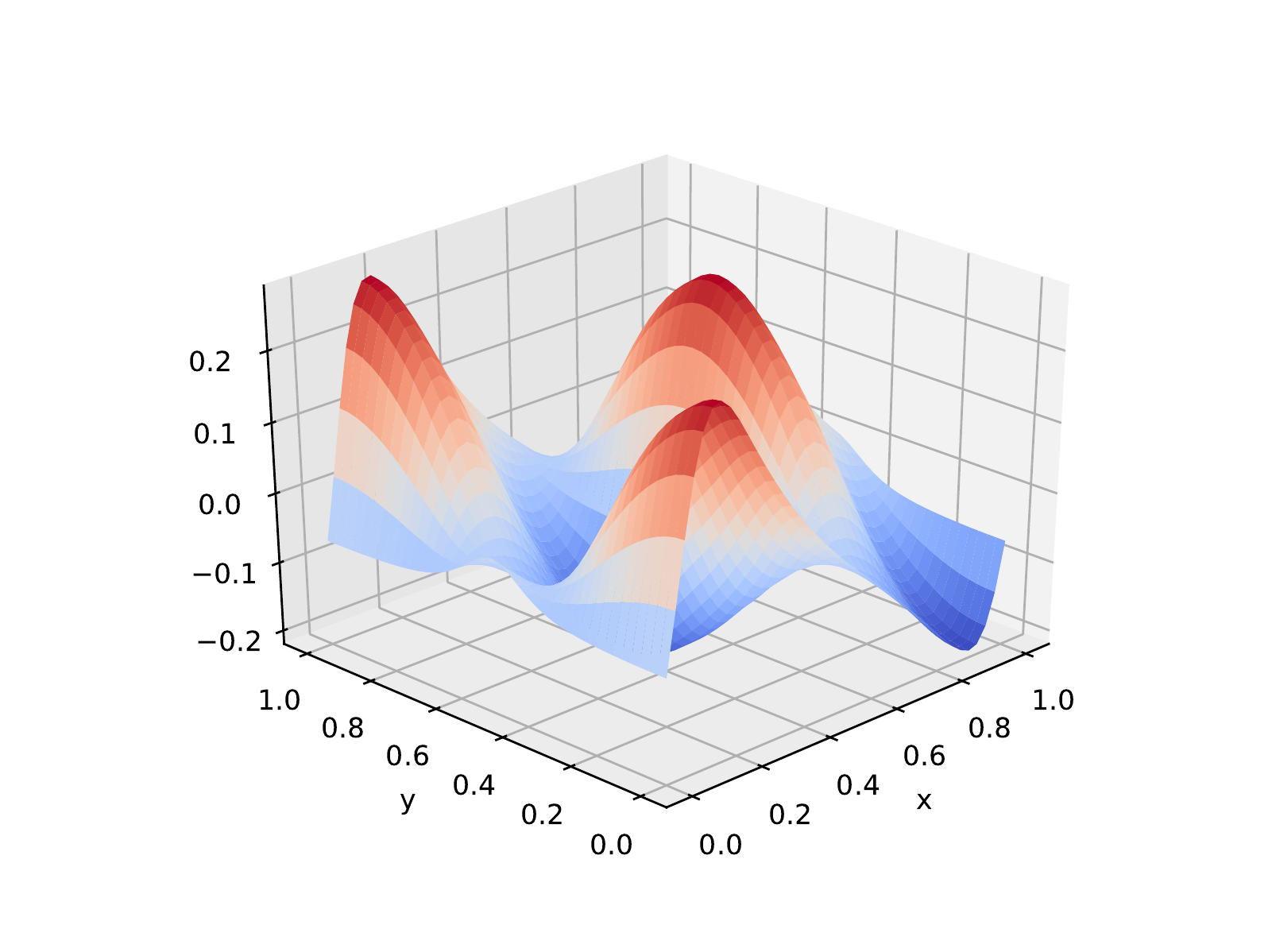}\par 
		\caption{\small  POD solution.}
	\end{subfigure} 
	\begin{subfigure}{0.325\textwidth}
		\includegraphics[width=1.25\linewidth]{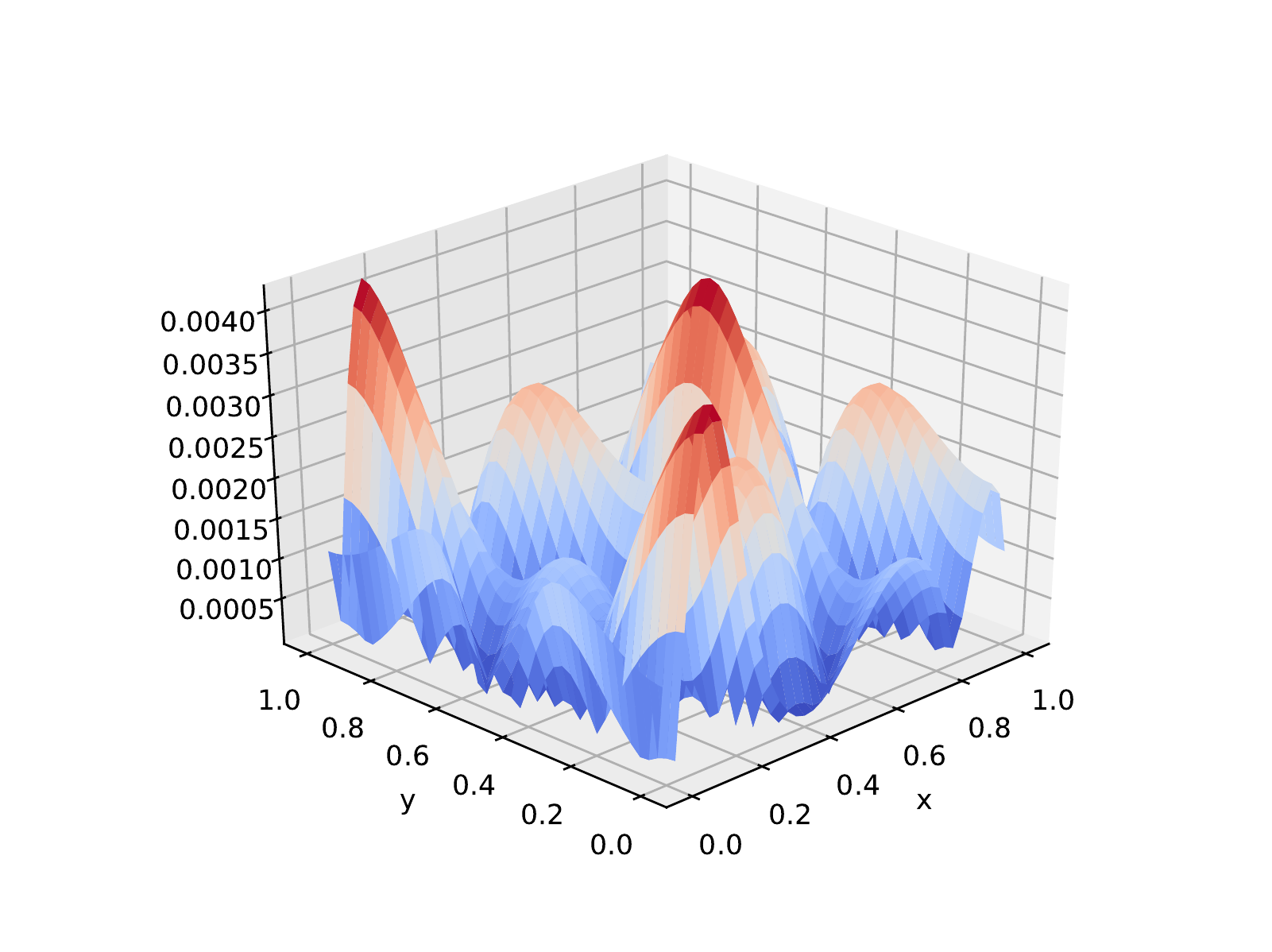}\par 
		\caption{{\small Absolute difference.}}
    \end{subfigure}
	\caption{Mean-free component of the solution at $T=1$ with $d=0.1$.}
	\label{fig:MeanFreeT1D01}
\end{figure}
\begin{figure}[h]
	\centering
	\begin{subfigure}{0.325\textwidth}
		\includegraphics[width=1.25\linewidth]{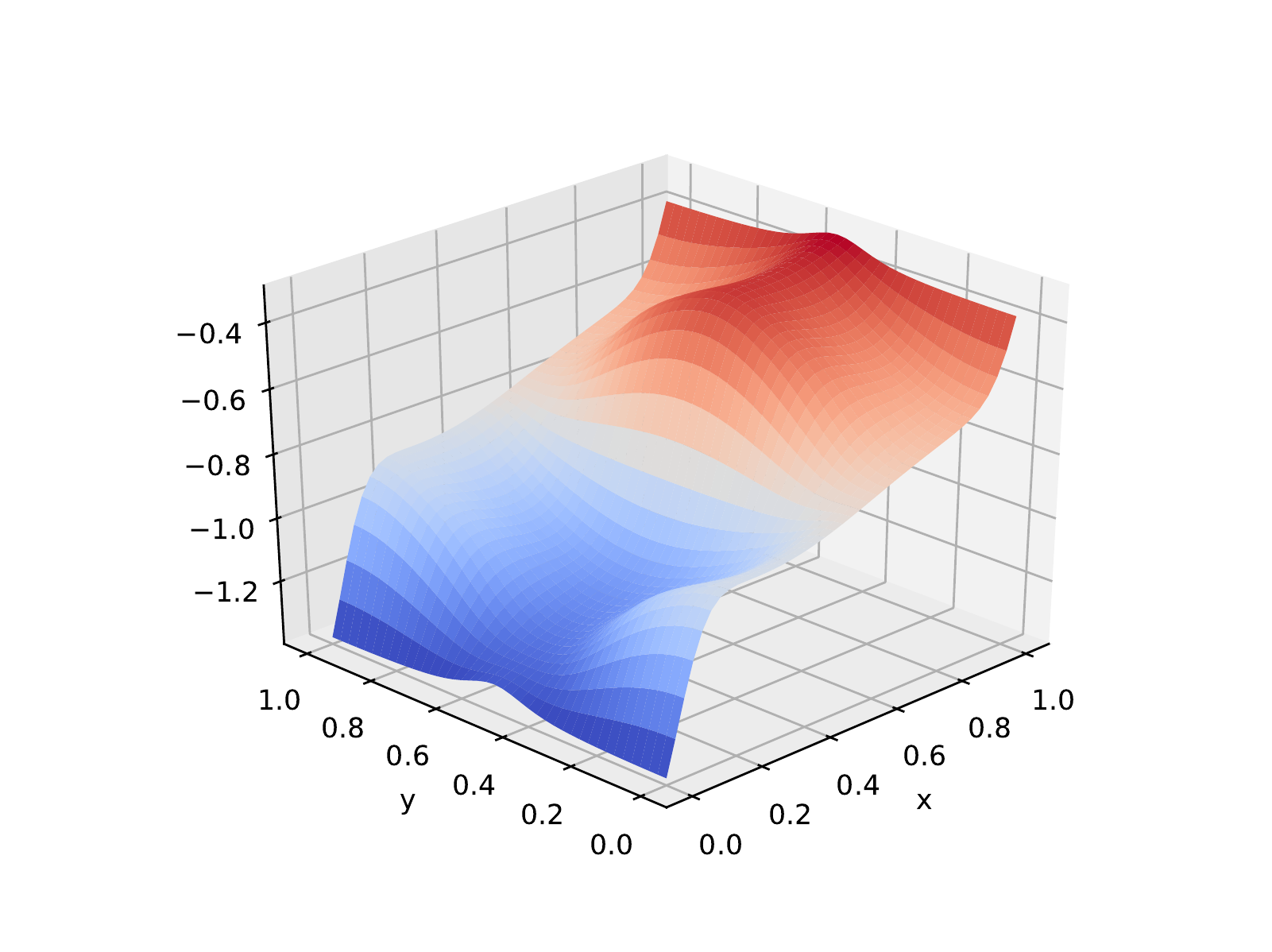}\par
		\caption{{\small Reference solution.}}
    \end{subfigure}
    \begin{subfigure}{0.325\textwidth}
		\includegraphics[width=1.25\linewidth]{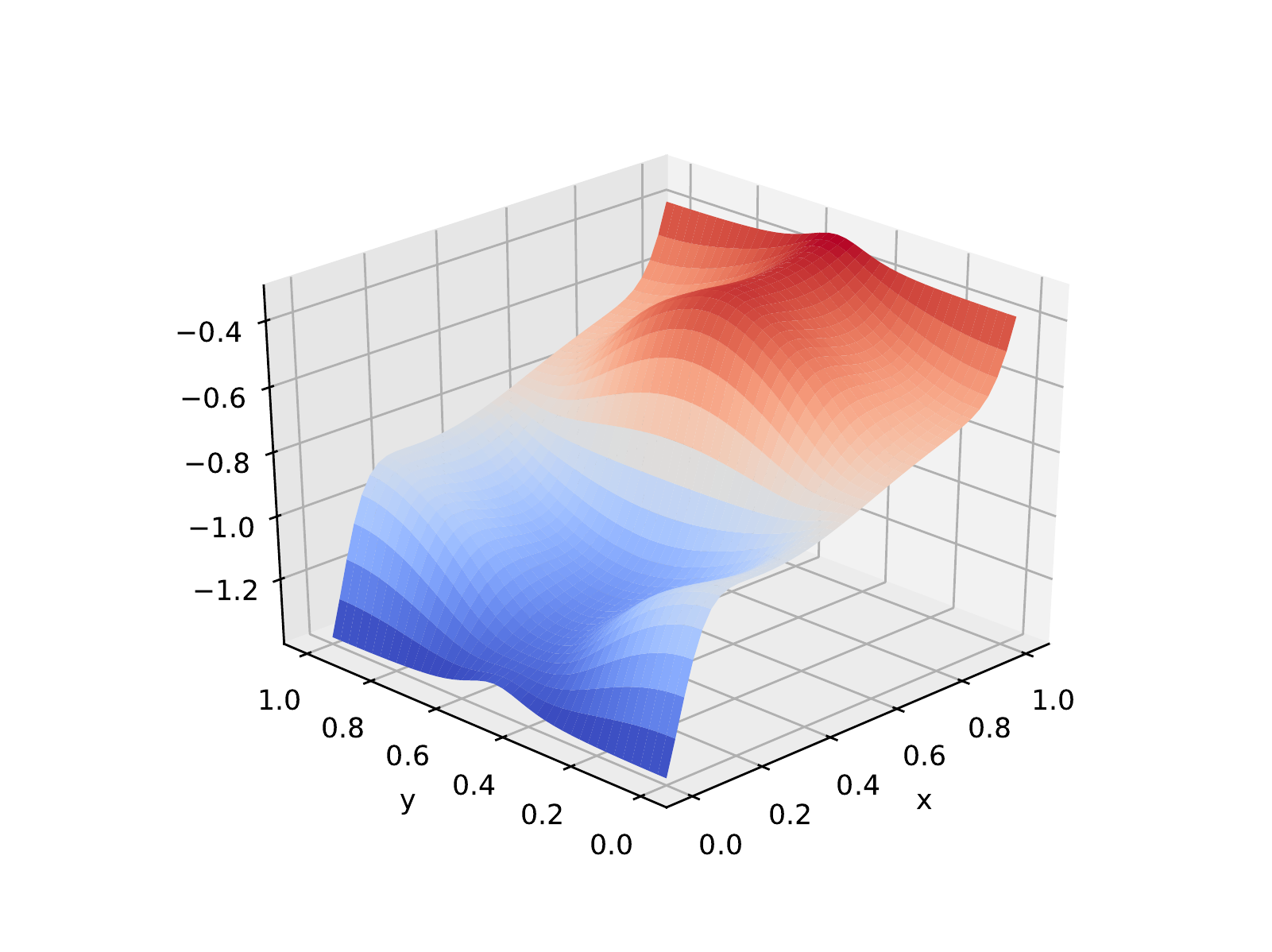}\par 
		\caption{\small  POD solution.}
    \end{subfigure}
	\begin{subfigure}{0.325\textwidth}
		\includegraphics[width=1.25\linewidth]{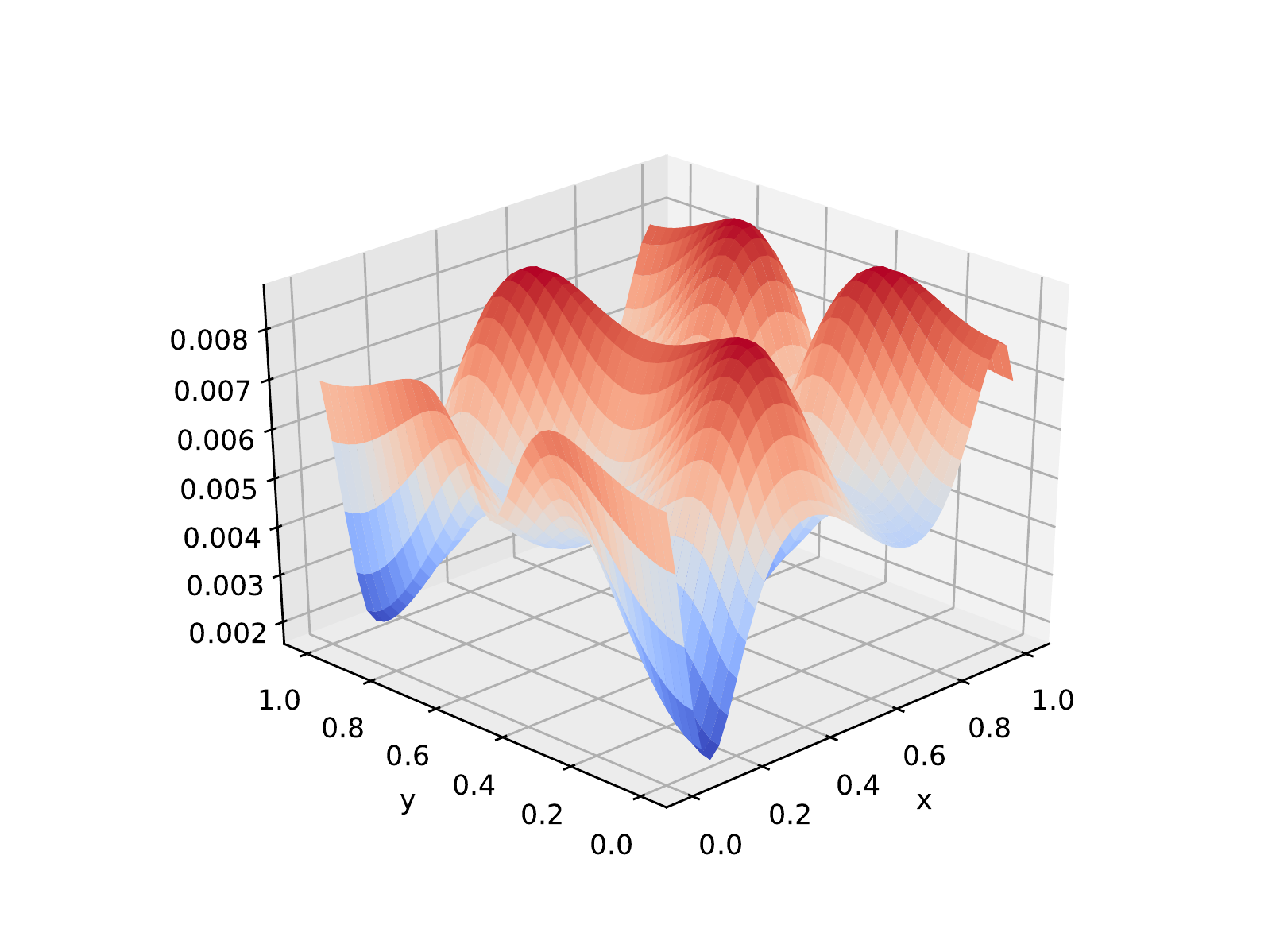}\par 
		\caption{{\small Absolute difference.}}
    \end{subfigure} 
	\caption{Solution of the viscous G-equation at $T=1$ with $d=0.1$.}
	\label{fig:RecoveredSolutionT1D01}
\end{figure}
  
\begin{table}[h]
	\centering
	\begin{tabular}{|c|c|c|c|c|c|c|c|}
		\hline
		\textbf{  d} & \textbf{0.01} & \textbf{0.02} & \textbf{0.03} & \textbf{0.04} & \textbf{0.05} \\ \hline
		\textbf{\textcolor{black}{Rela. error, Mean-free}} & 0.038844  & 0.069712  & 0.025673  & 0.021257  & 0.017649 \\ \hline
		\textbf{\textcolor{black}{Rela. error, Mean}} & 0.020031  & 0.023782  & 0.003230  & 0.002196  & 0.002320 \\ \hline
		\textbf{\textcolor{black}{Rela. error, Recovered}} & 0.020316  & 0.026285  & 0.005110  & 0.003020  & 0.003349 \\ \hline
	\end{tabular}
	\begin{tabular}{|c|c|c|c|c|c|c|c|}
		\hline
		\textbf{  d} & \textbf{0.06} & \textbf{0.07} & \textbf{0.08} & \textbf{0.09} & \textbf{0.1} \\ \hline
		\textbf{\textcolor{black}{Rela. error, Mean-free}} & 0.016216  & 0.014977  & 0.014998  & 0.013786  & 0.013793 \\ \hline
		\textbf{\textcolor{black}{Rela. error, Mean}} & 0.003074  & 0.003258  & 0.004991  & 0.006812  & 0.007230 \\ \hline
		\textbf{\textcolor{black}{Rela. error, Recovered}} & 0.004840  & 0.004997  & 0.005115  & 0.007431  & 0.007085 \\ \hline
	\end{tabular}
	\caption{The relative errors between POD solution and reference solution.}
	\label{Table:accuracy}
\end{table} 
  
\begin{figure}[tbph]
	\centering
	\begin{subfigure}{0.325\textwidth}
		\includegraphics[width=1.25\linewidth]{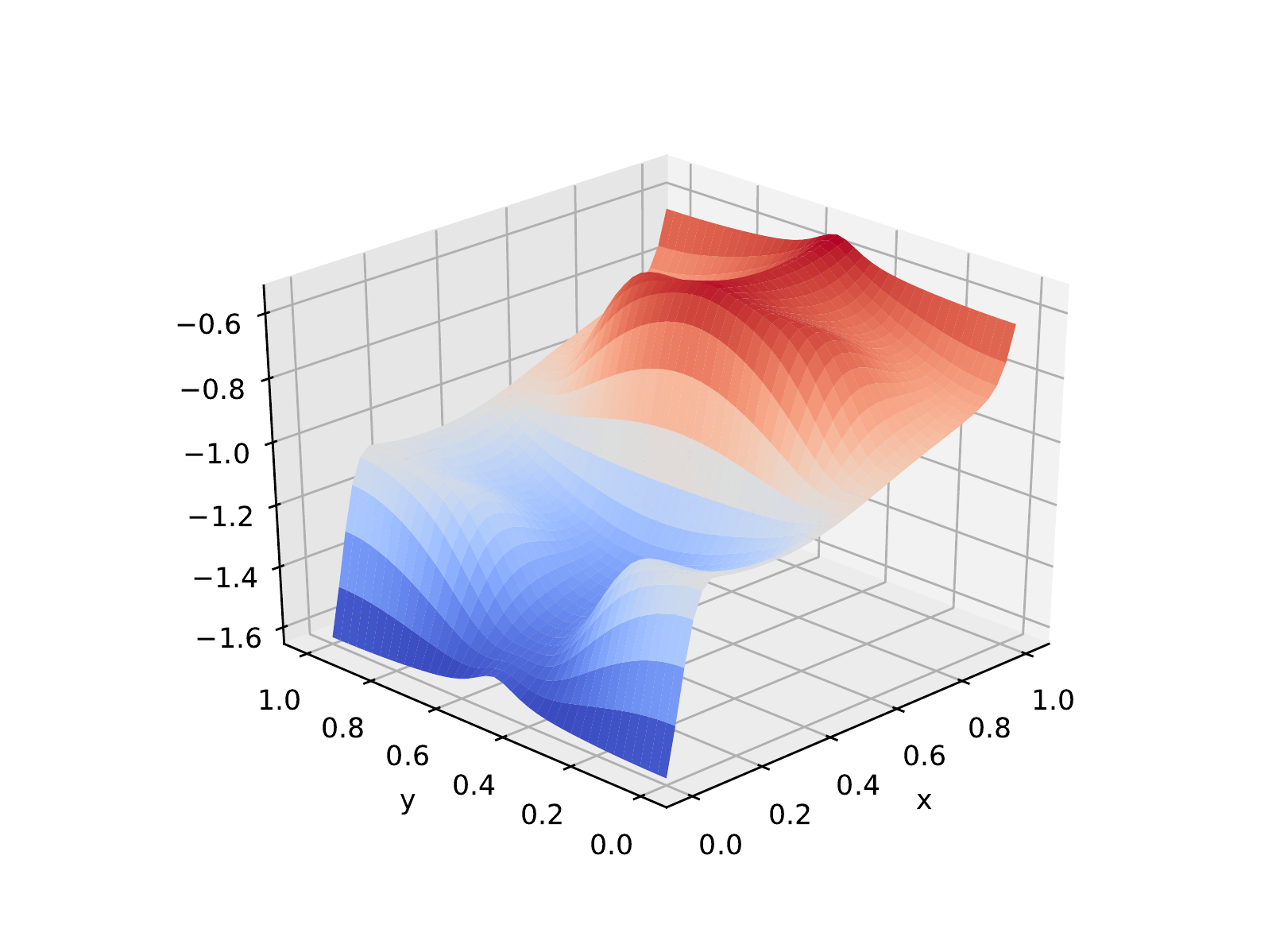}\par
		\caption{{\small Reference solution.}}
	\end{subfigure}
	\begin{subfigure}{0.325\textwidth}
		\includegraphics[width=1.25\linewidth]{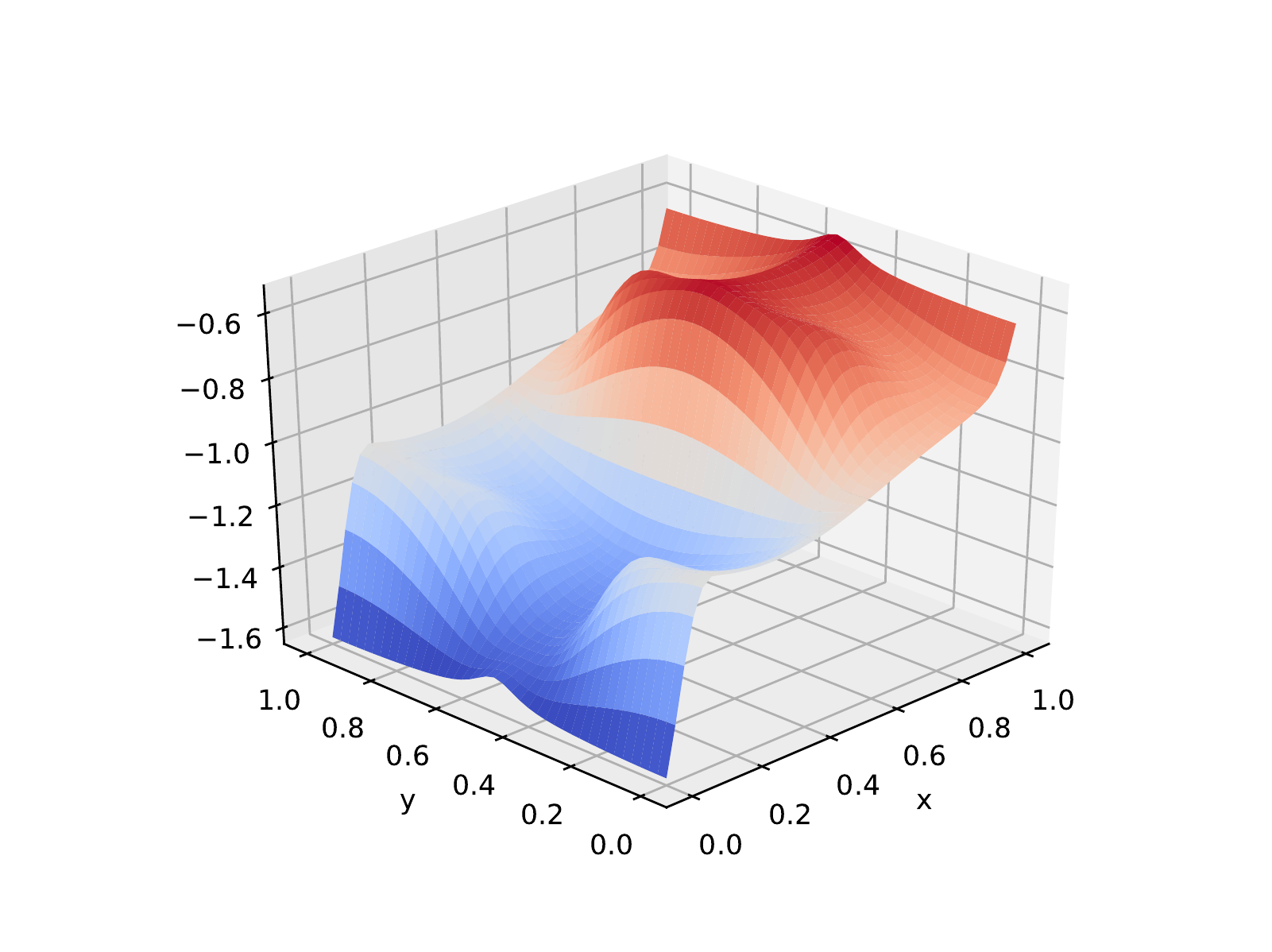}\par 
		\caption{\small POD solution.}
	\end{subfigure}
	\begin{subfigure}{0.325\textwidth}
		\includegraphics[width=1.25\linewidth]{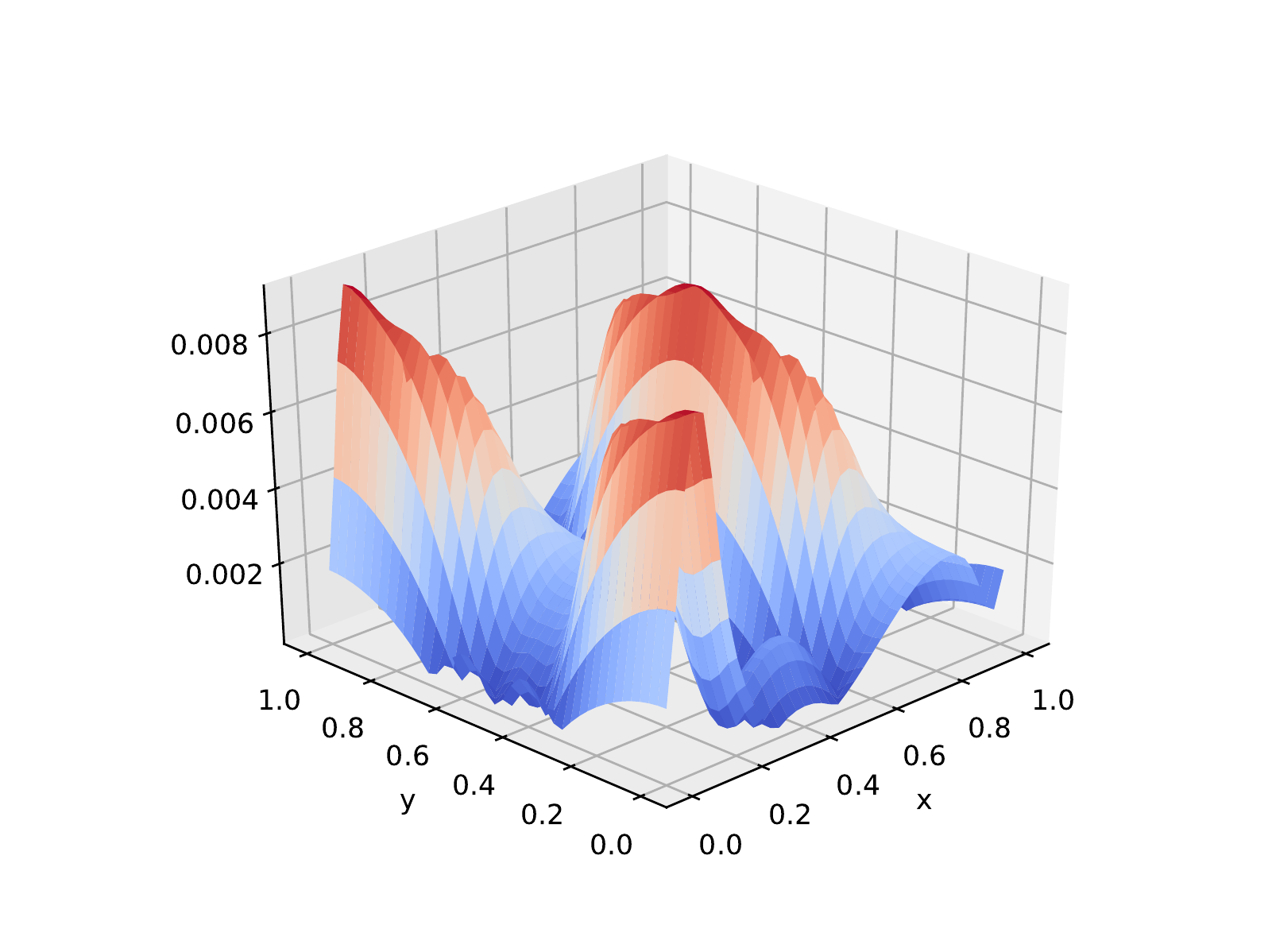}\par 
		\caption{{\small Absolute difference.}}
    \end{subfigure}
	\caption{Solution of the viscous G-equation at $T=1$ with $d=0.05$.}
	\label{fig:RecoveredSolutionT1D005}
\end{figure}  

\begin{figure}[tbph]
	\centering
	\begin{subfigure}{0.325\textwidth}
		\includegraphics[width=1.25\linewidth]{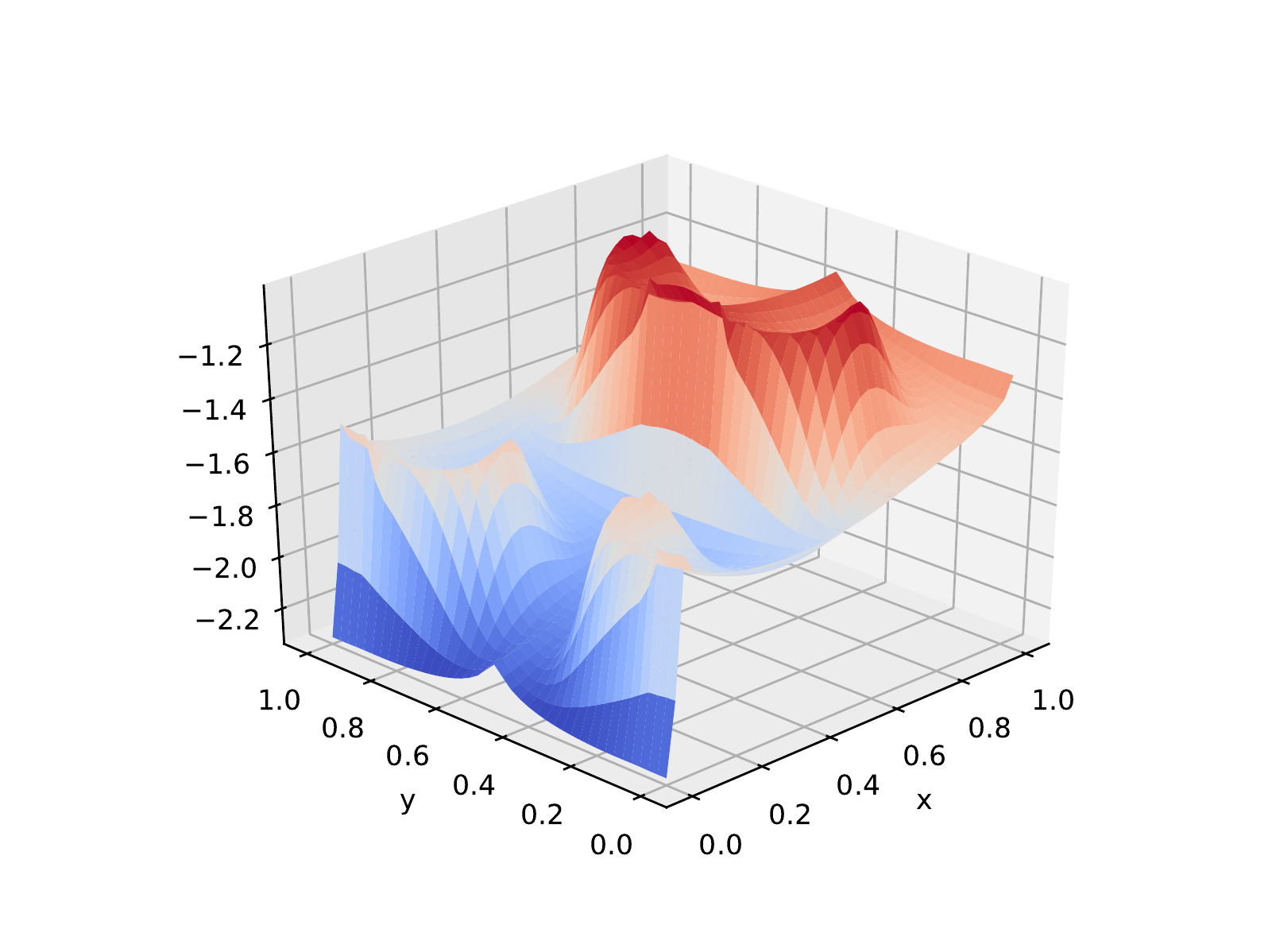}\par
		\caption{{\small Reference solution.}}
	\end{subfigure}
	\begin{subfigure}{0.325\textwidth}
		\includegraphics[width=1.25\linewidth]{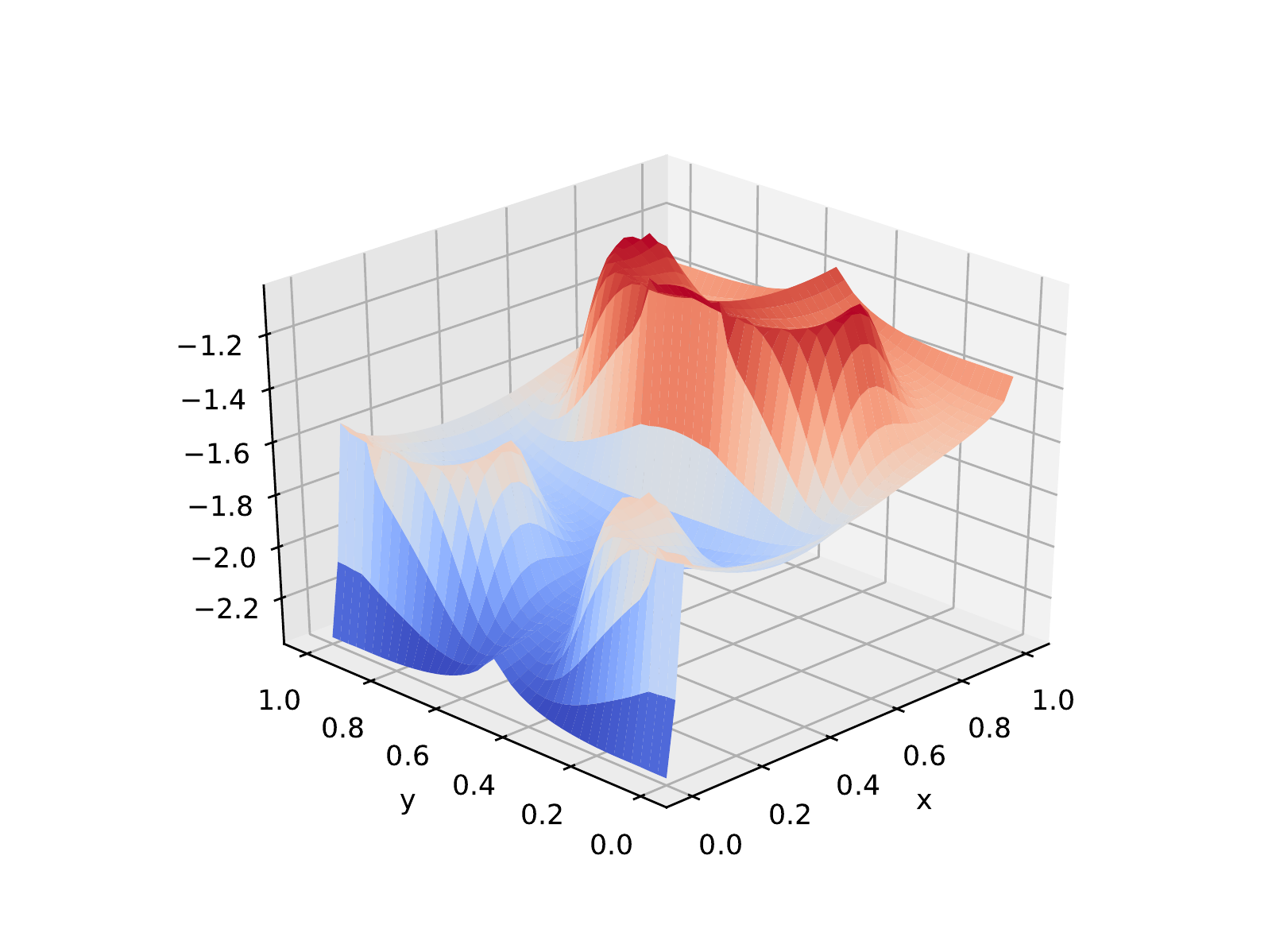}\par 
		\caption{\small POD solution.}
	\end{subfigure}
	\begin{subfigure}{0.325\textwidth}
		\includegraphics[width=1.25\linewidth]{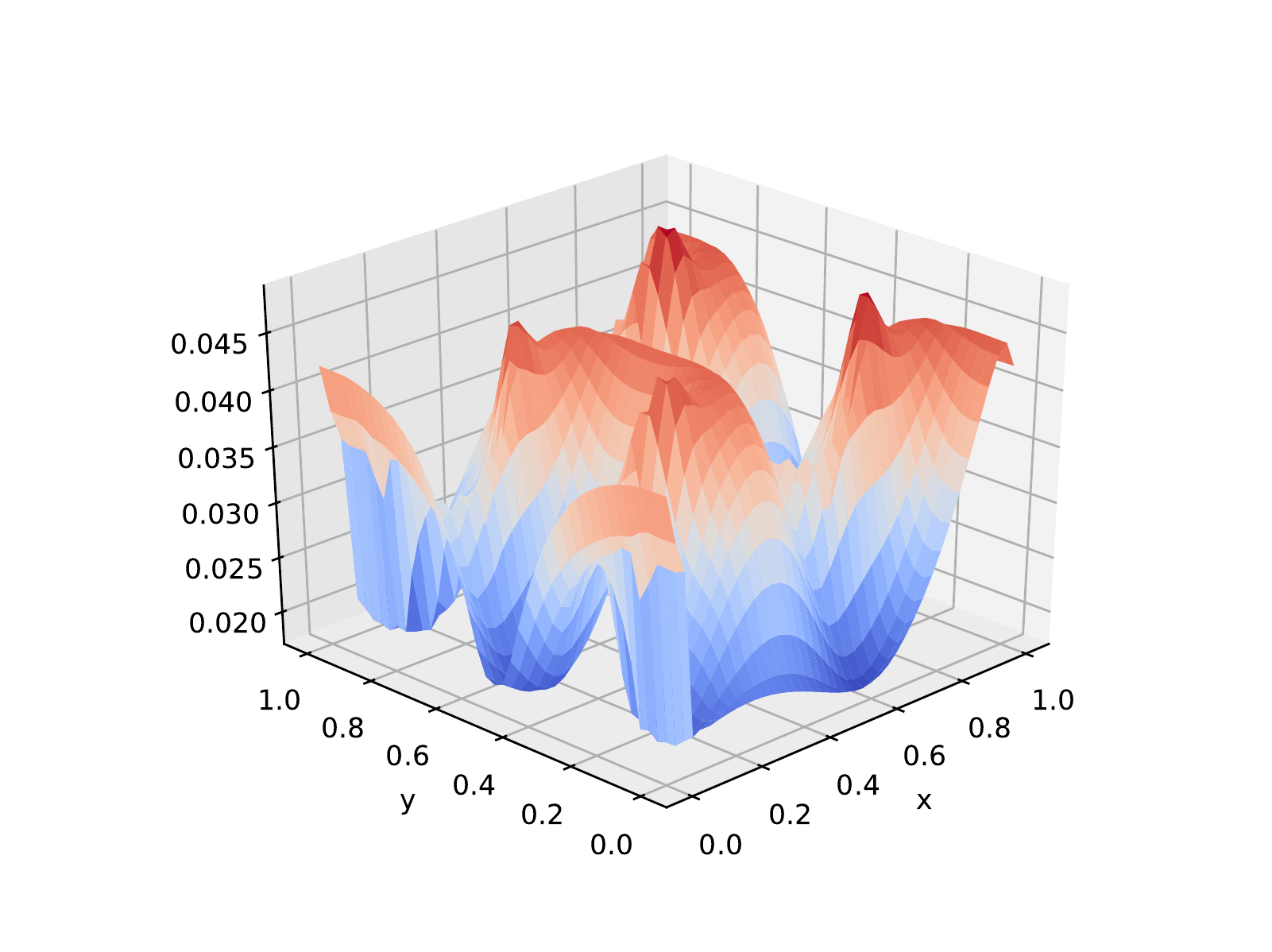}\par 
		\caption{{\small Absolute difference.}}
    \end{subfigure}
	\caption{Solution of the viscous G-equation at $T=1$ with $d=0.01$.}
	\label{fig:RecoveredSolutionT1D001}
\end{figure}  

\subsection{Test of the POD basis for longer time computations}\label{sec:Test2}
\noindent
In this numerical experiment, we first solve the viscous G-equation \eqref{eq:Geq_NumericalTests} from \textcolor{black}{$T=0$ to $T=1$} using the \textcolor{black}{finite difference} scheme to construct POD basis. Then, we adopt the POD method to \textcolor{black}{solve the equation} on a longer time \textcolor{black}{from $T=0$ to $T=2$} using that POD basis. 
%in computing both the mean-free component of the solution and the recovered solution.

\textcolor{black}{We conduct the experiments for $d=0.01,0.02,\dots,0.1$. The number of POD basis functions is fixed to be 6 across all different $d$.}

\textcolor{black}{In Fig.\ref{fig:MeanFreeT2D01} and Fig.\ref{fig:RecoveredSolutionT2D01}, the results for $d=0.1$ is reported in detail.
The POD solution agrees well with the finite difference solution, even through the basis is extracted from the first half of the whole time horizon. It is also worth pointing that the profiles in Fig.\ref{fig:MeanFreeT1D01} and Fig.\ref{fig:MeanFreeT2D01}, 
i.e., the mean-free components of the solutions are almost the same, while the profile in Fig.\ref{fig:RecoveredSolutionT2D01} can be viewed approximately as a downward shift of the profile in Fig.\ref{fig:RecoveredSolutionT1D01} by about 1.4 units. These interesting results show that POD method allows us to capture the stationary solution to the G-equation, which demonstrates its ability to track the long-term behavior of the system. Table \ref{Table:accuracyT2} shows relative errors for other choices of $d$.}
 
\textcolor{black}{In addition, for $d=0.05$ and $d=0.01$, we plot their recovered solutions obtained from both methods in Fig.\ref{fig:RecoveredSolutionT2D005} and Fig.\ref{fig:RecoveredSolutionT2D001} at $T=2$, respectively. 
Again, we find that the POD method performs well for those 2 settings. Moreover, the profiles in Figs.\ref{fig:RecoveredSolutionT2D005} and \ref{fig:RecoveredSolutionT2D001} can also be viewed approximately as downward shifts of the profiles in Figs.\ref{fig:RecoveredSolutionT1D005} and \ref{fig:RecoveredSolutionT1D001}, respectively. The smaller the $d$ is, the bigger shift will be observed. }

\begin{figure}[tbph]
	\centering
	\begin{subfigure}{0.325\textwidth}
		\includegraphics[width=1.25\linewidth]{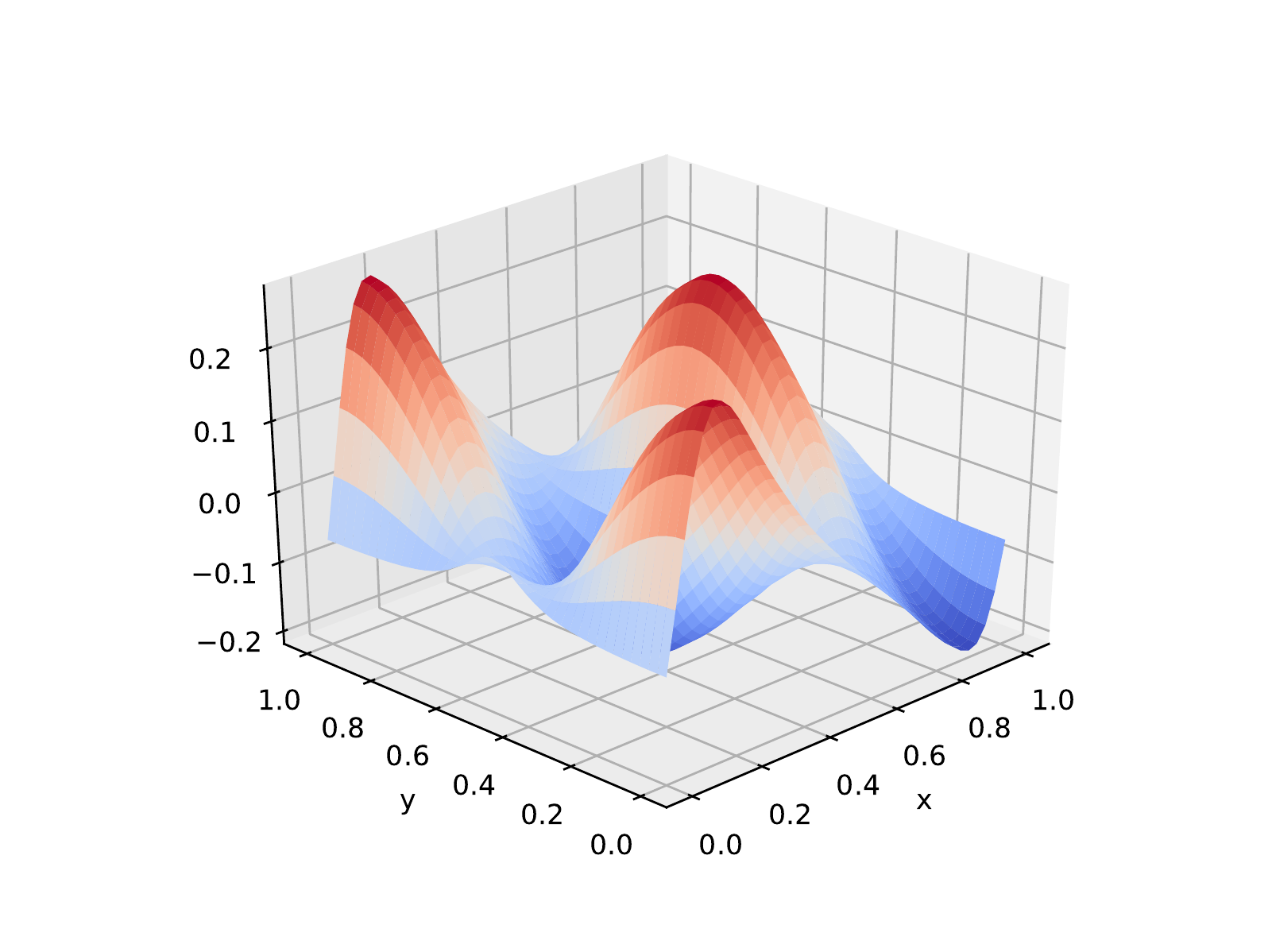}\par 
		\caption{{\small Reference solution.}}
	\end{subfigure}
	\begin{subfigure}{0.325\textwidth}
		\includegraphics[width=1.25\linewidth]{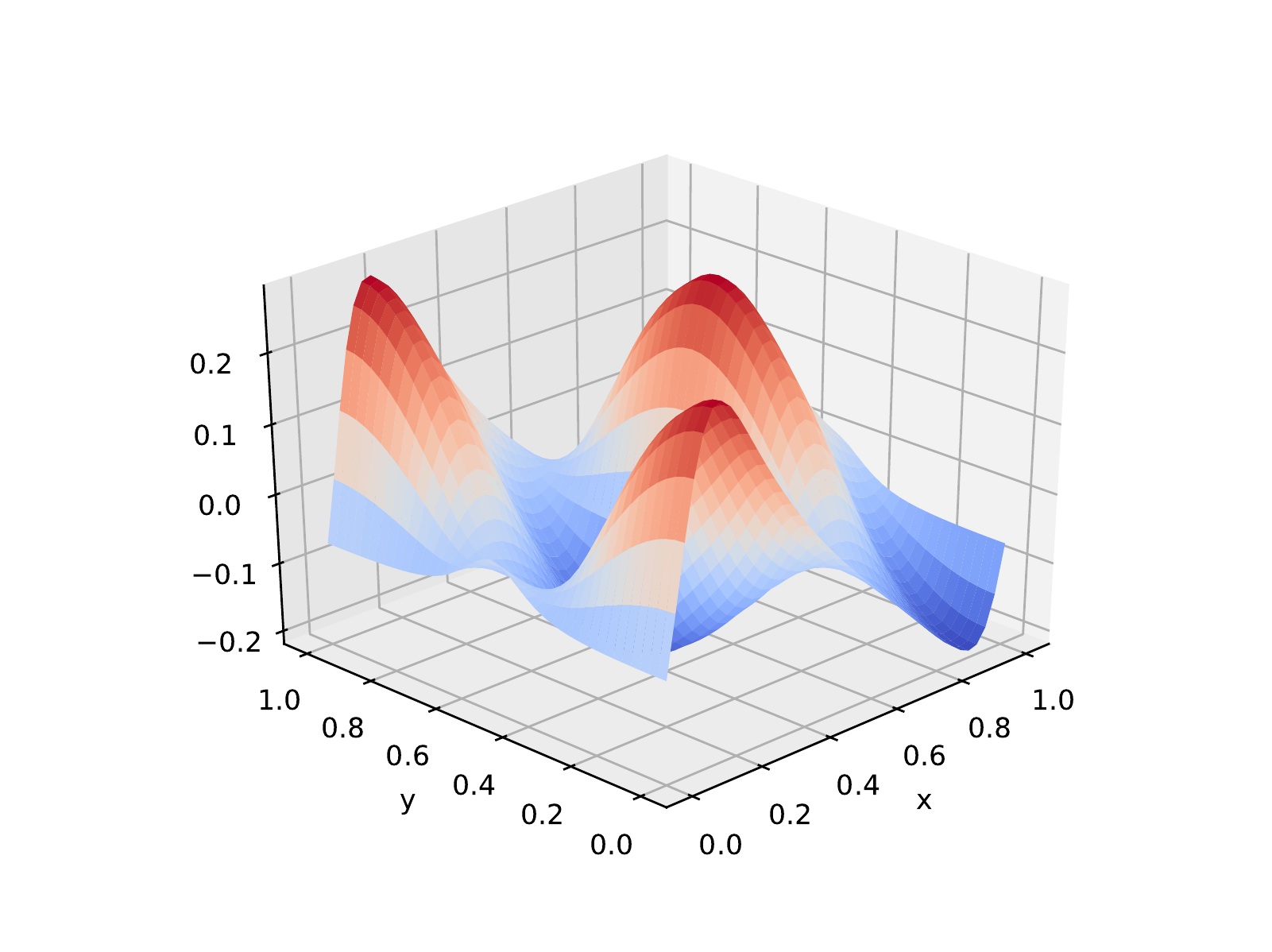}\par 
		\caption{\small POD solution.}
	\end{subfigure}
	\begin{subfigure}{0.325\textwidth}
		\includegraphics[width=1.25\linewidth]{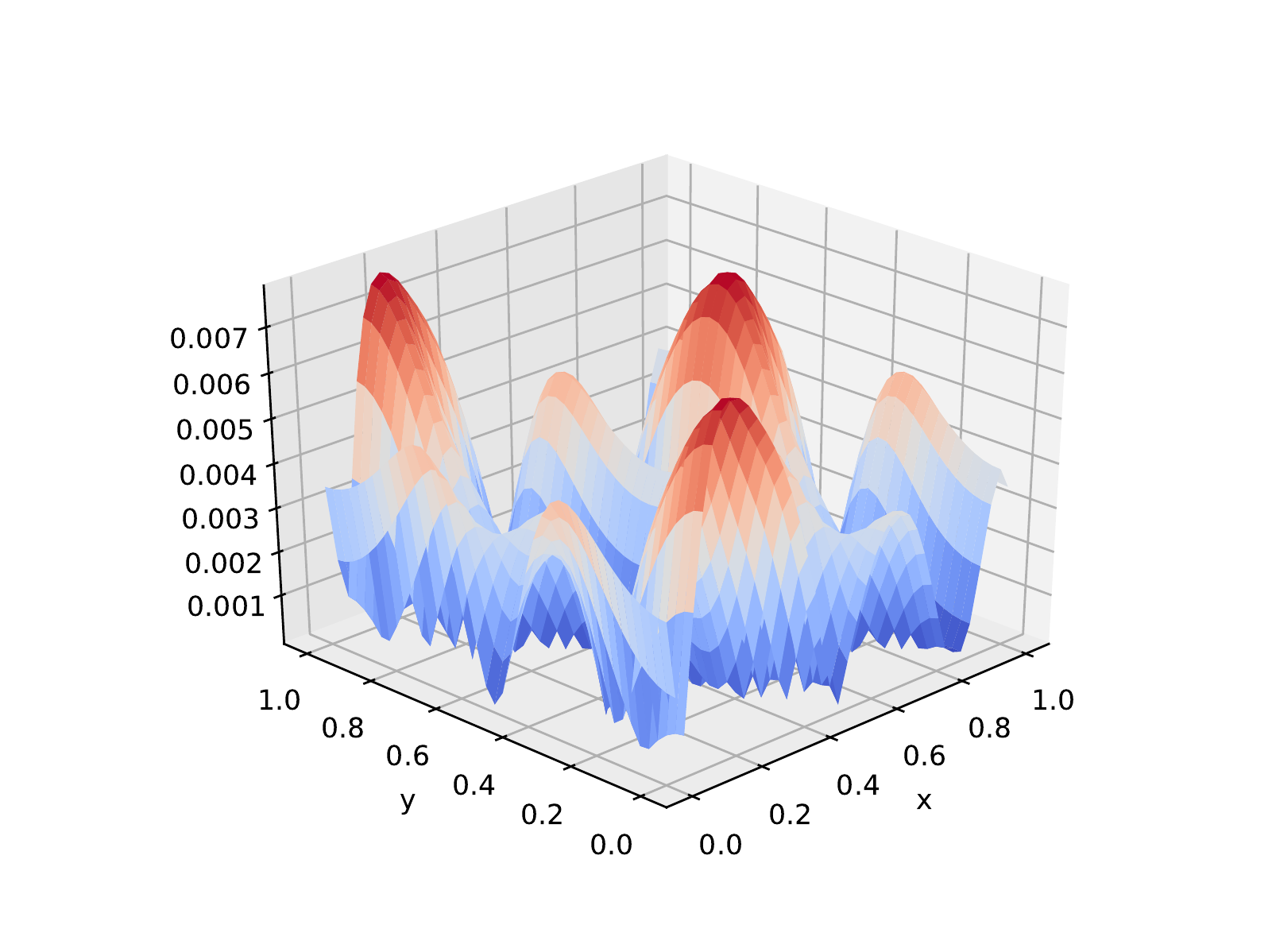}\par 
		\caption{{\small Absolute difference.}}
    \end{subfigure}
	\caption{Mean-free component of the solution at $T=2$ with $d=0.1$.}
    \label{fig:MeanFreeT2D01}
\end{figure} 

\begin{figure}[tbph]
	\centering
	\begin{subfigure}{0.325\textwidth}
		\includegraphics[width=1.25\linewidth]{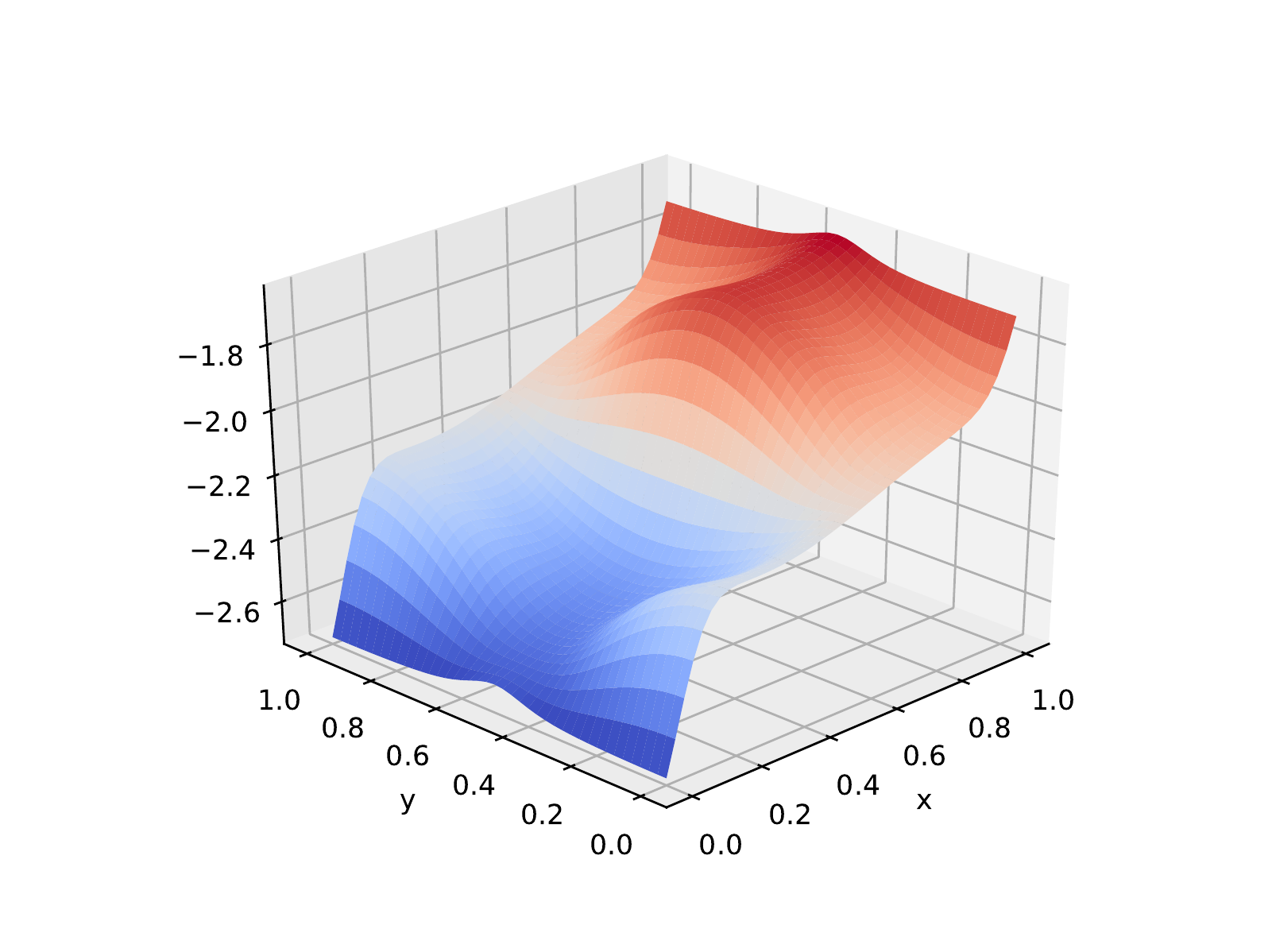}\par
		\caption{{\small Reference solution.}}
	\end{subfigure}
	\begin{subfigure}{0.325\textwidth}
		\includegraphics[width=1.25\linewidth]{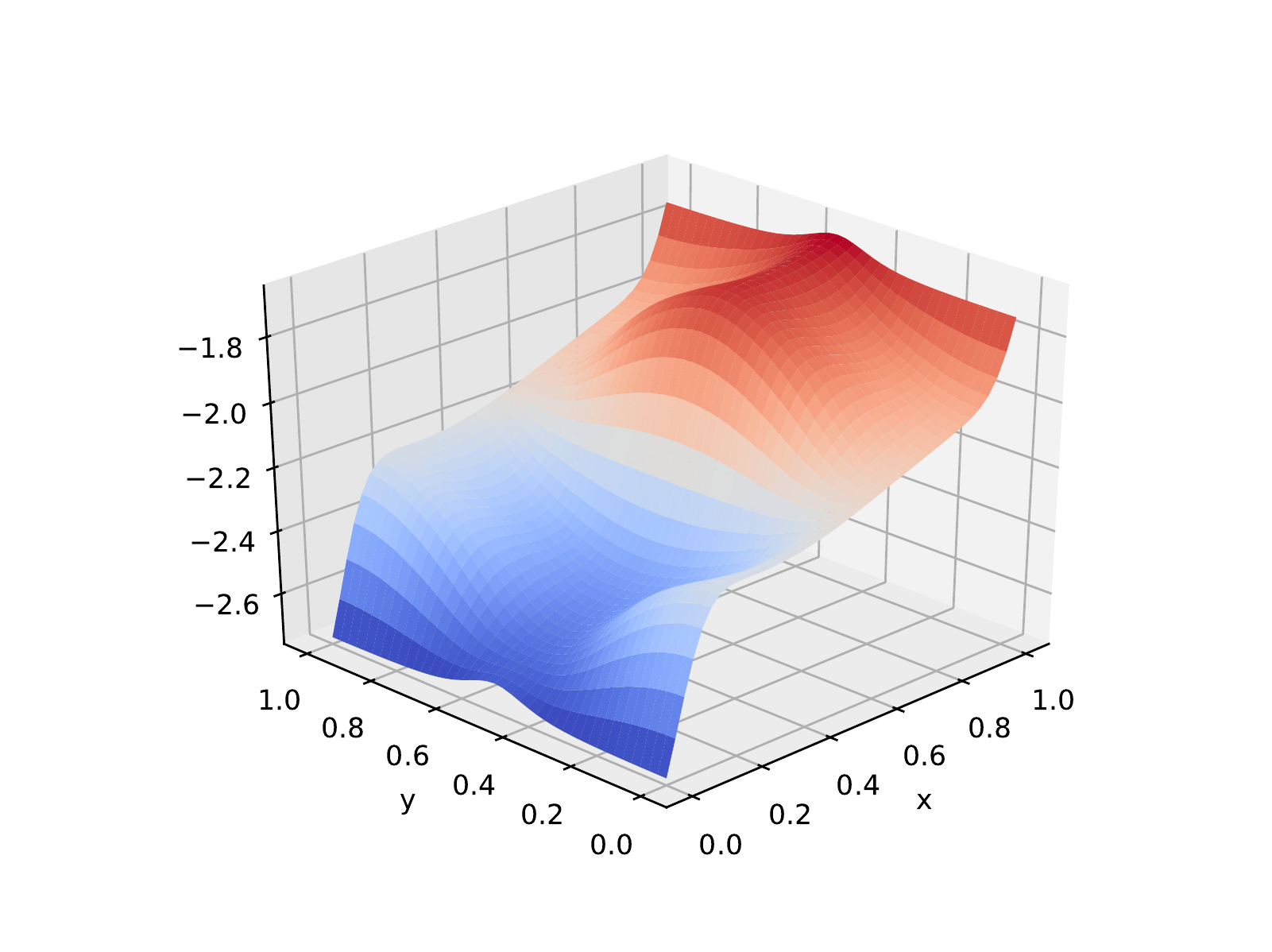}\par 
		\caption{\small POD solution.}
	\end{subfigure}
	\begin{subfigure}{0.325\textwidth}
		\includegraphics[width=1.25\linewidth]{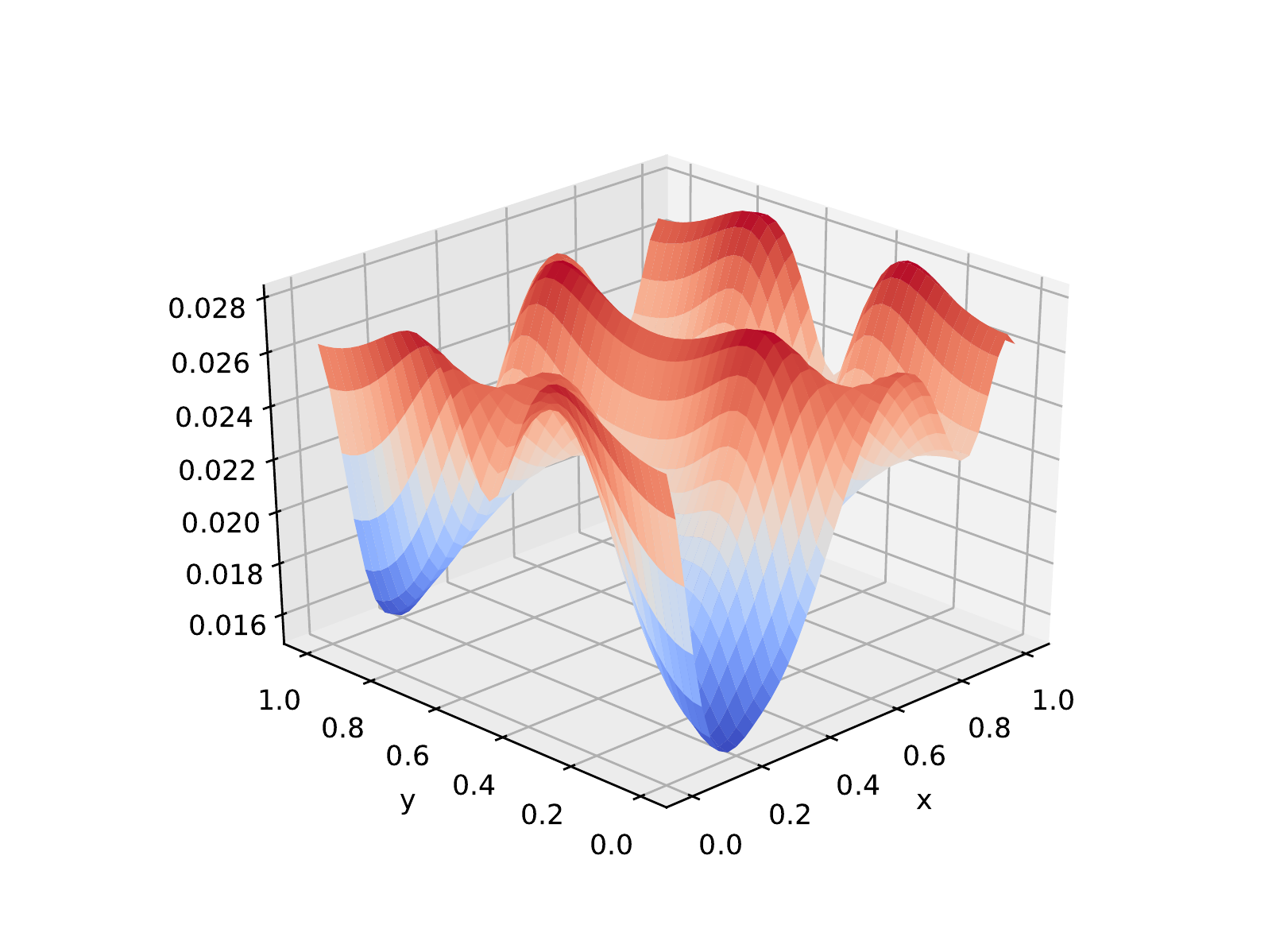}\par 
		\caption{{\small Absolute difference.}}
    \end{subfigure}
	\caption{Solution of the viscous G-equation at $T=2$ with $d=0.1$.}
	\label{fig:RecoveredSolutionT2D01}
\end{figure}

\begin{figure}[tbph] 
	\centering
	\begin{subfigure}{0.325\textwidth}
		\includegraphics[width=1.25\linewidth]{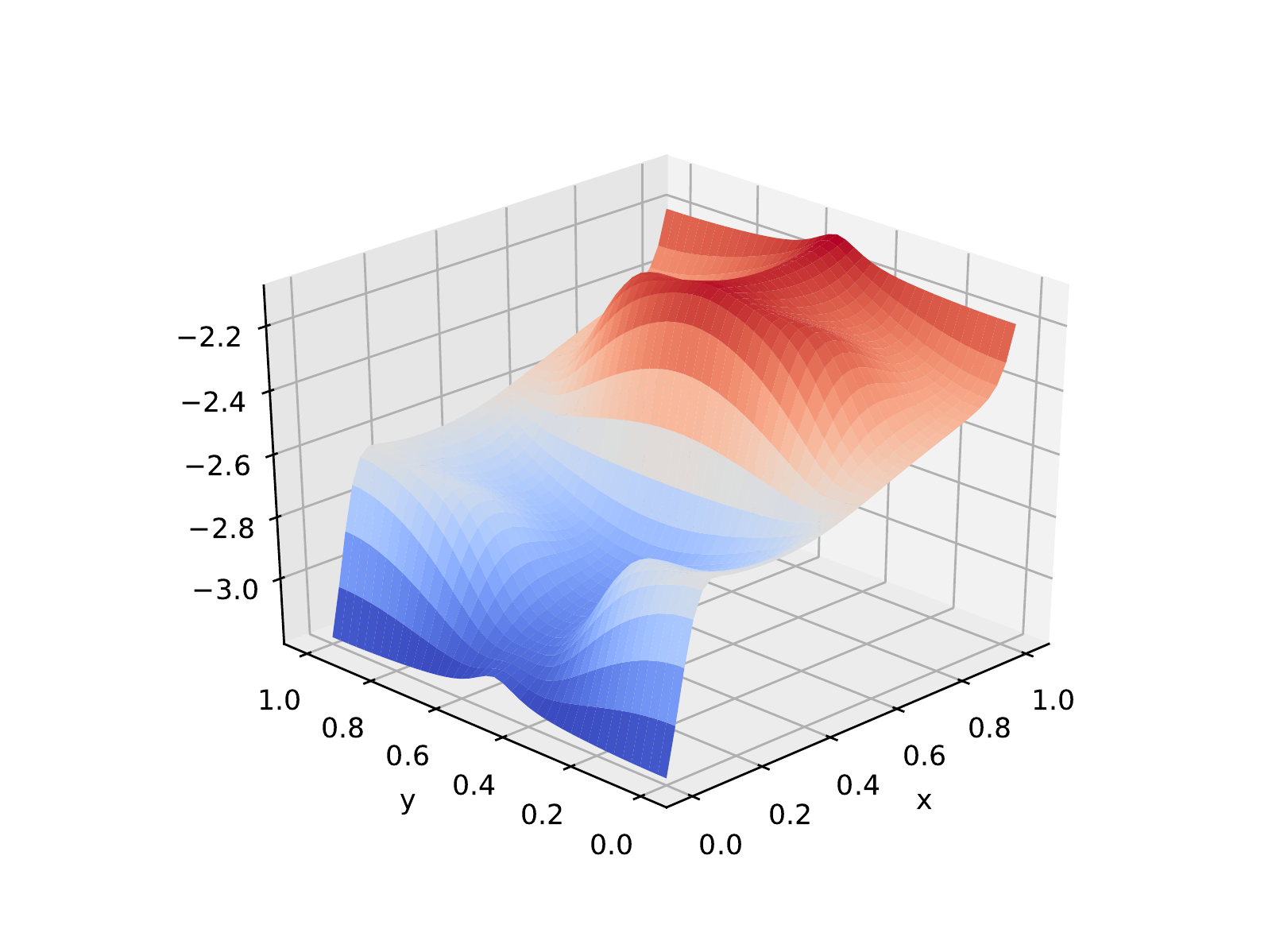}\par
		\caption{{\small Reference solution}}
	\end{subfigure}
	\begin{subfigure}{0.325\textwidth}
		\includegraphics[width=1.25\linewidth]{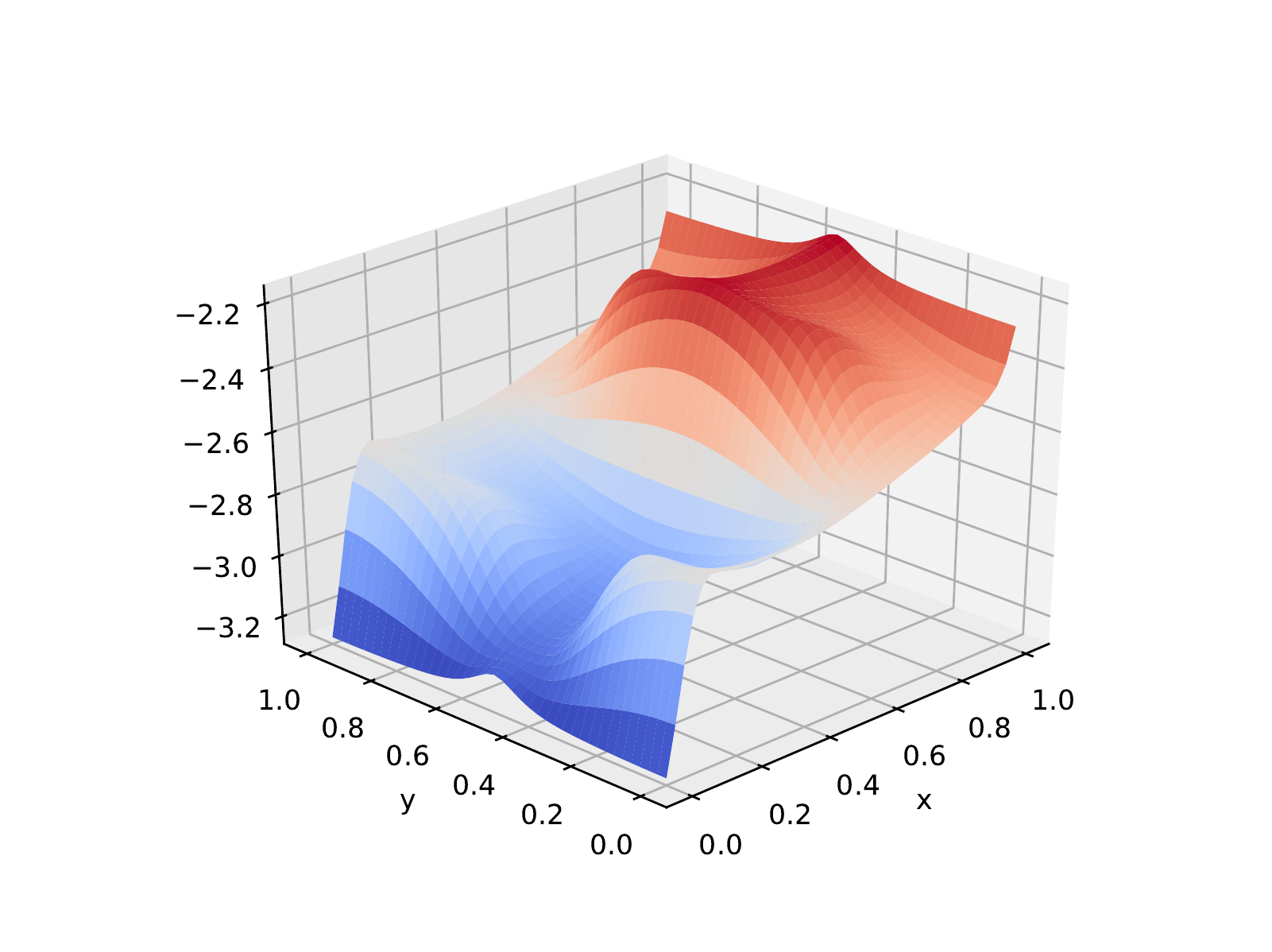}\par 
		\caption{\small POD solution.}
	\end{subfigure}
	\begin{subfigure}{0.325\textwidth}
		\includegraphics[width=1.25\linewidth]{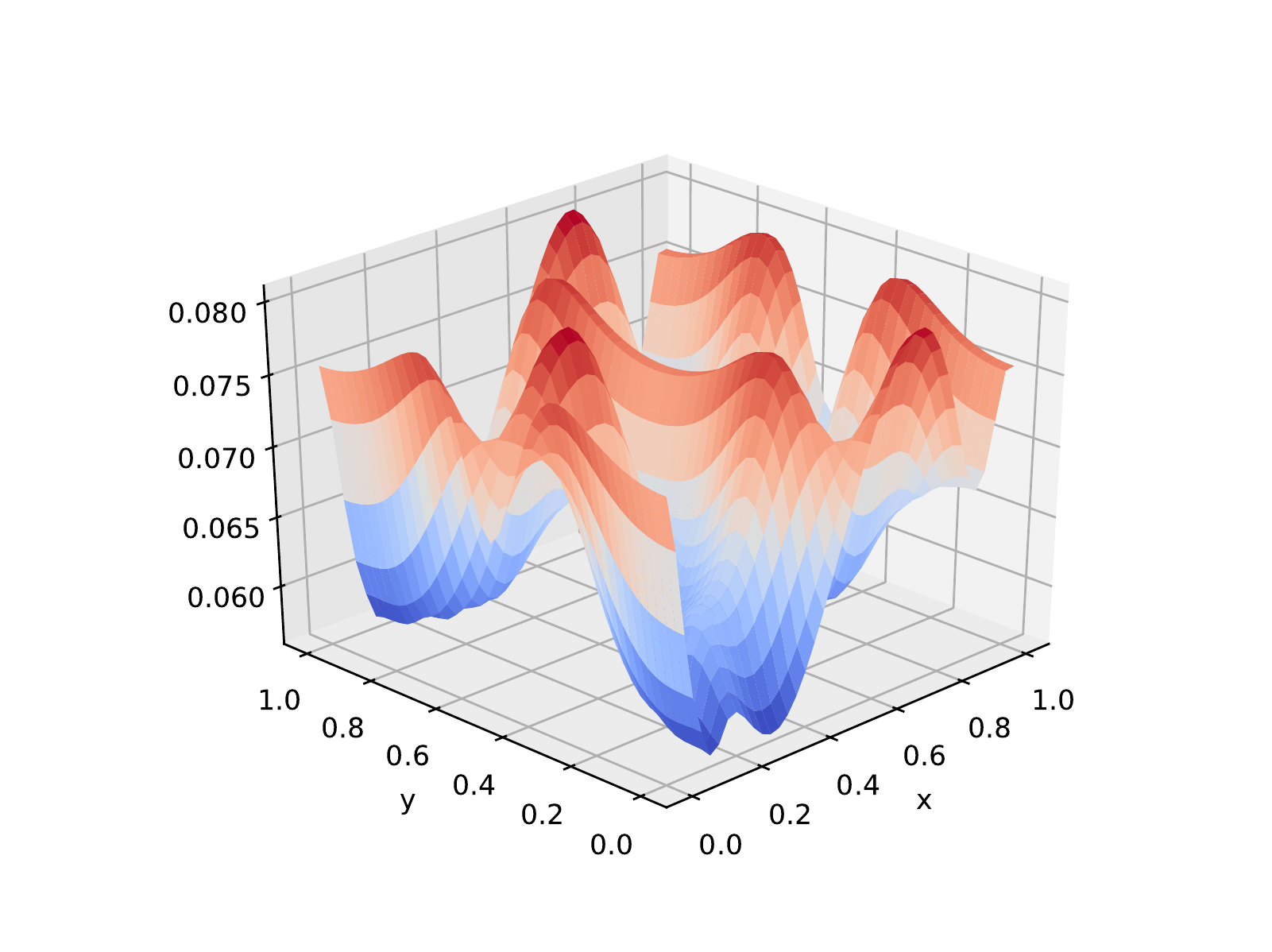}\par 
		\caption{{\small Absolute difference.}}
    \end{subfigure}
	\caption{Solution of the viscous G-equation at $T=2$ with $d=0.05$.}
	\label{fig:RecoveredSolutionT2D005}
\end{figure}

\begin{figure}[tbph] 
	\centering
	\begin{subfigure}{0.325\textwidth}
		\includegraphics[width=1.25\linewidth]{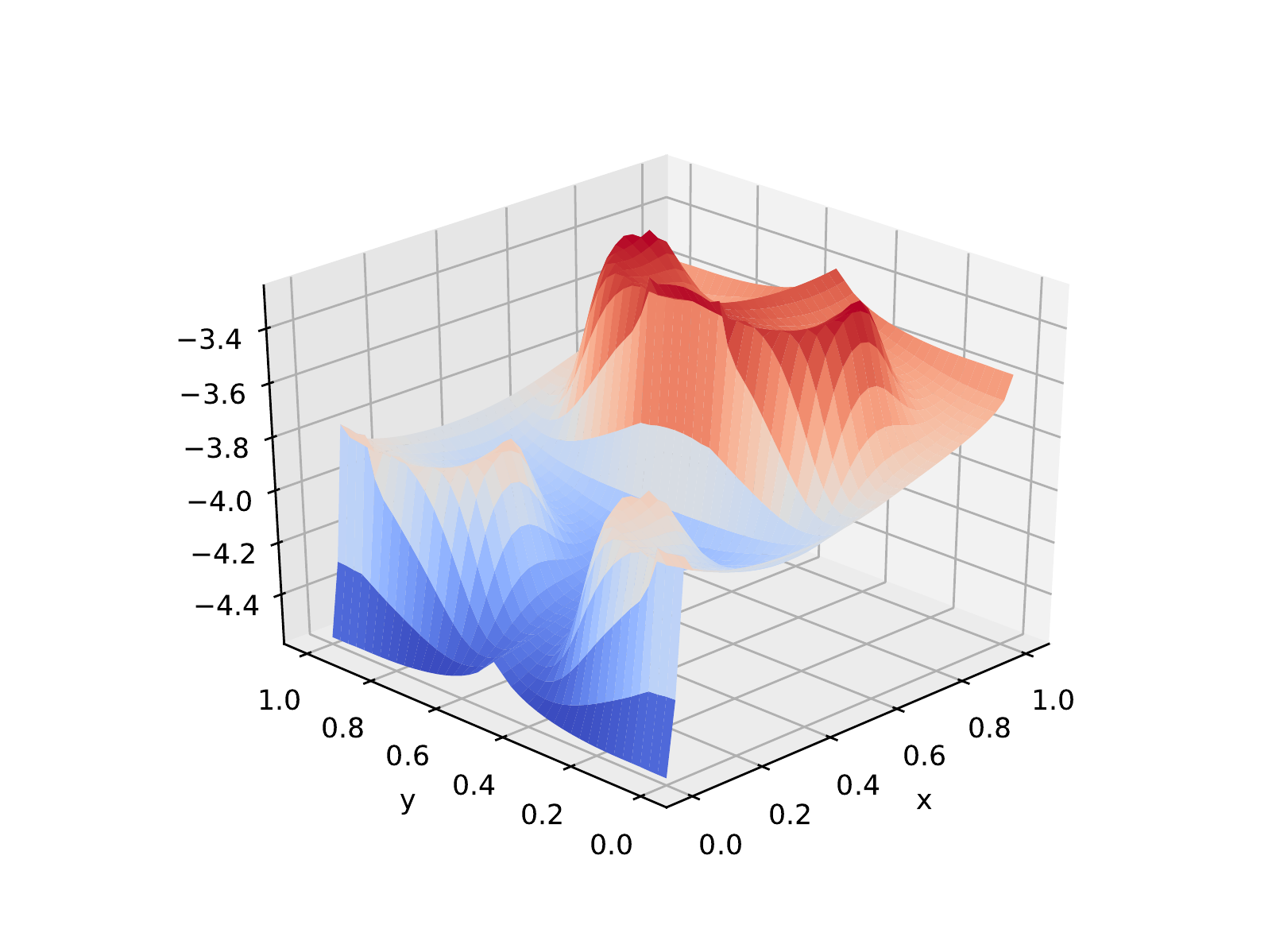}\par
		\caption{{\small Reference solution.}}
	\end{subfigure}
	\begin{subfigure}{0.325\textwidth}
		\includegraphics[width=1.25\linewidth]{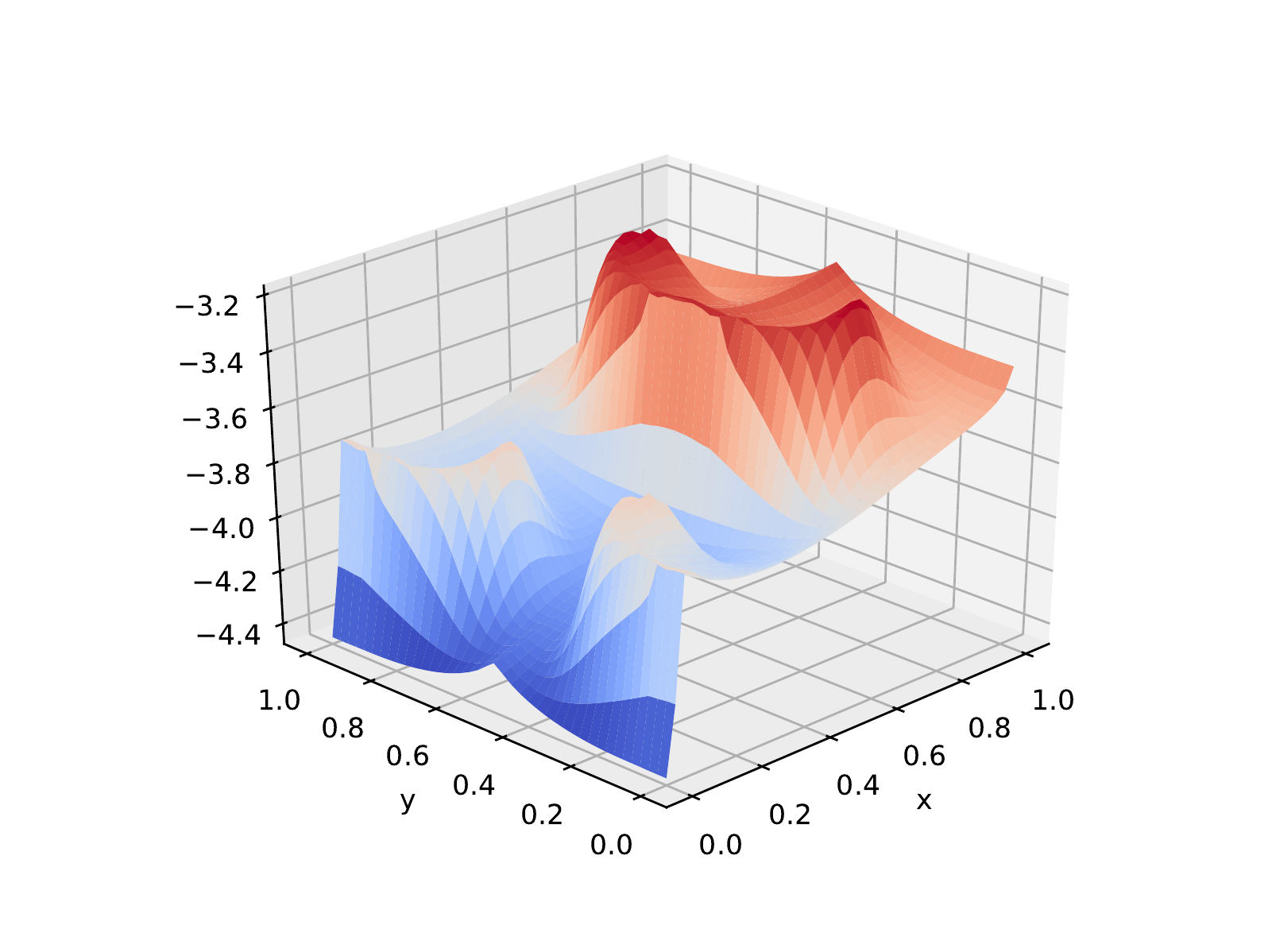}\par 
		\caption{\small POD solution.}
	\end{subfigure}
	\begin{subfigure}{0.325\textwidth}
		\includegraphics[width=1.25\linewidth]{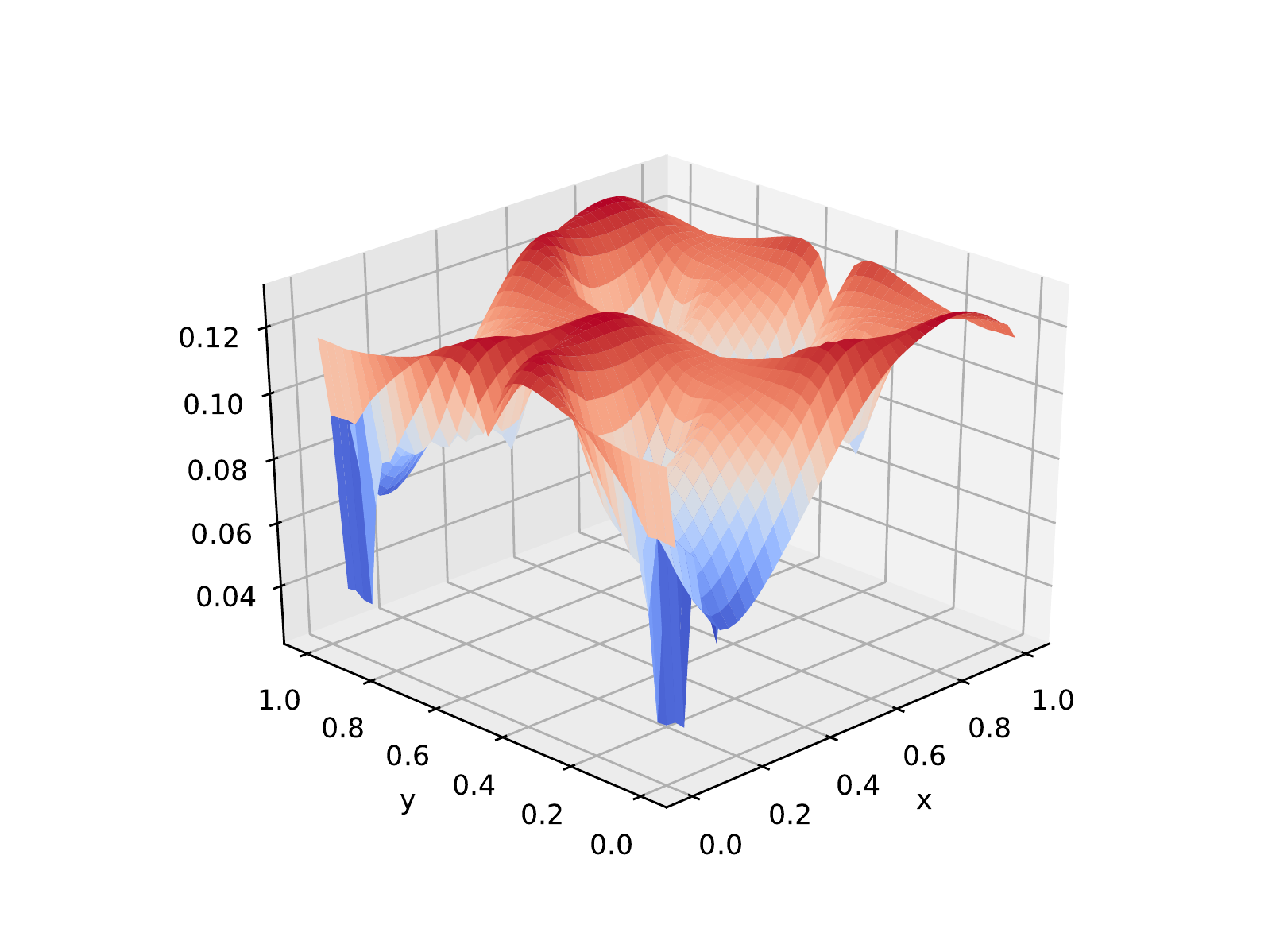}\par 
		\caption{{\small Absolute difference.}}
    \end{subfigure}
	\caption{Solution of the viscous G-equation at $T=2$ with $d=0.01$.}
	\label{fig:RecoveredSolutionT2D001}
\end{figure}

\begin{table}[h]
	\centering
	\begin{tabular}{|c|c|c|c|c|c|c|}
		\hline
		\textbf{ d} & \textbf{0.01} & \textbf{0.02} & \textbf{0.03} & \textbf{0.04} & \textbf{0.05} \\ \hline
		\textbf{Rela. error, Mean-free} & 0.082762  & 0.042995  & 0.029778  & 0.047371  & 0.042206 \\ \hline
		\textbf{Rela. error, Recovered} & 0.028087  & 0.020967  & 0.019781  & 0.033270  & 0.026356 \\ \hline
	\end{tabular}
	%\caption{Using the first 6 POD basis from 0 to 1s to compute the solution for 0 to 2s}
	%\centering
	\begin{tabular}{|c|c|c|c|c|c|c|}
		\hline
		\textbf{  d} & \textbf{0.06} & \textbf{0.07} & \textbf{0.08} & \textbf{0.09} & \textbf{0.1} \\ \hline
		\textbf{Rela. error, Mean-free} & 0.038563  & 0.03471  & 0.031951  & 0.029692  & 0.027939 \\ \hline
		\textbf{Rela. error, Recovered} & 0.020541  & 0.016758  & 0.013963  & 0.012019  & 0.010648 \\ \hline
	\end{tabular}
	\caption{The relative errors between POD solution and reference solution at $T=2$.}
	\label{Table:accuracyT2}
\end{table}

\subsection{Test of the POD basis for different parameters}\label{sec:Test3}
\noindent
In this subsection, we investigate the robustness of the POD basis across different diffusion parameter $d$'s. Specifically, we build the POD basis from the solution snapshots of the viscous G-equation with the diffusivity $d_0$. Then, we use that pre-computed POD basis to compute solutions of the viscous G-equation with other $d$'s.  

In our numerical experiment, we solve the viscous G-equation from $T=0$ to $T=1$ with $d_0=0.05$ to build the POD basis. 
In Table \ref{Table:different-d}, we show the relative errors between the POD solutions and the reference solutions to the viscous G-equation with different $d$. 
One can find that POD basis provides good approximations to the solution of the viscous G-equation when $d$ is chosen from $d=0.01$ to $d=0.1$. The closer $d$ is to $d_0$,  the smaller the error will be. 
The fact that the pre-computed POD basis can be used to compute solutions associated with different parameters with high accuracy is of great importance in practical computations. 

\begin{table}[h]
	\centering
	\begin{tabular}{|c|c|c|c|c|c|c|c|}
		\hline
		\textbf{ d} & \textbf{0.01} & \textbf{0.02} & \textbf{0.03} & \textbf{0.04} & \textbf{0.05} \\ \hline
		\textbf{Rela. error, Recovered} & 0.056233  & 0.030934  & 0.014142  & 0.008826  & 0.004398 \\ \hline
	\end{tabular}
	%\caption{Using the first 5 POD modes obtained from $d=0.05$ to compute the solutions for different $d$}
	%\label{Table:different-d} 
	\centering
	\begin{tabular}{|c|c|c|c|c|c|c|c|}
		\hline
		\textbf{  d} & \textbf{0.06} & \textbf{0.07} & \textbf{0.08} & \textbf{0.09} & \textbf{0.1} \\ \hline
		\textbf{Rela. error, Recovered} & 0.018682  & 0.022199  & 0.024353  & 0.025569  & 0.025963 \\ \hline
	\end{tabular}
	\caption{The relative errors between POD solutions and reference solutions.}
	\label{Table:different-d}
\end{table}

\subsection{Compute the turbulent flame speed $S_{T}$} \label{sec:Test4}
\noindent
Since the POD method is very efficient to approximately solve  the viscous G-equation, we apply it to compute the turbulent flame speed $S_{T}$, which is defined in the Eq.\eqref{eq:TurbulentFlameSpeed}. 
In this experiment, we compare the numerical results of the turbulent flame speeds obtained by two methods. First of all, for each parameter $d$, we compute the solution snapshots from $T=0$ to $T=1$ to \textcolor{black}{extract the first 6 POD basis functions.} 
\textcolor{black}{Then, with the POD basis, we solve the viscous G-equation for a much longer time, from $T=0$ to $T=8$, to obtain the approximated burned area $\mathcal{A}(t)$.} Finally, we show the growth rate of the $\mathcal{A}(t)$ over time $t$ in Fig.\ref{fig:TurbulentFlameSpeed}. We find that the results obtained by our method agree well with the reference method. 
\textcolor{black}{These results show that the POD basis can be used to track the long-time evolution of the turbulent flame speed. We refer the interested reader to the theoretical results on the long-time behavior of solutions to Hamilton-Jacobi equations or quasi-linear parabolic equations; see e.g. \cite{namah1997convergence,barles2001space,roquejoffre2001convergence}. Our results provide an interesting numerical illustration of these theoretical results.}

\begin{figure}[tbph] 
	\centering
	\begin{subfigure}[b]{0.325\textwidth}
		\includegraphics[width=1.10\linewidth]{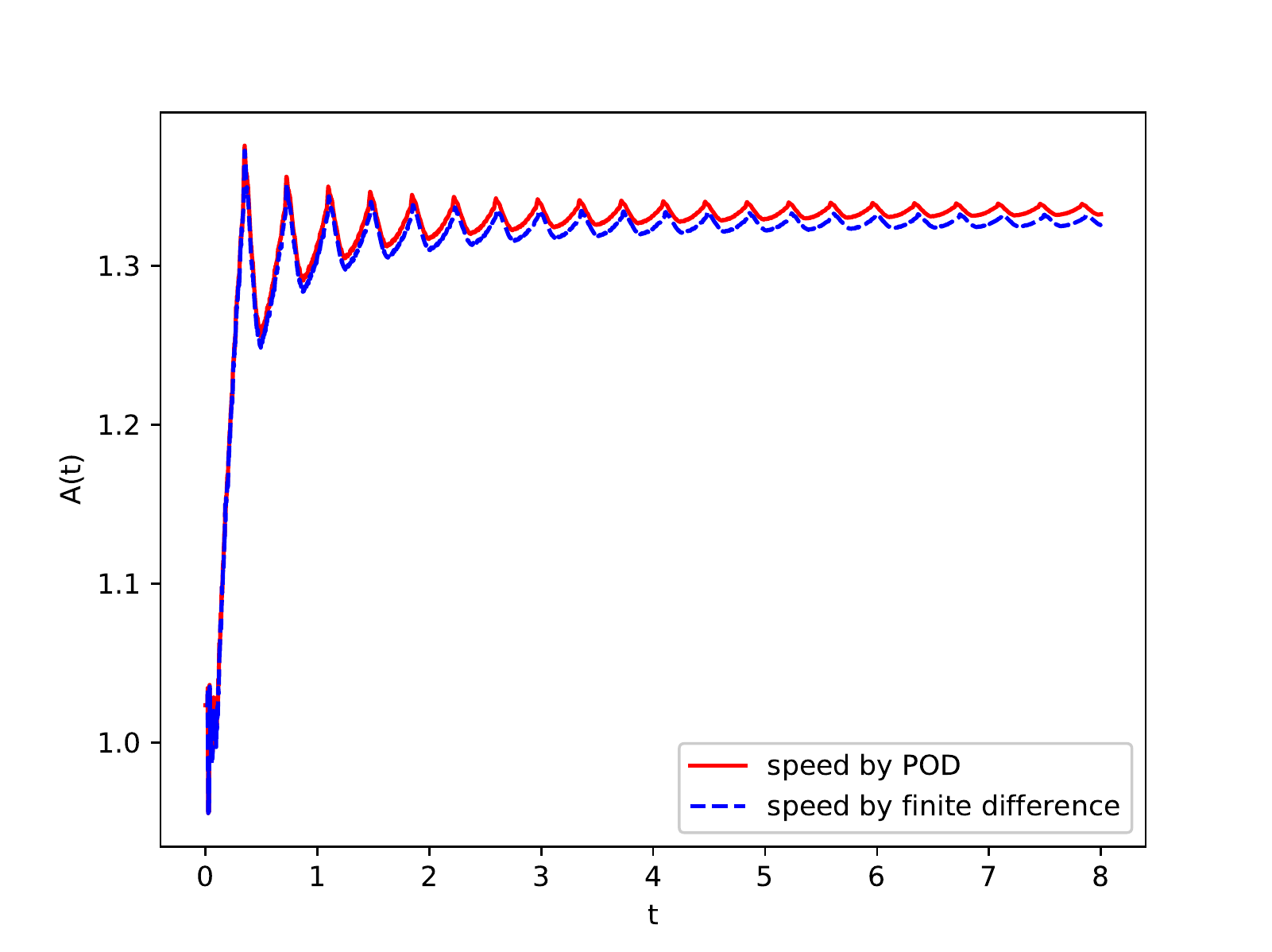} 
		\caption{{\small $d=0.1$.}}
	\end{subfigure}
	\begin{subfigure}[b]{0.325\textwidth}
		\includegraphics[width=1.10\linewidth]{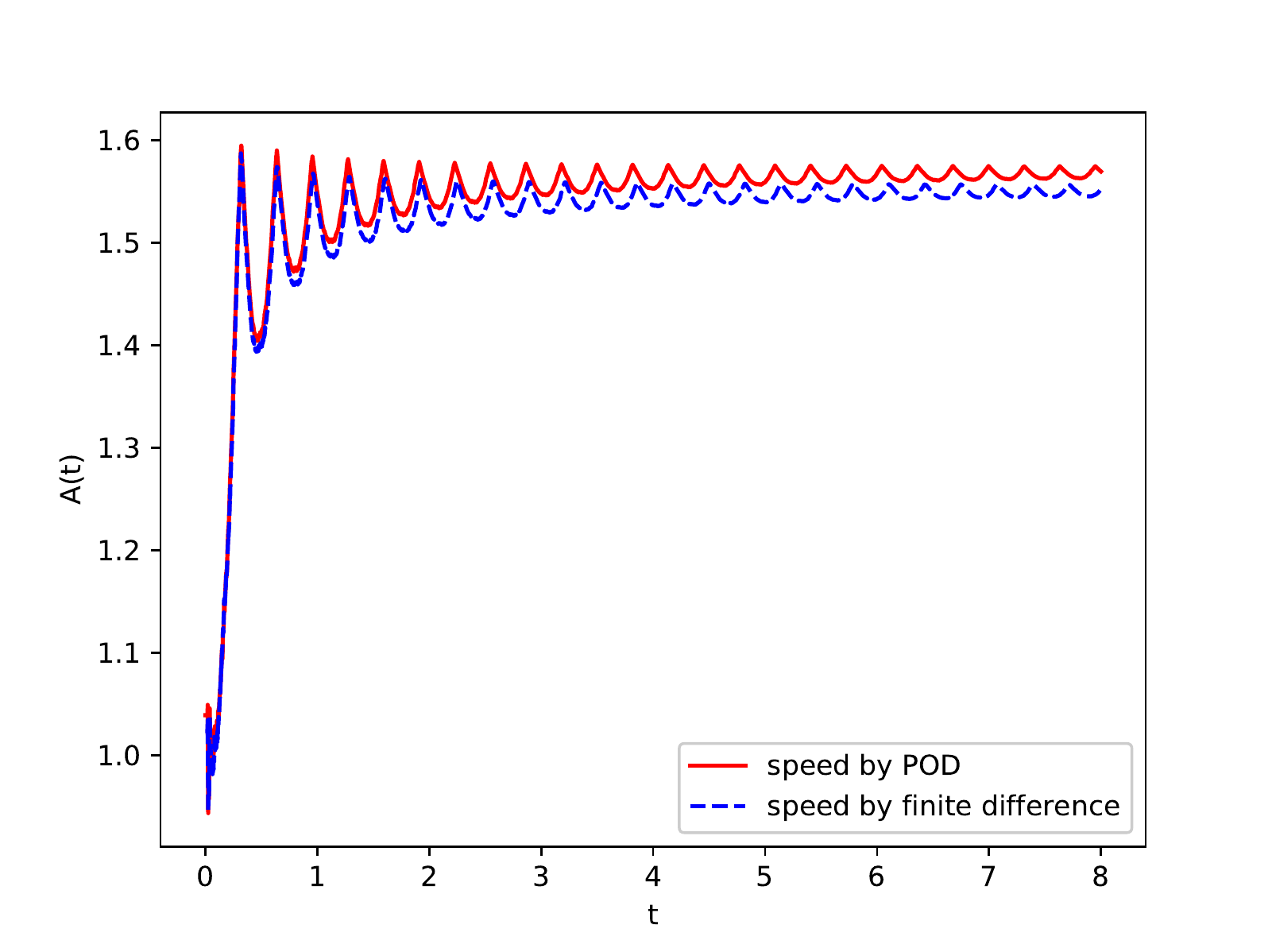} 
		\caption{\small $d=0.05$.}
	\end{subfigure}
	\begin{subfigure}[b]{0.325\textwidth}
		\includegraphics[width=1.10\linewidth]{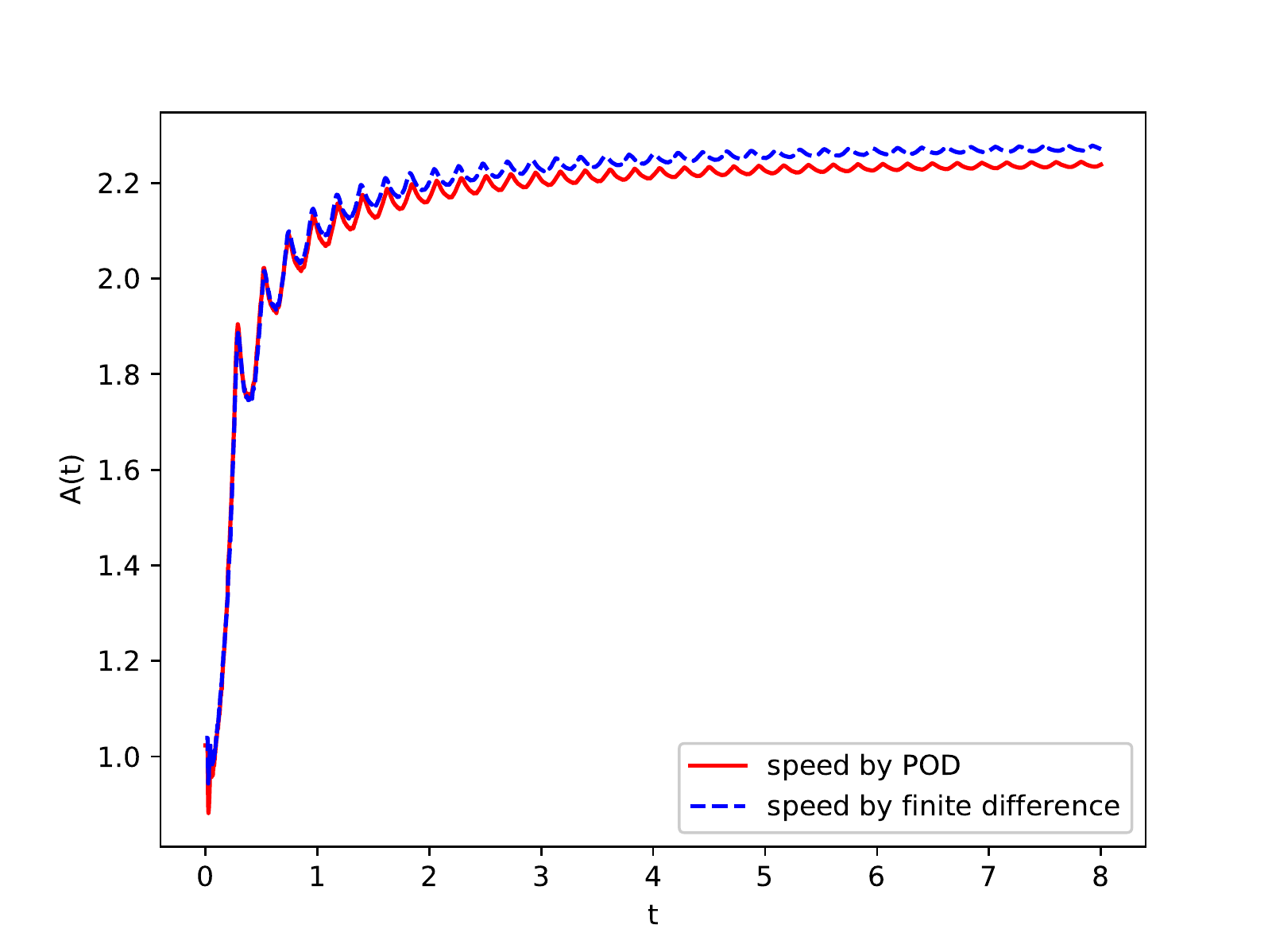} 
		\caption{\small $d=0.01$.}
	\end{subfigure}
	\caption{Numerical results of the turbulent flame speed obtained by different methods.}
	\label{fig:TurbulentFlameSpeed}
\end{figure}
 
\subsection{Comparison of the computational time}
\noindent
In Table~\ref{Table:runningtime} we compare the CPU time of two methods when solving the viscous G-equation from time $T=0$ to $T=1$ with different $d$. The unit of the computational time is second. One can see that computational cost of the POD method is far less than the finite difference method, which shows that model reduction method is very useful in many applications. Besides, one can find that the CPU time of the POD method slightly increases when $d$ decreases. This is due to the fact that we are using more POD basis for small $d$ in order to achieve the same POD truncation error threshold $e_{POD}=0.1\%$. 

\textcolor{black}{We observe that about $90\%$ of CPU time in the POD method is used to compute the nonlinear term (see $\textbf{f}_k$ in Eq.\eqref{eq:POD_scheme}), which may deteriorate the performance of the POD method for more challenging problems, such as 3D problems or convection-dominated problems. However, one can apply the discrete empirical interpolation method \cite{chaturantabut2010} to maintain the efficiency of the POD method.}

\begin{table}[h]
	\centering
	\begin{tabular}{|c|c|c|c|c|c|c|c|}
		\hline
		\textbf{d} & \textbf{0.01} & \textbf{0.02} & \textbf{0.03} & \textbf{0.04} & \textbf{0.05} \\ \hline
		\textbf{POD method} & 1.438894  & 1.296282  & 0.910010  & 0.945209  & 0.929022 \\ \hline
		\textbf{Reference method} & 480.6434  & 474.2021  & 483.5336  & 482.6826  & 480.4273 \\ \hline
	\end{tabular}
	\begin{tabular}{|c|c|c|c|c|c|c|c|}
		\hline
		\textbf{d} & \textbf{0.06} & \textbf{0.07} & \textbf{0.08} & \textbf{0.09} & \textbf{0.1} \\ \hline
		\textbf{POD method} & 0.873965  & 0.940874  & 0.906758  & 0.857222  & 0.786591 \\ \hline
		\textbf{Reference method} & 488.6499  & 476.8122  & 487.3498  & 483.3396  & 476.7582 \\ \hline
	\end{tabular}
	\caption{CPU time of two methods in solving the viscous G-equation from $T=0$ to $1$ with different $d$.}
		\label{Table:runningtime}
\end{table}

%\begin{table}[h]{\ContinuedFloat}
%	\centering
%	\begin{tabular}{|c|c|c|c|c|c|c|c|}
%		\hline
%		\textbf{d} & \textbf{0.06} & \textbf{0.07} & \textbf{0.08} & \textbf{0.09} & \textbf{0.1} \\ \hline
%		\textbf{No. of POD basis} & 3  & 3  & 3  & 2  & 2 \\ \hline
%	\end{tabular}
%	\caption{Number of POD basis we used to solve the equation with different $d$, determined by $e_{POD}=0.001$}
%\end{table}

\subsection{Test of the POD basis for a time-periodic fluid velocity}\label{sec:numPODtimeperiodicfluid}
\noindent
We consider a one-parameter family of time-dependent periodic cellular flow  
\begin{equation}\label{eq:TimePeriodicFlow}
\vec{V}(\vec{x},t)=\big(V_1,V_2\big)=A\big( \cos(2\pi y),\cos(2\pi x)\big)
+ A\theta\cos(2\pi t)\big((\sin(2\pi y), \sin(2\pi x) \big), 
\end{equation}
where $A$ is the amplitude of the velocity field. The first component of \eqref{eq:TimePeriodicFlow} is a steady cellular flow, while the second term of \eqref{eq:TimePeriodicFlow} is a time-periodic perturbation controlled by $\theta>0$ that introduces an increasing amount of disorder in the flow trajectories as $\theta$ increases.
\textcolor{black}{In the experiments reported in Fig.\ref{fig:TurbulentFlameSpeedTimeDependent}, we fix $A=4$, $\theta=1.0$ and test on three different $d$'s. For each $d$, the POD basis is extracted from the reference solution from time $T=0$ to $T=1$. Here, the numbers of POD basis functions are $r=11, 14, 27$ for $d=0.1,0.05,0.01$, respectively. Then using the POD basis, we compute the POD solution from time $T=0$ to $T=4$. The turbulent flame speeds $\mathcal{A}(t)/t$ obtained by the POD method are compared with the reference solutions.} 
%In the case $d=0.1$, we also plot the recovered solution at time $T=4$ in Fig.\ref{fig:RecoveredSolutionT4D001TD}.

The relative errors for the burning speed are basically negligible.  These results justify 
	that the POD method is still efficient in solving viscous G-equation with the time-periodic cellular flow.  
Compared with the results in Fig.\ref{fig:TurbulentFlameSpeed}, the turbulent flame speed in Fig.\ref{fig:TurbulentFlameSpeedTimeDependent} has different patterns, 
which is caused by the mixing and chaotic features of the fluid velocity \eqref{eq:TimePeriodicFlow}.  
\begin{figure}[h]
	\centering
	\begin{subfigure}[b]{0.325\textwidth}
		\includegraphics[width=1.12\linewidth]{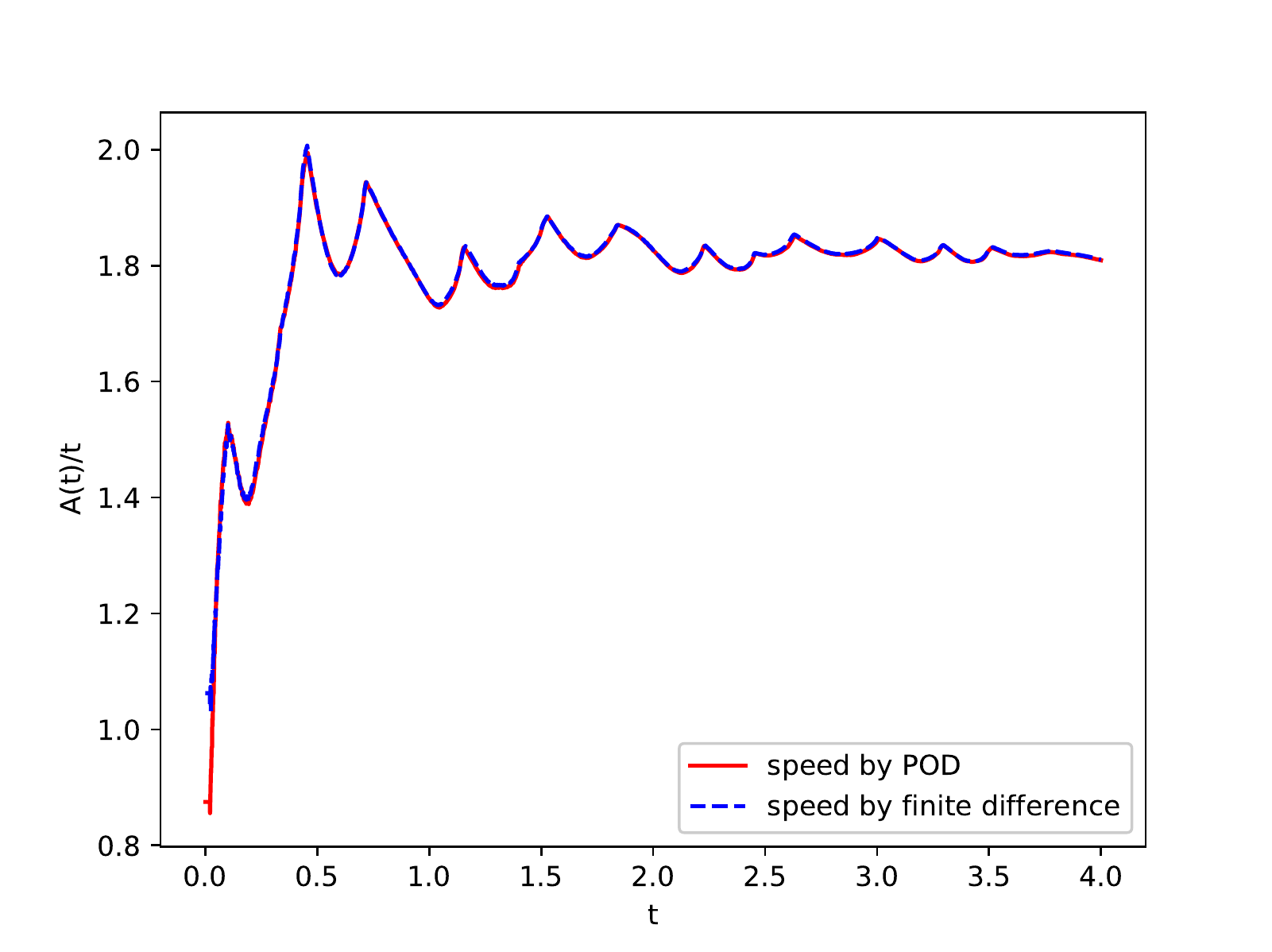} 
		\caption{{\small $d=0.1$.}}
	\end{subfigure}
	\begin{subfigure}[b]{0.325\textwidth}
		\includegraphics[width=1.12\linewidth]{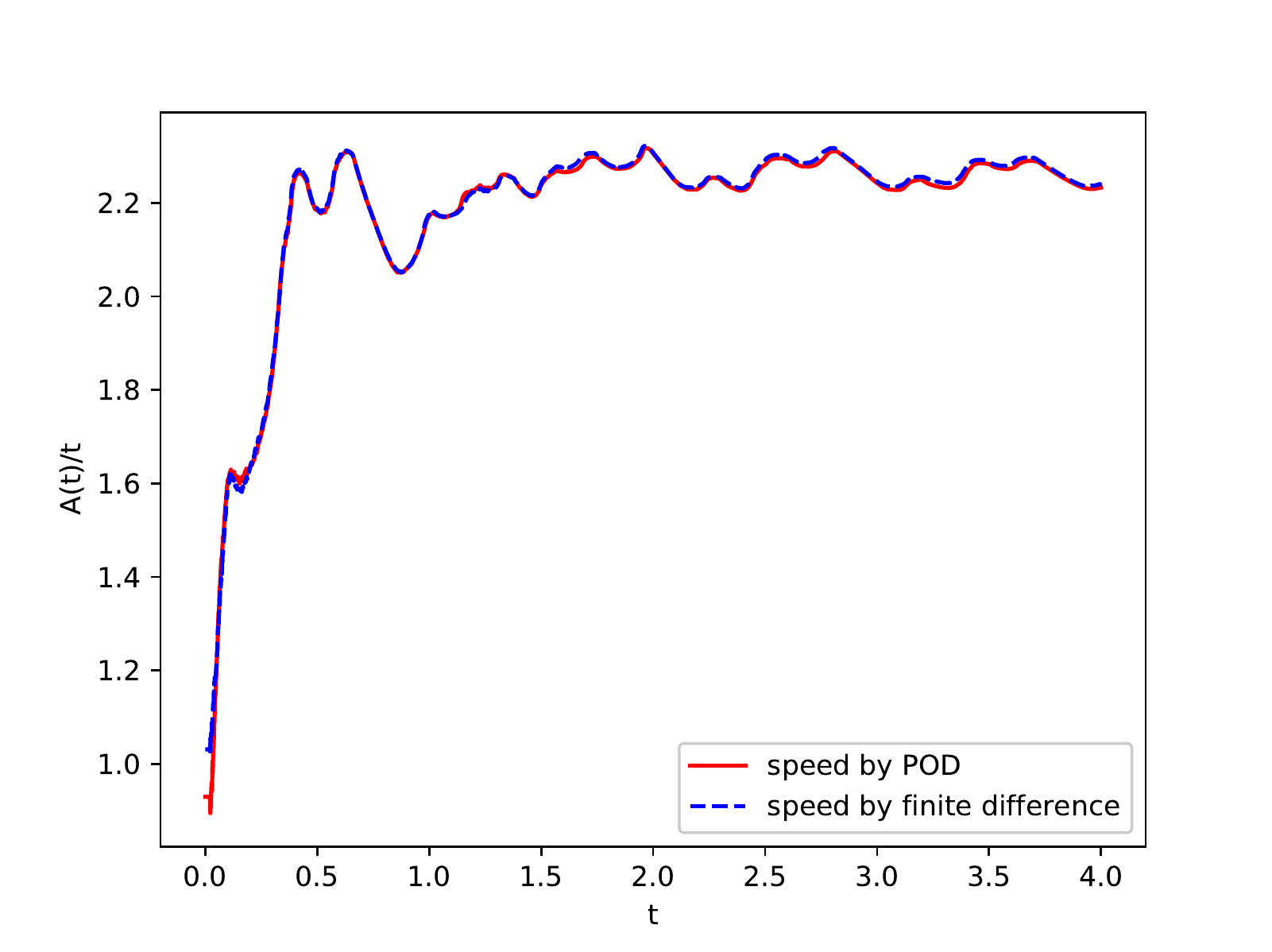} 
		\caption{\small $d=0.05$.}
	\end{subfigure}
	\begin{subfigure}[b]{0.325\textwidth}
		\includegraphics[width=1.12\linewidth]{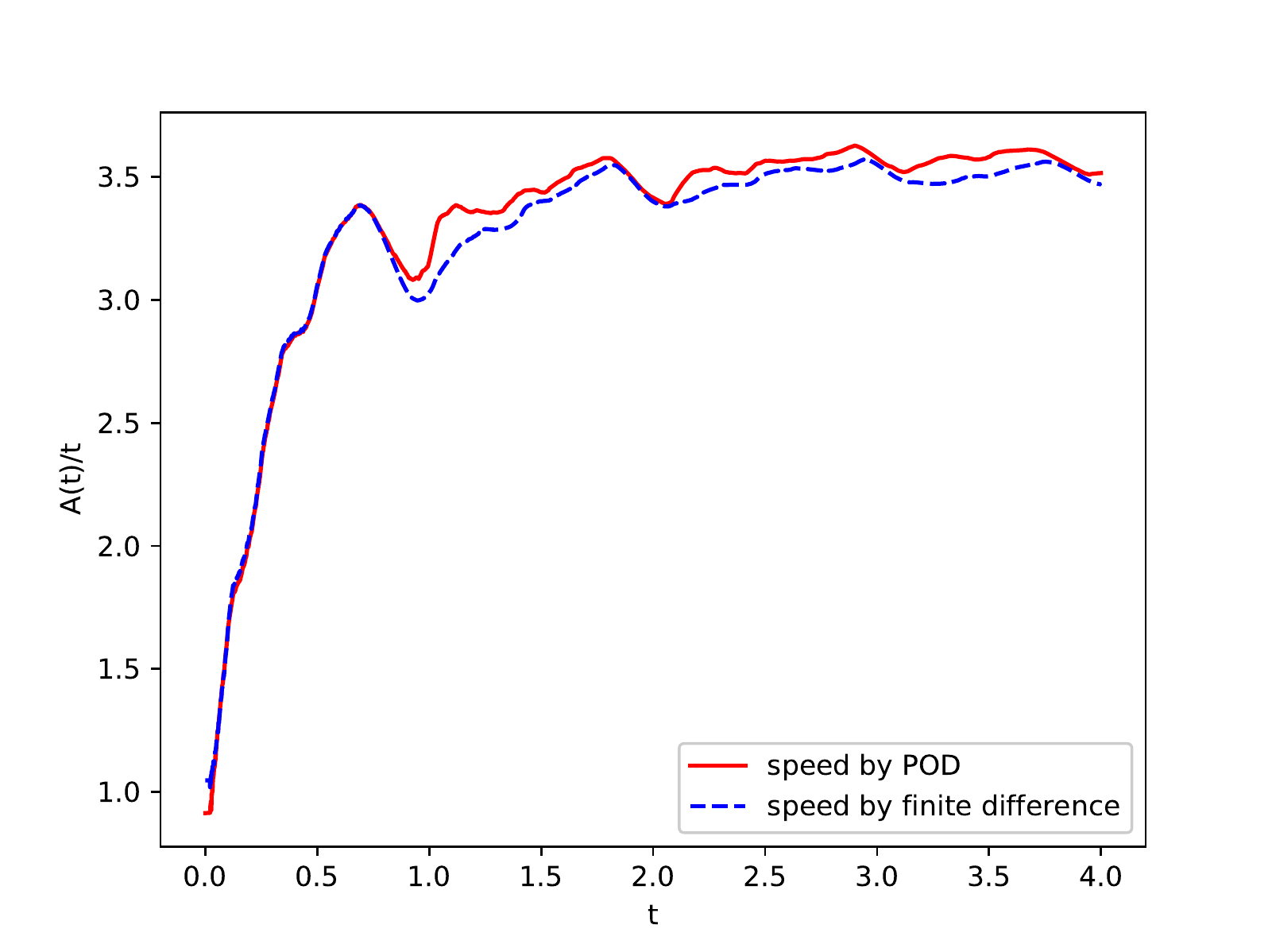} 
		\caption{\small $d=0.01$.}
	\end{subfigure}
	\caption{\textcolor{black}{Turbulent flame speeds obtained by different methods for the 
	time-dependent periodic cellular flow.}}
	\label{fig:TurbulentFlameSpeedTimeDependent}
\end{figure} 

%\begin{figure}[h] 
%	\centering
%	\begin{subfigure}{0.325\textwidth}
%		\includegraphics[width=1.25\linewidth]{POD/d01T4_FD_TD-eps-converted-to.pdf}\par
%		\caption{{\small Reference solution.}}
%	\end{subfigure}
%	\begin{subfigure}{0.325\textwidth}
%		\includegraphics[width=1.25\linewidth]{POD/d01T4_RPOD_TD-eps-converted-to.pdf}\par 
%		\caption{\small POD solution.}
%	\end{subfigure}
%	\begin{subfigure}{0.325\textwidth}
%		\includegraphics[width=1.25\linewidth]{POD/d01T4_Rdiff_TD-eps-converted-to.pdf}\par 
%		\caption{{\small absolute difference.}}
%    \end{subfigure}
%	\caption{\textcolor{black}{Solution of the viscous G-equation at $T=4$, $d=0.1$, time-dependent flow.}}
%	\label{fig:RecoveredSolutionT4D001TD}
%\end{figure}

\textcolor{black}{
We should point out that the strategy to choose the snapshots is a crucial point in the POD method. More generally, the POD basis obtained by minimizing the projection error over a set of snapshots chosen \textit{a priori}, e.g. \eqref{eq:PODerror} may not be able to capture the whole dynamics of equations with time-periodic or chaotic flows.  Therefore, we propose an adaptive strategy to enrich the POD basis functions, which was also used in \cite{cheng2013dynamically}. We provide the step-by-step algorithm of the adaptive strategy in the appendix \ref{sec:appedix_adaptive}. We consider the viscous G-equation with the time-dependent velocity field  \eqref{eq:TimePeriodicFlow} to demonstrate the main idea, where $A=4.0$, $\theta=1.0$, $d=0.1$.} 

\textcolor{black}{
We first compute the snapshots from time $T=0$ to $T=0.5$ and extract a set of POD basis functions from 
these snapshots. In this experiment, we obtain $r=7$ basis functions, which are called the initial POD basis functions. Notice that the time period of the velocity field \eqref{eq:TimePeriodicFlow} is one. Thus, 
the initial POD basis functions may not be able to capture the whole time-varying dynamics of the viscous G-equation.}

\textcolor{black}{
In our adaptive strategy, we will check the approximation ability of the current POD basis after every $\Delta T$ time period. At each checking time, we use the current POD solution as initial data and apply the finite difference scheme to compute a short period of time, say $N \Delta t$, to obtain $N$ new snapshots. Here, $\Delta t$ is the time step in the finite difference scheme and $N>0$ is an integer. We require $N \Delta t\ll \Delta T$. Then, we project the $N$ snapshots onto the current POD basis and compute the norm of the projection error. If the norm of the projection error is bigger than a prescribed threshold, we enrich the POD basis 
from the information in the projection error. }

\textcolor{black}{
In this experiment, we choose $\Delta T=0.5$, $\Delta t=0.001$ and $N=50$. We observe that every time, one or two basis functions will be added.  At the end of the experiment, there are $r=17$ basis functions. We compare the performances of the fixed-basis POD method with the adaptive POD method, in terms of estimating the flame speed.} \textcolor{black}{Fig.\ref{fig:TurbulentSpeedAdaptivePOD} clearly shows that the adaptive strategy indeed improves the performance of the POD method and obtains a more accurate estimation for the flame speed. More specifically, at the final time $T=4$, the relative error between the speed calculated from the adaptive POD method and that from the finite difference method is 1.0676\%, while the relative error for the fixed-basis POD method is 5.6648\%. Meanwhile, the adaptive method still achieves high computational efficiency, as the adaptive method is running the POD scheme for most time steps and only uses the finite difference scheme for approximately 10\% of the time steps. The CPU time for solving the equation from time $t=0$ to $t=4$ with the finite difference method is 2598.83 seconds, while for the adaptive POD method, it only takes 249.48 seconds}. 

\textcolor{black}{Notice that Fig.\ref{fig:TurbulentSpeedAdaptivePOD} still shows a gap between the finite difference solution and the adaptive POD solution. The reason is that we only check and enrich the POD basis every $\Delta T$ time period and we may lose a small amount of solution information between any two checking times. How to choose $\Delta T$ is important in the adaptivity of the POD method. To address this issue, one needs to obtain some \textit{a posteriori} error estimate for the solution obtained by the POD method. We will study this issue in our subsequent research.}

\begin{figure}[h]
	\centering
	\includegraphics[width=0.60\linewidth]{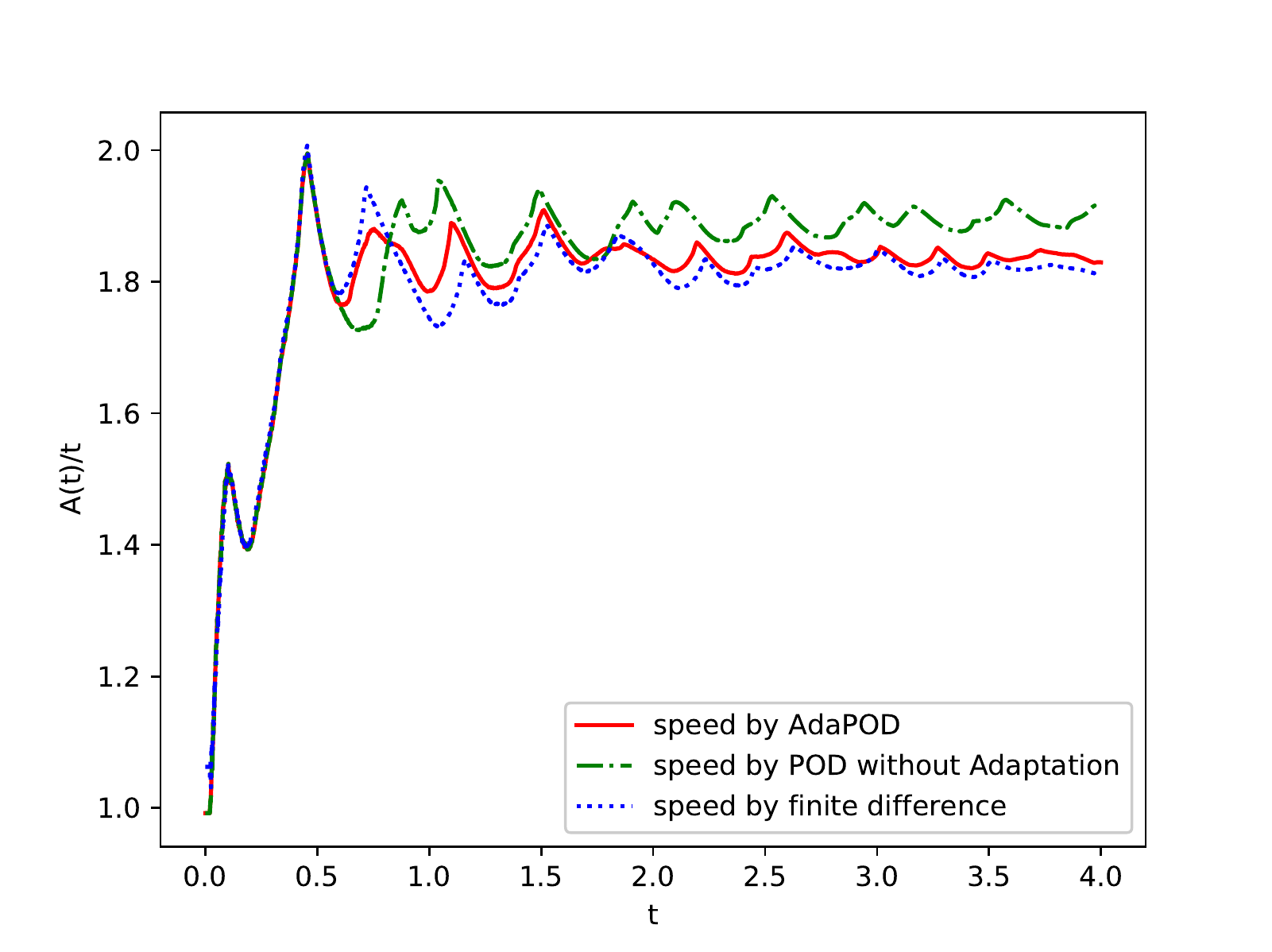}
	\caption{\textcolor{black}{A comparison of the POD method with a fixed basis strategy and 
			an adaptive strategy.}}
	%\caption{Illustration of local interface velocities in the G-equation and a flame front. Adaptive POD???}
	\label{fig:TurbulentSpeedAdaptivePOD}
\end{figure}
 
\subsection{Test of the POD method for curvature G-equations}\label{sec:Test5}
\noindent 
\textcolor{black}{To further investigate the performance of the POD method, we solve the curvature G-equations \eqref{eq:curvature-G} using the POD basis. Specifically, we consider the curvature G-equation with the periodic velocity field \eqref{eq:velo-field} and planar initial condition. If we write $G(\vec{x},t)=\vec{P}\cdot\vec{x}+u(\vec{x},t)$, then $u(\vec{x},t)$ is spatially periodic and 
satisfies the following periodic initial value problem }
\textcolor{black}{\begin{equation}\label{eq:curv-GonT}
\begin{cases}
u_t+\vec{V}\cdot(\vec{P}+\nabla u)+S_l|\vec{P}+\nabla u|=dS_l|\vec{P}+\nabla u|\bigg(\nabla\cdot\frac{\vec{P}+\nabla u}{|\vec{P}+\nabla u|}\bigg) &\text{in}\ \mathbb{T}^n\times(0,\infty),\\
u(\vec{x},0)=0 &\text{on}\ \mathbb{T}^n\times\{t=0\},
\end{cases}
\end{equation}}
\textcolor{black}{where $\mathbb{T}^n=[0,1]^n$. Again, we decompose $u$ into the mean-free part $\hat{u}$ and the mean $\bar{u}$. The derivation is almost same as what we have done in Section \ref{sec:decompotionstra}. The evolution equations for $\hat{u}$ and $\bar{u}$ are as follows}
\textcolor{black}{\begin{equation}\label{eq:curv-hatbar}
\begin{cases}
\hat{u}_t+\vec{V}\cdot\nabla\hat{u}-dS_l\big|\vec{P}+\nabla \hat{u}\big|\bigg(\nabla\cdot\frac{\vec{P}+\nabla \hat{u}}{|\vec{P}+\nabla \hat{u}|}\bigg)+\int_{\mathbb{T}^n}dS_l\big|\vec{P}+\nabla \hat{u}\big|\bigg(\nabla\cdot\frac{\vec{P}+\nabla \hat{u}}{|\vec{P}+\nabla \hat{u}|}\bigg)d\vec{x}+\\
S_l\big|\vec{P}+\nabla\hat{u}\big|-\int_{\mathbb{T}^n}S_l\big|\vec{P}+\nabla\hat{u}\big|d\vec{x}+\vec{V}\cdot\vec{P}=0 &\text{in}\ \mathbb{T}^n\times(0,\infty)\\
\bar{u}_t=-\int_{\mathbb{T}^n}S_l\big|\vec{P}+\nabla\hat{u}(\vec{x},t)\big|d\vec{x}+\int_{\mathbb{T}^n}dS_l|\vec{P}+\nabla \hat{u}|\bigg(\nabla\cdot\frac{\vec{P}+\nabla \hat{u}}{|\vec{P}+\nabla \hat{u}|}\bigg)d\vec{x} &\text{on}\ t\in(0,\infty)\\
\hat{u}(\vec{x},0)=0 &\text{on}\ \mathbb{T}^n\times\{t=0\}\\
\bar{u}(0)=0
\end{cases}
\end{equation}}
\textcolor{black}{The implementation of the POD method is the same as we have done for the viscous G-equation in Section \ref{sec:modelreductionviscGeq}. First of all, we use a reference numerical method to solve the curvature G-equation (see Section \ref{sec:appedix_ref2}) and construct the POD basis by using the mean-free part of the solution snapshots. Then, we use the POD-based Galerkin method to solve the evolution equation for $\hat{u}$ in Eq.\eqref{eq:curv-hatbar}.  The corresponding weak formulation is} 
\textcolor{black}{ 
\begin{equation}\label{eq:weak-curvG}
\langle\hat{u}_t,\psi\rangle+a(\hat{u},\psi)+\langle F(\hat{u}),\psi\rangle=\langle-\vec{P}\cdot\vec{V},\psi\rangle,\ \forall\psi\in H,  
\end{equation}}
\textcolor{black}{where the bilinear form $a(\hat{u},\psi)=\int_{\mathbb{T}^n}(\vec{V}\cdot\nabla \hat{u})\psi d\vec{x}$ and the nonlinear part $F(\hat{u})= S_l\big|\vec{P}+\nabla \hat{u}\big|-\int_{\mathbb{T}^n}S_l\big|\vec{P}+\nabla \hat{u}\big|d\vec{x}-dS_l|\vec{P}+\nabla \hat{u}|\big(\nabla\cdot\frac{\vec{P}+\nabla \hat{u}}{|\vec{P}+\nabla \hat{u}|}\big)+\int_{\mathbb{T}^n}dS_l|\vec{P}+\nabla \hat{u}|\big(\nabla\cdot\frac{\vec{P}+\nabla \hat{u}}{|\vec{P}+\nabla \hat{u}|}\big)d\vec{x}$. Finally, we solve a nonlinear equation system to compute 
the evolution equation for $\hat{u}$. The evolution equation for the mean $\bar{u}$ can be solved subsequently. }

%Hence, in this case, the POD scheme \textcolor{black}{for solving the mean-free part} is same as we have proposed in Section 3.3, with the bilinear form $a$ and the nonlinear part $F$ given by
%\begin{equation}\label{def:a_curv}
%a(u,v)=\int_{\mathbb{T}^n}(\vec{V}\cdot\nabla u)v d\vec{x}.
%\end{equation}
%\begin{align}
%F(u)=&S_l\big|\vec{P}+\nabla u\big|-\int_{\mathbb{T}^n}S_l\big|\vec{P}+\nabla u\big|d\vec{x}-\nonumber\\
%&dS_l|\vec{P}+\nabla u|\bigg(\nabla\cdot\frac{\vec{P}+\nabla u}{|\vec{P}+\nabla u|}\bigg)+\int_{\mathbb{T}^n}dS_l|\vec{P}+\nabla u|\bigg(\nabla\cdot\frac{\vec{P}+\nabla u}{|\vec{P}+\nabla u|}\bigg)d\vec{x}.\label{def:F_curv}
%\end{align}

\textcolor{black}{In our numerical experiment, we use the POD method to solve \eqref{eq:curv-hatbar} with the 
steady flow \eqref{eq:velo-field}. We choose $A=4.0$, $d=0.1$, $S_l=1$, and $\vec{P}=(1,0)$. 
We first construct $r=6$ POD basis functions from the reference solution snapshots from time $T=0$ to $T=1$. Then, we compute the mean-free POD solution on the same time interval using the POD-based Galerkin method. 
In the meanwhile, we compute the mean $\bar{u}$ solution. Finally, we can recover the solution $u$. The relative error of the recovered solutions between the POD method and reference method is 0.0220. 
Fig.\ref{fig:RecoveredSolutionT1D01Curv} shows the recovered solutions using the POD method and the reference method. We find that the POD method still performs well in solving the curvature G-equations.  }

\textcolor{black}{If we compare Fig.\ref{fig:RecoveredSolutionT1D01Curv} with Fig.\ref{fig:RecoveredSolutionT1D01}, we can see that with the same diffusion parameter $d=0.1$, the solution of the curvature G-equation is not very smooth. It is expected that a smaller $d$ in the curvature G-equation will bring more challenges in the 
method design and convergence analysis, which will be studied in our future research.  }

%Comparing Fig.\ref{fig:RecoveredSolutionT1D01Curv} with Fig.\ref{fig:RecoveredSolutionT1D01}, we see that with the same $d=0.01$, the curvature solution is indeed far less smooth than the viscous solution, which brings difficulties to the error analysis. But the numerical performance of the POD scheme is still satisfying. 

\begin{figure}[tbph] 
	\centering
	\begin{subfigure}{0.325\textwidth}
		\includegraphics[width=1.25\linewidth]{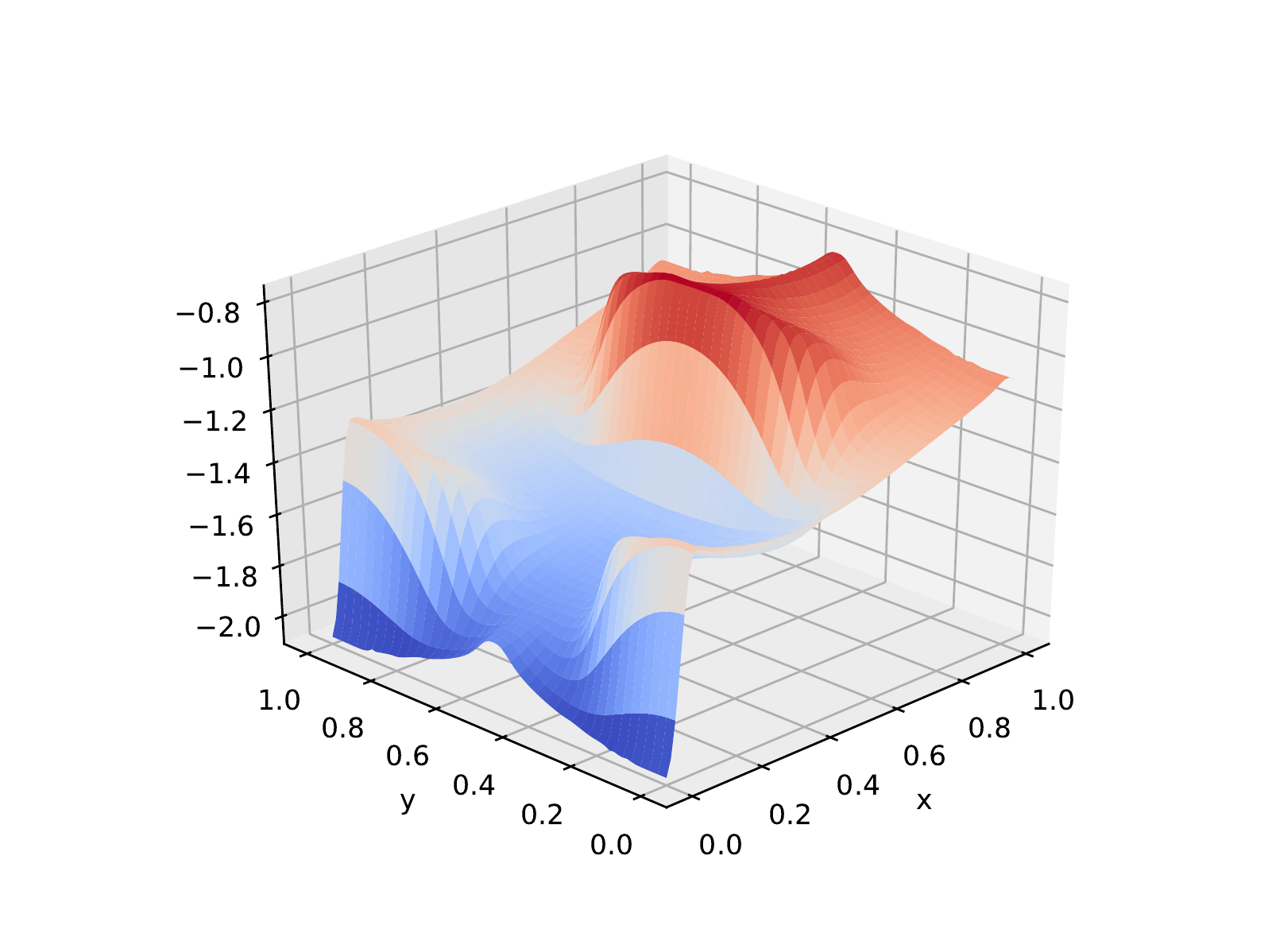}\par
		\caption{{\small Reference solution.}}
	\end{subfigure}
	\begin{subfigure}{0.325\textwidth}
		\includegraphics[width=1.25\linewidth]{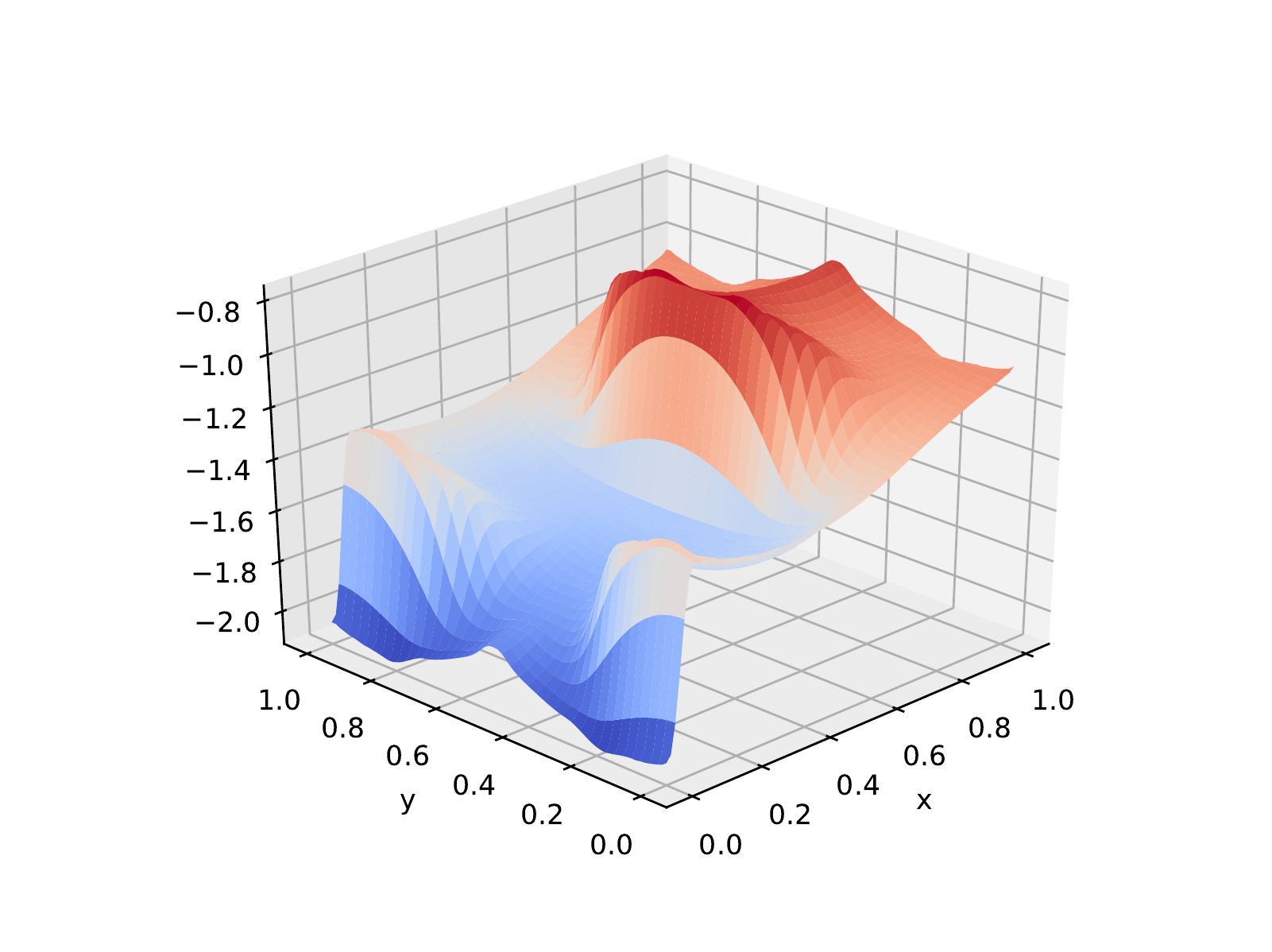}\par 
		\caption{\small POD solution.}
	\end{subfigure}
	\begin{subfigure}{0.325\textwidth}
		\includegraphics[width=1.25\linewidth]{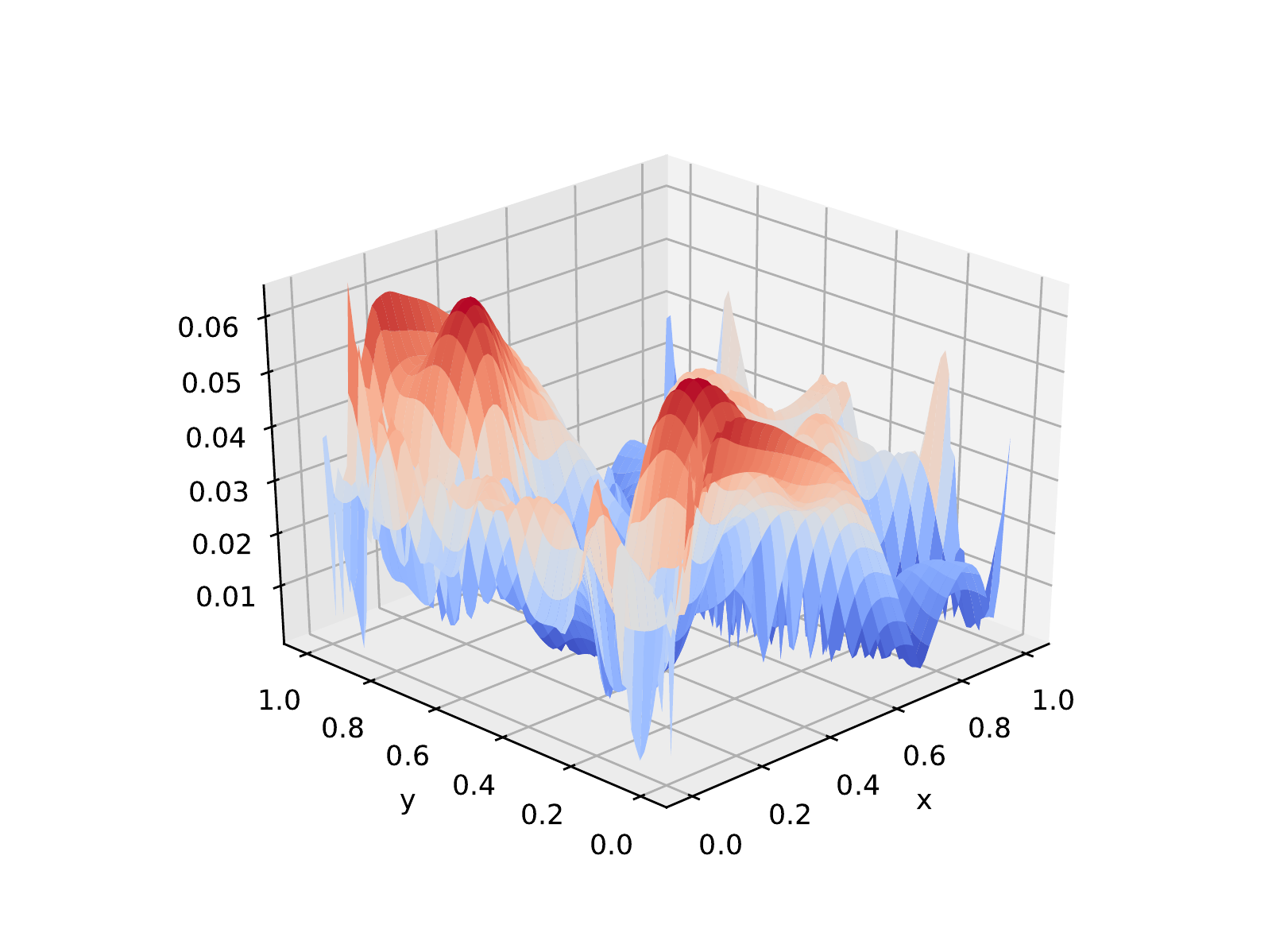}\par 
		\caption{{\small Absolute difference.}}
    \end{subfigure}
	\caption{Solution of the curvature G-equation at $T=1$ with $d=0.1$.}
	\label{fig:RecoveredSolutionT1D01Curv}
\end{figure}

%In this subsection, we make an further attempt to solve curvature G-equations using the POD method, based on our discussion in Section 2 and Section 3. Although the convergence analysis for the viscous case doesn't directly apply to the curvature case due to the degeneracy of the diffusion term, numerically the POD method can be easily adapted.

%%%%%%%%%%%%%%%%%%%%%%%%%%%%%%%%%%%%%%%%%
\section{Conclusion}\label{sec:Conclusion}
\noindent
We have proposed an efficient model reduction method to solve viscous G-equations, which have been very popular field models in combustion and physics literature for studying turbulent flame propagation. We constructed the POD basis based on learning the solution information from the snapshots. Then, we applied the Galerkin project method to solve \textcolor{black}{the viscous G-equation and smooth solutions of curvature G-equation} by using the POD basis. We provided rigorous error analysis for our numerical methods based on a decomposition strategy, where we decomposed the solution into a mean part and a mean-free part. We showed through numerical experiments that our methods can accurately compute the various G-equations with significant computational savings. In addition, we found that the POD basis allows us to compute long-time solution of the various G-equations. Thus, we can compute the corresponding turbulent flame speeds in cellular flows. In our future work, we plan to study turbulent flame speeds of G-equations in three dimensional spatially or spatiotemporally periodic vortical flows. \textcolor{black}{In addition, we will further investigate the adaptivity for the POD method and using the POD method to solve curvature G-equations.}

\section{Appendix}
\subsection{A reference method to solve viscous G-equations}\label{sec:appedix_ref}
\noindent
We first apply the finite difference scheme proposed in \cite{liu2013numerical} to solve the viscous G-equation from time 0 to $T$ seconds on the domain $D=[0,1]\times[0,1]$ to get the snapshots. Specifically, we employ a higher-order HJ-WENO scheme and TVD-RK scheme in spatial and time discretization, respectively; see \cite{osher1988fronts,gottlieb1998total,jiang2000weighted,fedkiw2002level} for more details of these schemes.

For a small $d$, the viscous G-equation is convection dominated and it should be treated like a hyperbolic equation. The forward Euler time discretization is given by
\begin{equation}\label{eq:forward_E}
\frac{G^{n+1}-G^n}{\Delta t}+H^n(G_x^-,G_x^+,G_y^-,G_y^+)-dS_l\Delta G^n=0,
\end{equation}
where $G_i^-$ and $G_i^+$ denote the left and right discretization of $G_i$ in the WENO5 scheme \cite{jiang2000weighted}. $H$ is a consistent and monotone numerical Hamiltonian. Then, we have 
\begin{equation}\label{eq:num-Ham}
H(G_x^-,G_x^+,G_y^-,G_y^+)=V_1G_x^{vel}+V_2G_y^{vel}\newline+S_l\sqrt[]{(G_x^{nor})^2+(G_y^{nor})^2},
\end{equation}
where the upwinding scheme and the Godunov scheme are applied for the convection term and the nonlinear term separately \cite{fedkiw2002level}.

\begin{equation}\label{eq:forward_Gx}
G_x^{vel}=\begin{cases}
G_x^- &\text{if}\ V_1>0,\\
G_x^+ &\text{if}\ V_1<0,
\end{cases}
\end{equation}
\begin{equation}\label{eq:forward_Gy}
G_y^{vel}=\begin{cases}
G_y^- &\text{if}\ V_2>0,\\
G_y^+ &\text{if}\ V_2<0,
\end{cases}
\end{equation}
\begin{equation}\label{eq:forward_Gx2}
(G_x^{nor})^2=\begin{cases}
(G_x^-)^2 &\text{if}\ V_1>S_l,\\
\max(\max(G_x^-,0)^2,\min(G_x^+,0)^2) &\text{if}\ |V_1|\leq S_l,\\
(G_x^+)^2 &\text{if}\ V_1<-S_l,
\end{cases}
\end{equation}
\begin{equation}\label{eq:forward_Gy2}
(G_y^{nor})^2=\begin{cases}
(G_y^-)^2 &\text{if}\ V_2>S_l,\\
\max(\max(G_y^-,0)^2,\min(G_y^+,0)^2) &\text{if}\ |V_2|\leq S_l,\\
(G_y^+)^2 &\text{if}\ V_2<-S_l.
\end{cases}
\end{equation}
For the diffusion term, we apply the central difference. For the time discretization, we apply the RK3 scheme \cite{gottlieb1998total}. The CFL condition in this case is 
\begin{equation}\label{eq:CFL-1}
\Delta t\bigg(\frac{S_l+|V_1|}{\Delta x}+\frac{S_l+|V_2|}{\Delta y}\bigg)<1.
\end{equation}

When $d$ is large, the time step size for the forward Euler scheme is very small, i.e., 
$\Delta t=O\big((\Delta x)^2+(\Delta y)^2)\big)$. To alleviate the stringent time step restriction, we introduce the following semi-implicit scheme:
\begin{equation}\label{eq:semi-implicit}
\frac{G^{n+1}-G^{n}}{dt}+\vec{V}\cdot\nabla G^{n+1}+S_l|\nabla G^n|=dS_l\Delta G^{n+1},
\end{equation}
where the convection and diffusion terms are discretized by the central difference, and the normal direction term is discretized by the Godunov and WENO5 scheme. In this case, the CFL condition is
\begin{equation}\label{eq:CFL-2}
\Delta t\bigg(\frac{S_l}{\Delta x}+\frac{S_l}{\Delta y}\bigg)<1.
\end{equation}
Numerical results in \cite{liu2013numerical} show that the proposed method is efficient in solving G-equations. The limitation is that the computational cost becomes large when one needs to choose a fine mesh to discretize the problem.

\subsection{A reference method to solve curvature G-equations}\label{sec:appedix_ref2}
\noindent 
\textcolor{black}{We also apply the finite difference scheme to solve the curvature G-equation~\eqref{eq:curvature-G}, which was proposed in \cite{liu2013numerical}. The spatial and time discretization are the same as for the viscous G-equations in Section \ref{sec:appedix_ref}. The curvature term is given by}
\textcolor{black}{\begin{equation}\label{eq:forward_curv}
|\nabla G|\bigg(\nabla\cdot\frac{\nabla G}{|\nabla G|}\bigg)=\frac{G_y^2G_{xx}-G_xG_yG_{xy}+G_x^2G_{yy}}{G_x^2+G_y^2},
\end{equation}}
\textcolor{black}{where each partial derivative is approximated by the central difference. The computation of the numerical Hamiltonian is same as before; see Eqns.\eqref{eq:num-Ham}-\eqref{eq:forward_Gy2}. For the time discretization, we apply the RK3 scheme. The time step restriction is}
\textcolor{black}{\begin{equation}\label{eq:CFL-curv}
\Delta t\bigg(\frac{S_l+|V_1|}{\Delta x}+\frac{S_l+|V_2|}{\Delta y}+\frac{2S_ld}{(\Delta x)^2}+\frac{2S_ld}{(\Delta y)^2}\bigg)<1.
\end{equation}}
\textcolor{black}{When $d$ is large, in order to avoid a small time step size $\Delta t=O((\Delta x)^2)$, the curvature term is decomposed as}
\textcolor{black}{\begin{equation}\label{eq:backward_curv}
|\nabla G|\bigg(\nabla\cdot\frac{\nabla G}{|\nabla G|}\bigg)=\Delta G-\Delta_\infty G=G_{xx}+G_{yy}-\frac{G_x^2G_{xx}+G_xG_yG_{xy}+G_y^2G_{yy}}{G_x^2+G_y^2}.
\end{equation}}
\textcolor{black}{If we apply the backward Euler scheme on $\Delta G$ and forward Euler scheme on $\Delta_\infty G$, then we have the following semi-implicit time discretization scheme}
\textcolor{black}{\begin{equation}\label{eq:semi-implicit-curv}
\frac{G^{n+1}-G^{n}}{dt}+\vec{V}\cdot\nabla G^{n+1}+S_l|\nabla G^n|=dS_l(\Delta G^{n+1} - \Delta_\infty G^n).
\end{equation}}
\textcolor{black}{The convection and diffusion terms are discretized by the central difference, and the normal direction term is discretized by the Godunov and WENO5 scheme. In this case, the CFL condition is same as \eqref{eq:CFL-1}.}

\subsection{Model reduction using the POD method}\label{sec:appedix_pod}
\noindent 
Let $X$ be a Hilbert space equipped with the inner product $\langle\cdot,\cdot\rangle_X$ and $u(\cdot, t)\in X$, $t\in[0, T]$ be the solution of a dynamic system. In practice, we approximate the space $X$ by a linear finite dimensional space $V$ with $dim(V)=N_{dof}$, where $N_{dof}$ represents the degree of freedom of the solution space and $N_{dof}$ can be extremely large for high-dimensional problem (consider the finite element method or finite difference method as examples). Given a set of snapshot of solutions, denoted as $\{u(\cdot,t_1),u(\cdot,t_2),\cdots,u(\cdot,t_m)\}$, where $t_1,\cdots,t_m\in[0, T]$ are different time instances and $m$ is the number of the solution snapshots.

The POD method aims to build a set of orthonormal basis $\{\psi_1(\cdot),\psi_2(\cdot),\cdots,\psi_r(\cdot)\}$
with $r\leq\min(m,N_{dof})$ that optimally approximates the solution snapshots.  More specifically, 
the POD method constructs the basis by solving the following optimization problem
\begin{equation}\label{eq:POD}
\min_{\psi_1,\cdots,\psi_r\in X}\frac{1}{m}\sum_{i=1}^{m}\Big\lVert u(\cdot,t_i)-\sum_{j=1}^{r}\langle u(\cdot,t_i),\psi_j(\cdot)\rangle_X\psi_j(\cdot)\Big\rVert^2_X, \quad s.t.\ \langle\psi_i,\psi_j\rangle_X=\delta_{ij}.
\end{equation}
In order to solve Eq.\eqref{eq:POD}, we consider the eigenvalue problem
\begin{equation*}
Kv=\lambda v,
\end{equation*}
where $K\in\mathbb{R}^{m\times m}$ and $K_{ij}=\frac{1}{m}\langle u(\cdot,t_i),u(\cdot,t_j)\rangle_X$ is the snapshot correlation matrix. Let $v_k,k=1,\cdots,n$ be the eigenvectors and $\lambda_1\geq\lambda_2\geq\cdots\geq\lambda_{m}>0$ be the positive eigenvalues. It has been shown in \cite{sirovich1987,volkwein2013proper} that the solution of the optimization problem \eqref{eq:POD} is given by
\begin{equation}
\psi_k(\cdot)=\frac{1}{\sqrt{\lambda_k}}\sum_{j=1}^{m}(v_k)_ju(\cdot,t_j),\ 1\leq k\leq r.
\end{equation}
It can also be shown that the following error formula holds
\begin{equation}\label{eq:POD_residue}
\frac{1}{m}\sum_{i=1}^{m}\lVert u(\cdot,t_i)-\sum_{j=1}^{r}\langle u(\cdot,t_i),\psi_j(\cdot)\rangle_X\psi_j(\cdot)\rVert^2_X=\sum_{l=r+1}^{m}\lambda_l.
\end{equation}
Finally, let $S^r$ denote the $r$-dimensional space spanned by  $\{\psi_1(\cdot),\psi_2(\cdot),\cdots,\psi_r(\cdot)\}$.

\subsection{An adaptive strategy to dynamically enrich the POD basis}\label{sec:appedix_adaptive}
\noindent  \textcolor{black}{In this subsection, we formalize the proposed adaptive strategy for enriching the POD basis 
in Algorithm \ref{alg:Ada_POD}, where the notations have been defined before}.  
\begin{algorithm}[h]
	\caption{\textbf{Adaptive strategy for enriching the POD basis}}
	\label{alg:Ada_POD}
	\begin{algorithmic}[1]
		\STATE \textbf{Input}: POD basis $\mathcal{S}=\{\varphi_i\}$, error threshold $\epsilon>0$, computational time $T$, period of checking time $\Delta T$, time step of POD or finite difference method $\Delta t$, $N_T=\lceil T/\Delta t\rceil$, $N_\text{check}=\lceil \Delta T/\Delta t\rceil$, and snapshots number $N$ (where $N \Delta t\ll \Delta T$). Parameters in the G-equations. Solution at time $t=0$: $U_0$.
		\FOR {$i=1:N_T$}
		\IF {$i>0$ and ($i$ mod $N_\text{check})=0$}
		\STATE {Set $\mathcal{U}=\emptyset$}.
		\FOR {$j=1:N$}
		\STATE Apply one step of the finite-difference scheme proposed by \cite{liu2013numerical} (see Appendix \ref{sec:appedix_ref} and \ref{sec:appedix_ref2}) to solve the equation from time $(i+j-1)\Delta t   $ to time $ (i+j)\Delta t$.
		\STATE Denote the finite-difference solution at time $\Delta t  (i+j)$ by $u_j$.
		\STATE $\mathcal{U}=\mathcal{U}\cup\{u_j\}$.
		\ENDFOR
		\STATE Let $P_\mathcal{S}$ be the projection onto the space spanned by $\mathcal{S}$. Compute $P_{\mathcal{S}}\mathcal{U}=\{P_{\mathcal{S}}u_j\}$.
		\STATE Compute the POD basis $\mathcal{S}'$ for $P_{\mathcal{S}}\mathcal{U}$, such that the approximation error (\ref{eq:POD_residue}) is less than $\epsilon$ (see Appendix \ref{sec:appedix_pod}).
		\STATE $\mathcal{S} = \mathcal{S} \cup \mathcal{S}'$.
		\ENDIF
		\STATE Apply one step of the backward POD-based scheme (\ref{eq:POD_scheme}), with basis $\mathcal{S}$, to solve the equation from time $(i-1)\Delta t $ to time $i\Delta t$. \STATE Denote the solution at time $i\Delta t $ by $U_i$.
		\ENDFOR
		\STATE \textbf{Output}: $\{U_t\}$
	\end{algorithmic}
\end{algorithm}

\section*{Acknowledgement}
\noindent
The research of J. Xin is partially supported by NSF grants DMS-1211179, DMS-1522383 and IIS-1632935. The research of Z. Zhang is supported by Hong Kong RGC grants (Projects 27300616, 17300817, and 17300318), National Natural Science Foundation of China (Project 11601457), Basic Research Programme (JCYJ20180307151603959) of The Science, Technology and Innovation Commission of Shenzhen Municipality, Seed Funding Programme for Basic Research (HKU), and an RAE Improvement Fund from the Faculty of Science (HKU).

%%%%%%%%%%%%%%%%%%%%%%%%%%%%%%%%%%%%%%%%%
%\section*{References}
%\bibliographystyle{plain}
%\bibliographystyle{amsalpha}
%\bibliographystyle{ieeetr}
\bibliographystyle{siam}
\bibliography{reference}
 
\end{document}